\documentclass[10pt]{amsart}
\usepackage[utf8]{inputenc}
\usepackage{amsfonts, amsthm, amsmath, amssymb,bm}
\usepackage{amscd}
\usepackage{mathtools}
\usepackage[numbers]{natbib}
\usepackage{esint}
\usepackage{bbm}
\usepackage[hidelinks,hyperfootnotes=false]{hyperref}

\usepackage{enumerate}
\usepackage{cases}

\newcommand{\vol}{\operatorname{vol}}
\newcommand{\loc}{\operatorname{loc}}
\newcommand{\Eis}{\mathrm{Eis}}
\newcommand{\Res}{\operatorname{Res}}

\newcommand{\gen}{\mathrm{gen}}
\newcommand{\Id}{\operatorname{Id}}

\newcommand{\bs}{\backslash}

\newcommand{\ol}{\overline}

\renewcommand{\hat}{\widehat}
\renewcommand{\tilde}{\widetilde}
\newcommand{\defeq}{\vcentcolon=}
\newcommand{\GL}{\mathrm{GL}}
\newcommand{\PGL}{\mathrm{PGL}}
\newcommand{\G}{\mathrm{G}}

\newcommand{\C}{\mathbb{C}}
\newcommand{\A}{\mathbb{A}}
\newcommand{\R}{\mathbb{R}}
\renewcommand{\H}{\mathcal{H}}
\newcommand{\Z}{\mathbb{Z}}
\newcommand{\Ad}{\mathrm{Ad}}

\newcommand{\vphi}{\varphi}
\newcommand{\M}{\mathbf{M}}

\newcommand{\B}{\mathcal{B}}
\renewcommand{\o}{\mathfrak{o}}
\newcommand{\p}{\mathfrak{p}}
\newcommand{\q}{\mathfrak{q}}

\renewcommand{\a}{\mathfrak{a}}

\newcommand{\1}{\mathbf{1}}

\renewcommand{\P}{\mathcal{P}}
\newcommand{\Q}{\mathcal{Q}}

\newcommand{\W}{\mathcal{W}}
\newcommand{\Ind}{\mathrm{Ind}}
\newcommand{\V}{\mathcal{V}}

\newtheorem{thm}{Theorem}[section]
\newtheorem{thmA}{Theorem}

\newtheorem{lem}[thm]{Lemma}

\newtheorem{prop}[thm]{Proposition}

\theoremstyle{definition}
\newtheorem{rmk}{Remark}[section]

\newcommand{\pmat}[1]{\begin{pmatrix} #1 \end{pmatrix}}
\newcommand\IP[1]{\left\langle #1\right\rangle}

\title[Regularized Spectral Expansion and Moments]
{Regularized spectral expansion of a Rankin--Selberg period and moments}

\author{Jakub Dobrowolski}
\address{Queen Mary University of London, Mile End Road, London E14 NS, United Kingdom}
\email{j.dobrowolski@qmul.ac.uk}

\author{Subhajit Jana}
\address{Indian Statistical Institute, 8th Mile Mysore Road, RVCE Post, Bangalore 560059, India}
\email{s.jana@isibang.ac.in}

\author{Ramon Nunes}
\address{Universidade Federal do Cear\'a, Campus do Pici, Bloco 914, 60440-900 Fortaleza-CE, Brasil}
\email{ramon@mat.ufc.br}

\thanks{The first and second authors were in part supported by the EPSRC grant EP/Y016769/1. The third author was supported by Instituto Serrapilheira, grant number 8277 and FUNCAP, grant number AJC-0222-00009.01.00/24.}

\begin{document}

	\begin{abstract}
		We establish a regularized spectral decomposition of the squared magnitude of a maximal degenerate Eisenstein series as a Schwartz distribution on $\mathrm{GL}_n$. Although the original problem is on $\mathrm{GL}_n$, its spectral components are described entirely by the automorphic spectrum of $\mathrm{GL}_2$. As an application, we prove a partial reciprocity formula relating the second moment of $\mathrm{GL}_n\times\mathrm{GL}_n$ Rankin--Selberg central $L$-values and a mixed moment of $L$-values of $\mathrm{GL}_2$. We also prove, assuming the generalized Ramanujan conjecture, an asymptotic formula for the second moment of the above Rankin--Selberg $L$-functions over a conductor-aspect family with an error term of square-root-cancellation strength.
	\end{abstract}

	\date{\today}

	\maketitle

	\section{Introduction}\label{sec:intro}

	\subsection{Motivation}

	$L$-functions are fundamental objects of modern number theory and the study of their moments over natural families is one of the central problems in the subject. Let $n>2$ be a natural number and $F$ be a number field. In this paper we study the second moment of Rankin--Selberg $L$-functions of the form
	\(L(1/2,\pi\otimes \pi_0)\) where $\pi_0$ is a fixed cuspidal representation for $\PGL_n(F)$ and $\pi$ varies over generic representations for $\PGL_n(F)$ in a non-archimedean conductor-aspect family.

	We approach this problem via a period-theoretic viewpoint inspired by \cite{MV2010subconvexity}. We start with the following automorphic  integral which captures the desired second moment. Let $\phi_0\in \pi_0$ be a fixed cusp form and $E$ be a maximal degenerate Eisenstein series on $\PGL_n$. We consider the integral
	\[\mathcal{I}(\phi_0,E)\defeq \int_{[\PGL_n]}|\phi_0|^2\cdot|E|^2\]
	where $[\PGL_n]$ denotes the automorphic quotient $\PGL_n(F)\bs\PGL_n(\A)$.
	As $\phi_0$ is a cusp form we know that $\phi_0\cdot E$ lies in $ L^2([\PGL_n])$ and hence the integral converges absolutely. Moreover, writing the above as an inner product
	\[\IP{\phi_0\cdot E,\phi_0\cdot E}_{[\PGL_n]},\]
	we use Langlands $L^2$-decomposition and Rankin--Selberg unfolding to express it as a weighted second moment
	\[\mathcal{I}(\phi_0,E)=\intop_{\mathrm{aut}}{\left\lvert L(\tfrac{1}{2},\pi_0\otimes \ol{\pi})\right\rvert^2} \H_{\phi_0,E}(\pi)\, d\pi.\]
	Here $\H_{\phi_0,E}$ is a certain weight function which,
	for suitable choices of $\phi_0$ and $E$, detects the intended conductor-aspect family.

	The problem is therefore to understand the period $\mathcal{I}(\phi_0,E)$ with sufficient precision.
	We already encounter a serious analytic difficulty as $|E|^2$ is not integrable on $[\PGL_n]$. So, we cannot, at least not naively as we did before, spectrally expand this inner product. Hence we face the problem of finding a \emph{regularized} spectral expansion of  $\left\lvert E\right\rvert^2$ as a distribution against $|\phi_0|^2$.

	We start with the \emph{deformed} integral
	\[\int_{[\PGL_n]} |\phi_0|^2\cdot \ol{E}\cdot E_s,\]
	where for $s\in\C$ in general position $E_s$ is a certain deformation of $E$ varying meromorphically in $s$. The goal is to understand the meromorphic behavior of the above deformed integral near $s=0$, as the $s\to 0$ limit of it (if it exists) equals $\mathcal{I}(\phi_0,E)$.
	For sufficiently large $\Re(s)$, we apply the standard unfolding via $E_s$,
	along with the Fourier--Whittaker expansion of $|\phi_0|^2$, which transforms
	the above integral into a sum of a few Whittaker periods.
	A crucial observation at this point is that all but two Whittaker
	periods vanish because $E$ is very far from being \emph{generic}.
	Among the two non-vanishing periods, one of them can be decomposed into
	two explicitly computable terms which contribute to the
	main term. The other one is essentially a $\GL_2\times \GL_2$ Rankin--Selberg
	period.
	This allows us to obtain a (partial)
	\emph{reciprocity} formula relating the above moment of $\GL_n\times \GL_n$
	Rankin--Selberg $L$-values and $L$-values of $\GL_2$;  see \S\ref{sec:sketch-thmA}. Furthermore, carefully choosing test vectors $\phi_0$ and $E$ we obtain an asymptotic expansion
	of the above moment with an error term of (essentially)
	square-root-cancellation strength.

	The use of the period $\mathcal{I}(\phi_0,E)$ to analyze the second moment above was first introduced by Blomer in \cite{Blomer2012RS}, albeit in a different set-up. In his case the representation $\pi_0$ also varies, resulting in a conductor-dropping scenario. In any case, his method in our set-up would allow us to show that the growth of $\ol{E}\cdot E_s$ is bounded by that of an Eisenstein series in the \emph{absolute convergence} region. This only yields a Lindel{\"o}f-consistent upper bound of the above second moment. The second author in \cite{Jana2020RS} later was able to extract a \emph{main term} $\mathbf{E}$ so that
	\[|E|^2-\mathbf{E}\in L^1([\PGL_n]),\]
	leading to an asymptotic expansion of the above moment, in the archimedean aspect, with an error term of size $O(1)$. This gave a motivation to extract further terms from $|E|^2$ so that one can obtain a power saving error term. In the non-archimedean setting for $n=2$, the corresponding problem was resolved by the first author in \cite{dobrowolski2025second}, using a general regularized spectral decomposition developed adelically in \cite{MV2010subconvexity}. For $n>2$ the problem becomes substantially more difficult, as many more different types of Eisenstein contributions may occur, and there is presently no general regularized spectral decomposition.  It is worth noting that Zhang in \cite{Zhang2019QUE} analyzed the measure coming from the squared magnitude of a maximal degenerate Eisenstein series $|E|^2$ in the context of \emph{quantum unique ergodicity}. The author of \cite{Zhang2019QUE} also observes a phenomenon similar to ours, namely, the spectral components that contribute essentially come from $\GL_2$. However, for his application it suffices to integrate $|E|^2$ \emph{only} against cusp forms and incomplete Eisenstein series so there is no need for regularization, as the periods already converge absolutely.

	\subsection{Main theorems}

	Let $F$ be a number field with discriminant $\Delta$ and $n>2$ be a natural number. Let $\1_{n-1}\boxplus\1_1$ denote the normalized induction of the trivial automorphic representation of the parabolic subgroup of $\GL_n$ attached to the partition $n=(n-1)+1$. Let $\M$ denote the standard intertwiner from $\1_{n-1}\boxplus\1_1$ to $\1_1\boxplus\1_{n-1}$; see \S\ref{sec:intertwiners} for definitions. Given $s\in\C$ we construct certain meromorphic sections $f_s$ in the deformed representation $\1_{n-1}\cdot|\det|^s\boxplus\1_1\cdot|\cdot|^{-(n-1)s}$ as in \eqref{eq:def-f-from-phi}. Finally, given any induced section $\varphi$ and its deformation $\varphi_\lambda$ we denote by $\Eis(\varphi_\lambda)=\Eis(\varphi,\lambda)$ the corresponding Eisenstein series; see \eqref{eq:def-eis-from-f}.

	\begin{thmA}\label{thm:spectral-expansion}
		Fix $S$ to be a sufficiently large finite set of places of $F$ containing all the archimedean and ramified places. For $f=f_0$ arising from a factorizable $\Phi$ with $\Phi_v=\mathbbm{1}_{\o_v^n}$ for $v\notin S$, an induced
		vector in $\1_{n-1}\boxplus\1_1$, we have the following spectral decomposition of $|\Eis(f)|^2$ as a Schwartz distribution on $\PGL_n(F)\bs\PGL_n(\A)$, tested against Schwartz functions that are right $K_{n,v}$-invariant for $v\notin S$ (enlarging $S$ if necessary):
		\begin{multline*}
			|\Eis(f)|^2 = \mathbf{E} + \mathbf{F}+\V\sum_{\sigma}\frac{L^S\left(\tfrac{n-1}{2},\sigma\right)L^S\left(\tfrac{1}{2},\sigma\right)}{\sqrt{L^S(1,\sigma,\Ad)}}\sum_{\varphi\in\B(\1_{n-2}\boxplus\sigma)}\Eis(\bar{\varphi},0)\P_S(\varphi,f)\\
			+\V\sum_{\chi}\intop_{i\mathbb{R}}\frac{\prod_{\pm}L^S\left(\tfrac{n-1}{2}\pm z,\chi^\pm\right)L^S\left(\tfrac{1}{2}\pm z,\chi^\pm\right)}{L^S(1+2z,\chi^2)}\\
			\sum_{\varphi\in\B(\1_{n-2}\boxplus\chi\boxplus\chi^{-1})}\Eis\left(\bar{\varphi},(0,-z,z)\right)\P_S(z,\varphi,f)\, dz,
		\end{multline*}
		where 
    %$\V=\vol([G_{n-2}]^1)\vol([G_1]^1)$ is the volume factor
    $\V$ is given in \eqref{def-mathcal-V},
		\begin{equation*}
			\mathbf{E}:=\lim_{s\to 0}\,\left(\Eis\left(\bar{f}\cdot f_s\right)+\Eis\left(\M\bar{f}\cdot\M f_s\right)\right)\quad\text{and}\quad\mathbf{F}:=\lim_{s\to 0}\,\left(\Eis\left(\bar{f}\cdot \M f_s\right)+\Eis\left(\M\bar{f}\cdot f_s\right)\right),
		\end{equation*}
		and $\P_S(\varphi,f)$ and $\P_S(z,\varphi,f)$ are $S$-adic local periods described in  \eqref{def-main-PSf-disc} and \eqref{def-main-PSf-cont}, respectively. Here the $\sigma$-sum runs over the cuspidal automorphic representations for $\PGL_2(F)$ and the $\chi$-sum runs over Hecke characters for $F^\times$ that are trivial on $\R_+$. For $n=3$ and trivial $\chi$, the value at $z=0$ is understood by continuity.
	\end{thmA}

\subsubsection{A partial reciprocity-type formula}
		Fix a cuspidal automorphic representation $\pi_0$ for $\PGL_n(F)$ that is unramified at every finite place. Choose $\phi_0\in\pi_0$ a cusp form that is spherical at every finite place. We apply Theorem \ref{thm:spectral-expansion} against the Schwartz function $|\phi_0|^2$. Then the Rankin--Selberg theory produces the following identity:
		\begin{equation*}
			\intop_{\pi\text{ on }\GL_n}\frac{|L(\frac{1}{2},\pi\otimes\pi_0)|^2}{L(1,\pi,\Ad)}\  \approx  \intop_{\sigma\text{ on }\GL_2}\frac{L\left(\tfrac{n-1}{2},\sigma\right)L\left(\tfrac{1}{2},\sigma\right)}{\sqrt{L(1,\sigma,\Ad)}}\sum_{\varphi\in\B(\1_{n-2}\boxplus\sigma)}\langle|\phi_0|^2,\Eis(\varphi)\rangle.
		\end{equation*}
		Here by the $\approx$ sign we have hidden boundary terms and local periods. A striking phenomenon here is that the left hand side represents an average of $L$-functions of $\GL_n\times\GL_n$, whereas the right hand side only involves $L$-functions of $\GL_2$. Such formul\ae\ are known as reciprocity formul\ae\ and are extremely rare in nature, especially in higher rank; see, \emph{e.g.} \cite{blomer2024local, BLM2019spectral,BlKh2019reciprocity} for smaller-rank examples and \cite{JaNu2021reciprocity} for a higher-rank example.

\subsubsection{A geometric viewpoint for the appearance of only the $\GL_2$ spectrum}\label{sec:geom-pov}
The maximal degenerate Eisenstein series is induced from
the maximal parabolic $\Q_{n-1}\subset \GL_n$ attached to the partition $n=(n-1)+1$. In the Schwartz-space realization, it may be written as 
\[
E(g)
\approx \sum_{x\in \mathbb{P}^{n-1}(F)}
\int_{\A^\times}
\Phi(txg)|t|^{\frac{n}{2}}\,d^\times t,
\]
(see \eqref{eq:def-eis-from-f} and \eqref{eq:def-f-from-phi}) since \[
\mathcal{Q}_{n-1}(F)\backslash \GL_n(F)
\simeq \mathbb{P}^{n-1}(F).
\]
We will pretend that the above integral converges and there is no pre-factor of $\det g$.
Thus $E$ is built from lines in $F^n$.  Consequently, the product of two such
Eisenstein series is geometrically described by a pair of lines
\[
(\ell_1,\ell_2)\in
\mathbb{P}^{n-1}(F)\times\mathbb{P}^{n-1}(F).
\]
Under the diagonal action of $\GL_n(F)$ there are only two possibilities:
either $\ell_1=\ell_2$, or $\ell_1\neq \ell_2$.  In the second case,
\[
W=\ell_1+\ell_2
\]
is a two-dimensional subspace of $F^n$.  Thus a nondegenerate pair of lines
determines a $2$-plane $W$ and all the remaining nontrivial geometry takes place inside this
two-dimensional space.  Its natural symmetry group is
$\GL(W)\simeq \GL_2$.  This is the geometric reason that only the
$\GL_2$-spectrum appears. We also refer to \cite[\S2.2, Theorem 2.2 and Remark 2.3]{mao2026strongly} for comparison from the theta-correspondence point of view.

\vspace{5mm}

	In order to study a second-moment average of Rankin--Selberg $L$-functions, we make suitable choices for the cusp form $\phi_0$ and the induced vector $f$ to detect a non-archimedean, conductor-aspect family. We use the spectral decomposition of $\langle|\phi_0|^2,|\Eis(f)|^2\rangle$ to arrive at the following second moment asymptotics.

	Let $\ell(\pi)$ denote the harmonic weight relating the automorphic norm to the product of the local Whittaker norms in \eqref{eq:def-inner-prod-whittaker}; its comparison with \cite[Lemma 4.1]{JaNu2021reciprocity} is given in \S\ref{sec:firstspecexp}.
	Also, let $d\mu_{\loc}$ (resp.\ $d\mu_{\mathrm{aut}}$) denote the local (resp.\ automorphic) Plancherel density compatible with the Haar measures on $\PGL_n$ over the underlying local field (resp.\ over adeles) and let $\int_\gen$ denote an integral over the generic representations.

	\begin{thmA}\label{thm:second_moment_asymptotic}
		Let $\pi_0$ be a fixed cuspidal representation for $\PGL_n(F)$ that is unramified at every finite place and tempered at all archimedean places.
		Let $\q=\prod_v \p_v^{r_v}$ be an integral ideal coprime to the different ideal of $F$ such that $N(\q)\to \infty$ and let $(C_v)_{v\mid \infty}$ be a collection of fixed real numbers satisfying $C_v\geq 1$. Then there exists a spectral weight function $\H=\H_\q$ implicitly depending on $\pi_0$ and $(C_v)_v$ defined on the generic automorphic representations for $\PGL_n(F)$, factorizing as
		\begin{equation*}
			\H(\pi)=\prod_{v<\infty}H_v(\pi_v)\prod_{v\mid\infty}h_v(\pi_v),\quad \pi=\otimes'_v\pi_v,
		\end{equation*}
		satisfying the following properties:
		\begin{itemize}
			\item if $v<\infty$ and $c(\pi_v)>r_v$ then $H_v(\pi_v)=0$;
			\item if $v<\infty$ and $c(\pi_v)\le r_v$ then $H_v(\pi_v)\geq 1$ with equality when $r_v=0$;
			\item if $v\mid\infty$ and $C(\pi_v)\leq C_v$ then $h_v(\pi_v) \gg_{C_v} 1$;
			\item if $v\mid\q$ then $\int_{\gen}H_v(\pi_v)\, d\mu_{\loc}\pi_v\asymp \mathrm{vol}^{-1}\left(K_0\left(\p_v^{r_v}\right)\right)$;
		\end{itemize}
		such that we have
		\begin{multline*}
			\int_{\mathrm{gen}} \frac{|L(1/2,\pi_0\otimes \tilde{\pi})|^2}{\ell(\pi)}\H(\pi)\,d\mu_{\mathrm{aut}}(\pi)\\
			=\vol(K_0(\q))^{-1}\left(\Delta^{\mu}L(1,\pi_0,\mathrm{Ad})\frac{n\zeta^\ast(1)\zeta(\frac{n}{2})^2}{\zeta(n)}\log N(\q)+B_\q\right.\\
			\left.+N(\q)^{-\frac{n-2}2}\frac{\zeta_\q(\frac{n}{2})}{\zeta_\q(1)}\left(C\log N(\q)+D_\q\right)+O_{F,\pi_0,\epsilon}\left(N(\q)^{-\frac{n-1}{2}+\vartheta+\epsilon}\right)\right),
		\end{multline*}
		where $0\le\vartheta\le\frac{7}{64}$ is a bound towards the generalized Ramanujan conjecture and $\mu$ is a constant depending only on $n$. Here $B_\q$ and $D_\q$ are explicit $\q$-dependent constants, given by \eqref{eq:def-Bq} and \eqref{eq:def-Dq}, respectively, and are $O_{F,\pi_0}(1)$; and $C$ is a $\q$-independent constant given by \eqref{eq:def-C}.
	\end{thmA}

	\begin{rmk}
		We note that for $F=\mathbb{Q}$ the leading coefficient in the asymptotic formula of Theorem \ref{thm:second_moment_asymptotic} matches, with compatible Haar measures and test-vector normalizations, that of \cite[Theorem 1]{Jana2020RS}.
	\end{rmk}

	\begin{rmk}
		We remark on the properties of $\H(\pi)$ described in Theorem \ref{thm:second_moment_asymptotic}. The first three properties assert that it essentially works as a projector on the space of representations $\pi$ with finite conductor dividing $\q$ and bounded archimedean conductors. The fourth property of $\H(\pi)$ states that its $L^1$-norm is of size $\vol(K_0(\q))^{-1}$.
	\end{rmk}

	\begin{rmk}
		The strength of Theorem \ref{thm:second_moment_asymptotic} is essentially of the order of square-root cancellation assuming the generalized Ramanujan conjecture, which can be perceived as the boundary of the current technology.
	\end{rmk}

	\begin{rmk}
		Finally, since the family is of size $N(\q)^{n-1}$ and the square of a Rankin--Selberg $L$-function has conductor up to $N(\q)^{2n}$, for $n>2$ this is outside the subconvexity range.
	\end{rmk}

	\subsection{Proof ideas of Theorem \ref{thm:spectral-expansion}}\label{sec:sketch-thmA}

	We give a high level sketch of the proof of the regularized spectral decomposition of $|\Eis(f)|^2$. For that we start with a Schwartz function $\phi$ on $[\PGL_n]$ and analyze the period
	\[
	\int_{[\PGL_{n}]}\phi\cdot\ol{\Eis(f)}\cdot\Eis(f_s)
	\]
	with sufficiently large $\Re(s)$. We unfold one Eisenstein series and apply the Fourier--Whittaker expansion of $\phi$ from \cite{IchinoYamana2016}. Because $\Eis(f)$ is highly non-generic, all but two types of partial Whittaker coefficients get annihilated. One of them contributes
	\begin{equation*}
		\int_{[\PGL_{n}]}\phi\cdot\Eis(\bar{f}\cdot f_s)+\int_{[\PGL_{n}]}\phi\cdot\Eis(\M\bar{f}\cdot f_s),
	\end{equation*}
	where $\M$ is the intertwining operator mentioned above.
	The other one, named $I_2(s)$, can be rewritten as a period involving $F_s$ that resembles a $\GL_2\times\GL_2$ global zeta integral, where $F_s$ is an automorphic function on $[\PGL_2]$. We first spectrally decompose $F_s$ over the $[\PGL_2]$ spectrum
	\[F_s(g)=\intop_{\pi\text{ on }\PGL_2}\sum_{\varphi\in \B(\pi)} \langle F_s, \varphi\rangle \varphi(g)\, d\pi.\]
	Our treatment of the above period is inspired by the recent work of Boisseau \cite{fine-boi}.
	We ``inflate'' the above into a part of the $[\PGL_n]$ spectrum and get an expression roughly of the form
	\[F_s(g)=\intop_{\pi\text{ on }\PGL_2}\sum_{\varphi\in\B\left(\1_{n-2}\boxplus\pi\right)}
	\IP{\phi,\Eis(\varphi, \ol{\lambda_\pi(s)})}\varphi_{-\lambda_{\pi}(s)}\left[\begin{pmatrix}\mathrm{I}_{n-2}&\\&g\end{pmatrix}\right]\, d\pi,\]
	where the spectral integral is over automorphic generic representations for $\PGL_2(F)$ and $\lambda_\pi(s)$ is a certain co-weight. Subsequently, we obtain
	\begin{equation*}
		I_2(s) = \intop_{\pi\text{ gen on }\PGL_2}\sum_{\varphi\in\B\left(\1_{n-2}\boxplus\pi\right)}\IP{\phi,\Eis(\varphi, \ol{\lambda_\pi(s)})}\P(\varphi_{-\lambda_{\pi}(s)},\bar{f},f_s)\, d\pi,
	\end{equation*}
	where $\P$ is a certain period which resembles a $\GL_2\times\GL_2$ Rankin--Selberg zeta integral.

	The next task is to meromorphically continue $I_2(s)$ towards $s=0$ from $\Re(s)>0$, which is the most technical part of the paper. The analysis in this part has certain similarities to that of \cite{fine-boi}. There are two sources of polar terms in the above expression of $I_2(s)$, namely, from the poles of $\Eis(\cdots)$ and poles of $\P(\cdots)$. Fortunately, they occur in disjoint hyperplanes, where the assumption $n>2$ becomes necessary\footnote{In fact, from \cite{dobrowolski2025second} it can be seen that these poles coincide for $n=2$.}. We extensively use the meromorphic properties of $\Eis(\cdots)$ in the ``positive Weyl chamber'' and that of the local zeta integral $\P(\cdots)$ to compute the orders of the respective poles and residues. As a pleasant surprise, the residues are exactly equal to
	\begin{equation*}
		\int_{[\PGL_n]}\phi\cdot\Eis\left(\M\bar{f}\cdot\M f_s\right)\quad\text{and}\quad\int_{[\PGL_n]}\phi\cdot\Eis\left(\bar{f}\cdot \M f_s\right).
	\end{equation*}
	We emphasize that determining the residues, which, \emph{a priori}, appear as spectral sums of periods, equal to the above inner products requires a non-trivial amount of work. We use spectral expansion, the Plancherel formula, and uniqueness of Whittaker functionals on $\GL_2$. We expect that a more direct proof should be possible, using uniqueness of trilinear functionals of (not-necessarily generic) representations of $\GL_n$, somewhat along the lines of \cite{SZ2012multiplicity,sun2012multiplicity}, which treat bilinear functionals. However, we were not able to identify a precise reference in the literature. Combining all the components and taking the limit $s\to 0$ yields Theorem \ref{thm:spectral-expansion}.

	\subsection{Proof ideas of Theorem \ref{thm:second_moment_asymptotic}}

	We prove Theorem \ref{thm:second_moment_asymptotic} after choosing a specific test vector $f$ in Theorem \ref{thm:spectral-expansion}, integrating against the squared magnitude of a chosen cusp form $\phi_0\in\pi_0$ and suitably (\emph{cf}.\ beginning of \S \ref{sec:firstspecexp}) normalizing the period. The choices are inspired by those of \cite{Jana2020RS}. The main idea is to choose the vectors so that $W_{\phi_0}\cdot f$ at the $\q$-adic place resembles the characteristic function of the Hecke congruence subgroup $K_0(\q)$, so that the \emph{newvector theory} can be applied to pick up the family of representations that have conductors dividing $\q$.
	The main inputs are from the local computations, which give the main terms
	\[\int_{[\PGL_n]}|\phi_0|^2\cdot \mathbf{E}=(A\log N(\q) + B), \]
	and
	\[\int_{[\PGL_n]}|\phi_0|^2\cdot\mathbf{F}=N(\q)^{-\frac{n-2}{2}}(C\log N(\q) + D), \]
	and the error term
	\begin{equation*}
		\P(\cdots) \ll N(\q)^{-\frac{n-1}{2}+\vartheta}
	\end{equation*}
	where $\vartheta$ is any bound towards the Ramanujan Conjecture for $\GL_2$.

	\section*{Acknowledgments}
	We are deeply thankful to Paul Boisseau for explaining several technical parts from his preprint \cite{fine-boi} and correcting a few inaccuracies, Paul Nelson for feedback and helping us with the point of view in \S\ref{sec:geom-pov}, and Valentin Blomer for several helpful comments on an earlier draft. The authors are indebted to Aarhus University, Queen Mary University of London and Universidade Federal do Cear\'a for their wonderful hospitality and working conditions where most of this work was done.

	\section*{Statement on the use of AI}

	The work was completed without use of any AI. After finishing we have used generative AI to help with grammar, language and proofreading the text.

	\part{Preliminaries}

	\section{Preliminary Discussion and Notations}

	\subsection{General notation}\label{sec:general-notion}

	Let $n\ge 3$ be a natural number. Let $F$ denote a number field with adele ring $\A$. We denote the norm on $\A^\times$ by $|\cdot|$ and the norm on the $v$-adic completion of $F$ by $|\cdot|_v$, so that $|\cdot|=\prod_v|\cdot|_v$. We abbreviate $F_\infty:=\prod_{v\mid\infty} F_v$. For any global object $\pi$ (such as an automorphic representation, an $L$-function etc.), we denote the corresponding $v$-adic local factor by $\pi_v$. Also, for a finite set $S$ of places of $F$ we abbreviate $\pi^S:=\otimes_{v\notin S}\pi_v$ and $\pi_S:=\otimes_{v\in S}\pi_v$. Also, for $\mathfrak{A}$ being $\infty$ or an integral ideal $\q$ and $S_{\mathfrak{A}}=\{v\mid\mathfrak{A}\},$ we abbreviate $\pi^{\mathfrak{A}}$ for $\pi^{S_{\mathfrak{A}}}$ and $\pi_{\mathfrak{A}}$ for $\pi_{S_\mathfrak{A}}$.

	If $v$ is non-archimedean we denote the ring of integers of $F_v$ by $\o_v$. In this case, we denote the maximal ideal of $\o_v$ as $\p_v$ and its uniformizer by $\varpi_v$. Finally, we define $N(\p_v)$ to be the cardinality of $\o_v/\p_v$ and define $\zeta_v(s):=(1-N(\p_v)^{-s})^{-1}$. The global zeta function is denoted by $\xi=\prod_v\zeta_v$. Similarly, the global $L$-function is denoted by $\Lambda=\prod_v L_v$. We also use the notation $L=\Lambda^\infty$ and $\zeta=\xi^\infty$.

	Let $G$ be a reductive algebraic group defined over $F$. Let $Z$ denote the center of $G$ defined over $F$ and we write $\bar{G}:=Z\backslash G$. For any algebraic subgroup $H$ of $G$ defined over $F$, by $[H]$ we denote the quotient $H(F)\backslash H(\A)$. For each place $v$ of $F$ we fix a hyperspecial maximal compact subgroup $K_{n,v}$ of $G(F_v)$ and $K=\prod_v K_{n,v}$ of $G(\A)$.

	Let $P$ be a standard parabolic subgroup of $G$ with Levi subgroup $M_P$ and unipotent radical $U_P$. When the parabolic subgroup is clear from the context we will drop the subscript $P$. Let $X^\ast(M)$ denote the lattice of the $F$-rational characters of $M$. We denote the $\R$-vector space spanned by $X^\ast(M)$ by $\a_M^\ast$, its dual by $\a_M$ and its complexification $\a^\ast_M\otimes_\R\C$ by $\a^\ast_{M,\C}$. We fix a pairing $\langle,\rangle$ between $\a_M$ and $\a^\ast_M$ and define $H_M:M(\A)\to\a_M$ to be the natural homomorphism defined by
	\begin{equation*}
		\exp(\langle\chi, H_M(m)\rangle) =|\chi(m)|,\quad \chi\in X^\ast(M).
	\end{equation*}
	We also denote the kernel of $H_M$ by $M(\A)^1$. We also use the notations $\a_P=\a_M$, $\a^\ast_P=\a^\ast_M$ and $\a^\ast_{P,\C}=\a^\ast_{M,\C}$. We extend $H_M$ to $H_P: G(\A)\to\a_P$ by making the former left-$U(\A)$ and right-$K$ invariant.

	The modular character attached to $P$, denoted by $\delta_P$, is given by $m\mapsto \exp(\langle 2\rho_P, H_M(m)\rangle)$, where $\rho_P$ is the half-sum of the positive roots associated to $P$. Finally, we call a $\lambda\in\a^\ast_{P,\C}$ sufficiently dominant if $\Re\langle\lambda,\alpha\rangle$ are sufficiently positive for all positive roots $\alpha$.

	We fix a maximal $F$-split torus and denote its centralizer in $G$ by $M_0$.
	We fix a minimal parabolic $P_0$ containing $M_0$ and call a parabolic standard if it contains $P_0$. Let $W^P$ denote the Weyl group of $M$. We abbreviate $W^G$ as $W$.

	\subsection{Characters and Measures}\label{subsec:chars-and-measures}

	We fix an additive character $\psi_0$ of $F\backslash \A$, for which, for convenience, we make the specific choice $\psi_0:=\psi_{\mathbb{Q}}\circ \operatorname{tr}_{F/\mathbb{Q}}$ where $\psi_{\mathbb{Q}}:\mathbb{Q}\backslash\A_\mathbb{Q}\to \C$ is an everywhere unramified character. We know that $\psi_0=\otimes \psi_{0,v}$, where $\psi_{0,v}$ is unramified everywhere except for the places dividing the discriminant of $F$. Let $\p_v^{d_v}$ be the conductor of $\psi_{0,v}$. We let $\Delta_v:=N(\p_v)^{d_v}$, so that $\Delta_v=1$ for almost every $v$. If we further extend this definition to archimedean places by letting $\Delta_v=1$ if $v\mid \infty$, then the discriminant of $F$, which we denote by $\Delta$, satisfies $\Delta=\prod_{v<\infty}\Delta_v$.

	For each place $v$ of $F$ we fix a Lebesgue measure $dx_v$ on $F_v$ such that if $v$ is non-archimedean it gives volume $\Delta^{-1/2}_{v}$ to $\o_v$ and if $v$ is real (resp. complex) $dx_v$ is given by the usual (resp. twice the usual) Lebesgue measure  of $F_v$. Then we take the Lebesgue measure on $\A$ given by $dx=\prod_v dx_v$, which gives volume one to the quotient $F\bs \A$ when $F$ is equipped with the counting measure. Similarly, we fix a Haar measure on $F_v^\times$ given by $d^\times y_v:=\zeta_v(1)\frac{dy_v}{|y_v|}$ at archimedean places and $d^\times y_v:=\Delta_v^{1/2}\zeta_v(1)\frac{dy_v}{|y_v|}$ at non-archimedean places. If $v$ is non-archimedean then $d^\times y_v$-volume of $\o_v^\times$ is $1$. Finally we equip $\A^\times$ with the measure $d^\times y:=\prod_v d^\times y_v$ on $\A^\times$.

	Finally, for $G=G_n$ we equip $G(F_v)$ with a measure inherited from the Iwasawa decomposition as follows: We fix the probability Haar measure $dk_v$ on $K_{n,v}$, we equip $T_{n,v}$ and $N_{n,v}$ with product measures coming from the obvious isomorphisms $T_n(F_v)\simeq (F_v^\times)^n$ and $N_n(F_v)\simeq F_v^{\frac{n(n-1)}{2}}$. We then choose $dg_v$ to be the unique Haar measure such that
	\[
	\int_{G_n(F_v)}f(g_v)\,dg_v=\int_{N_n(F_v)}\int_{T_n(F_v)}\int_{K_{n,v}}f(n_vt_vk_v)
	\delta_{B_n}^{-1}(t_v)\,dn_v\,dt_v\,dk_v,
	\]
	where $\delta_{B_n}$ is the modular character of $B_n$.
	By the usual process we obtain Haar measures on the groups of $G_n(\A),\ N_n(\A),\ T_n(\A)$ and $K_n$.

	Let $P$ be any standard parabolic of $G$. Let $U$ be its unipotent radical and $M$ its Levi subgroup. For each place $v$, we equip $M(F_v)$ with the measure inherited from the natural isomorphism with $\GL_{n_1}(F_v)\times\dots\times\GL_{n_k}(F_v)$ and similarly for $U(F_v)$ and $F_v^{\dim U}$. We claim that this choice entails the compatibility formula
	\begin{equation}\label{dg-for-iwasawa-PK}
		dg_v =\delta_P^{-1}(m_v)\,du_v\,dm_v\,dk_v,
	\end{equation}
	where, now, $\delta_{P}$ is the modular character of $P$.
	This can be easily verified at non-archimedean places by comparing the volume of the maximal compact subgroup with respect to both measures and noticing that, at archimedean places we have that $P(F_v)\cap K_{n,v}$ is the maximal compact subgroup of $M(F_v)$, which has volume one by our choice.
	We then equip $G_n(\A),\ N_n(\A),\ T_n(\A)$ and $K_n$ with product measures.

	For any subgroup $H\le G$ we equip $H(F)$ with the counting measure. By construction we have that $\vol([U])=1$ for any unipotent radical of a standard parabolic subgroup. Moreover by comparing to the Tamagawa measure (see \cite[\S X.3]{cassels1967algebraic}) we observe that, with our choice of measure, one has
	\begin{equation}\label{volume-Gk1}
		\vol([\G_k^1])=\Delta^{\frac{k^2+k}{4}}\xi^*(1)\xi(2)\cdots\xi(k).
	\end{equation}
	We also fix measures on $\a_M$ so that the volume of
  $\a_M/\mathrm{Hom}(X^\ast(M),\Z)$ is $1$. We denote the pull-back
  of $dm$ to $M(\A)^1$, compatible with the chosen Lebesgue measure on
  $\a_M$, by the same notation. It is well-known that the volume of
  $[M^1]$ according to $dm$ is finite.
  We write $[G]^1=G(F)\bs G(\A)^1$. The measure on $[\bar G_n]$ is the
  quotient of $dg$ by $d^\times z$.
  Throughout, $dz$ in contour integrals denotes $1/(2\pi i)$ times the ordinary complex differential.
  %, oriented upwards on $i\R$.
	We also equip $U(F_v)$ (resp. $M(F_v)$) with local Haar measures $du_v$ (resp. $dm_v$) satisfying $du=\prod_vdu_v$ (resp. $dm=\prod_v dm_v$).

	For a subgroup $H$ of $G$ defined over $F$ we also define a quotient measure on $H(F_v)\bs G(F_v)$ according to \cite[end of Section 1.2]{feigon2012representation}. Namely, we realize $dx$ on $H(F_v)\bs G(F_v)$ as a linear functional on $\tfrac{\delta_H}{\delta_G}$-left equivariant continuous functions $f$ so that the following is true:
	\begin{equation*}
		\int_{G(F_v)} f(g)\, dg = \int_{H(F_v)\bs G(F_v)}\int_{H(F_v)}\frac{\delta_G}{\delta_H}(h)f(hx)\, dh\, dx.
	\end{equation*}

	\subsection{Automorphic representations}\label{sec:auto-rep}

	Let $P$ be a parabolic subgroup with Levi decomposition $P=MU$. Let $\sigma$ be a discrete automorphic representation of $M(\A)$ and $\lambda\in\a_{P,\C}^\ast$. By $\Ind_{P(\A)}^{G(\A)}\,\sigma\otimes e^{\langle\lambda,H_P(\cdot)\rangle}$, abbreviated as $\mathcal{I}(\sigma,\lambda)$, we denote the unitary parabolic induction of $\sigma$ twisted by $\lambda$. A non-trivial unitary inner product on the space $\mathcal{I}(\sigma):=\mathcal{I}(\sigma,0)$ is given by
	\begin{equation*}
		\langle \varphi_1, \varphi_2\rangle_{\Ind}:=\int_{K}\langle \varphi_1(k),\varphi_2(k)\rangle_\sigma\, dk,
	\end{equation*}
	where
	\[\langle \phi_1,\phi_2 \rangle_{\sigma}=\int_{M(F)\backslash M(\A)^1}\phi_1(m)\overline{\phi_2(m)}\,dm.\]

	Let $\varphi$ be a vector in $\mathcal{I}(\sigma)$ and $\lambda\in\a^\ast_{P,\C}$. We define the Eisenstein intertwiner
	$$\Eis: \mathcal{I}(\sigma,\lambda)\to C^\infty([G])$$
	as
	\begin{equation}\label{eq:def-eis-from-f}
		\Eis(\varphi,\lambda):=\Eis(\varphi_\lambda):= \sum_{\gamma\in P(F)\backslash G(F)}\varphi_\lambda(\gamma\cdot)(1),
	\end{equation}
	where $\varphi_\lambda:=\varphi\cdot e^{\langle\lambda,H_P(\cdot)\rangle}$.
	The above sum converges absolutely if $\lambda$ is sufficiently dominant and admits a meromorphic continuation elsewhere; see \cite[Chapter IV]{MW} or \cite{bernstein2024meromorphic} for a more recent and simpler proof.

	Given any automorphic form $\phi$ on $U(F)\bs G(\A)$ we denote the constant term of $\phi$ along $P$ by
	\begin{equation*}
		\phi_P := \int_{[U]}\phi (u\cdot)\, du.
	\end{equation*}
	Note that $\phi_P$ is naturally an automorphic form on $U(\A)M(F)\backslash G(\A)$. We say $\phi$ is cuspidal if $\phi_P=0$ for any proper parabolic $P$ of $G$ defined over $F$. On the other hand, if $\phi$ is Eisenstein then using a Bruhat decomposition one can compute its constant term. We record the following result from \cite[Lemma 6.10]{bernstein2024meromorphic} (also see, \cite[Proposition II.1.7]{MW} for the cuspidally induced Eisenstein case). Let $Q$ be a parabolic of $G$. Then
	\begin{equation}\label{eq:constant-term-expansion}
		\Eis(\varphi,\lambda)_Q=\sum_{w\in W^Q\bs W/W^P} \Eis^Q\left(\M(w,\lambda)\left(\varphi_{P_w}\right),w\lambda\right),
	\end{equation}
	where $P_w$ is a standard parabolic subgroup with Levi $M_P\cap w^{-1}M_Qw$. See \S\ref{sec:intertwiners} for the definition of $\M(w,\lambda)$.

	Here and elsewhere in the paper, for an induced vector $\varphi$ from a parabolic $Q'\subset Q$ and $\lambda\in\a_{Q',\C}$, by $\Eis^Q(\varphi,\lambda)$ we denote the \emph{partial} Eisenstein series where we replace the sum over $P(F)\bs G(F)$ in \eqref{eq:def-eis-from-f} by the sum over $Q'(F)\bs Q(F)$, which converges for sufficiently dominant $\lambda$ and has meromorphic continuation elsewhere.

	\subsection{Intertwining operators}\label{sec:intertwiners}

	For a Weyl element $w$ such that $wM_Pw^{-1}=M_Q$ we define $\M_\sigma(w,\lambda)$, an intertwining operator, as a map
	\begin{equation*}
		\M_\sigma(w,\lambda):\mathcal{I}(\sigma)\to \mathcal{I}(w\sigma),
	\end{equation*}
	defined by
	\begin{equation*}
		\varphi\mapsto e^{\langle-w\lambda,H_Q(\cdot)\rangle}\int_{(wU_Pw^{-1}\cap U_Q)(\A)\bs U_Q(\A)}\varphi_\lambda(w^{-1}u)\, du.
	\end{equation*}
	The above converges absolutely if $\lambda$ is sufficiently dominant with respect to the positive roots relative to $P$ and can be meromorphically continued to all $\a^\ast_{P,\C}$.

	We remark that the expression \eqref{eq:constant-term-expansion} and the definition of the intertwiner seem different from those in \cite[\S 6.10]{bernstein2024meromorphic}. However, one can prove their equivalence using \cite[\S 6.7]{bernstein2024meromorphic} and \cite[\S 3.5]{jana2024local-l2}.

	We also adopt the following normalizations of the intertwining operators. First, we define \emph{local} intertwining operators $\M_{\sigma_v}(w,\lambda)$ acting on $\mathcal{I}(\sigma_v)$ analogously to the above global definition. Then there exist scalar-valued functions $n_{\sigma_v}(w,\lambda)$ which are meromorphic in $\lambda$ such that if $\sigma_v$ is unramified then
	\begin{equation*}
		\widetilde{\M}_{\sigma_v}(w,\lambda):=n_{\sigma_v}(w,\lambda)^{-1}\M_{\sigma_v}(w,\lambda)
	\end{equation*}
	maps the normalized spherical vector of $\mathcal{I}(\sigma_v)$ to that of $\mathcal{I}(w\sigma_v)$. Here, by normalized spherical vector we mean a vector in the induced model whose restriction to the maximal compact $K_v$ equals $1$. The normalized intertwining operators are unitary if $\sigma_v$ is unitary and $\lambda\in i\a^\ast_{P}$. Consequently, we define the \emph{global} normalizing scalar $n_\sigma(w,\lambda)$ by $\prod_v n_{\sigma_v}(w,\lambda)$ which converges absolutely for sufficiently dominant $\lambda$ and has meromorphic continuation elsewhere. It then follows that
	\begin{equation*}
		\M_\sigma(w,\lambda) = n_\sigma(w,\lambda)\prod_v\widetilde{\M}_{\sigma_v}(w,\lambda).
	\end{equation*}
	We refer to \cite[\S 2]{muller2002spectral} for a more detailed discussion.

	We record the functional equations for the intertwiners and Eisenstein series from, \emph{e.g.} \cite[Theorem 2.3]{bernstein2024meromorphic}. In our notation, we have
	\begin{equation}\label{eq:FE-for-intertwiner}
		\M_{w\sigma}(w',w\lambda)\circ\M_\sigma(w,\lambda)=\M_\sigma(w'w,\lambda)
	\end{equation}
	and
	\begin{equation}\label{eq:FE-for-Eis}
		\Eis(\varphi,\lambda) = \Eis\left(\M_\sigma(w,\lambda)\varphi,w\lambda\right),\quad\varphi\in\mathcal{I}(\sigma).
	\end{equation}
	We will often abbreviate the intertwiners as $\M(w)$ (resp.\ $\M$) when $\sigma$ and $\lambda$ (resp.\ $w$) are clear from the context, and also abbreviate \eqref{eq:FE-for-Eis} as $\Eis\left(\varphi_\lambda\right)=\Eis\left(\M\varphi_\lambda\right)$.

	Finally, we abbreviate
	\begin{equation}\label{eq:def-main-intertwiner}
		w_\ell:=w_\ell^n:=\begin{pmatrix}&1\\ \mathrm{I}_{n-1}&\end{pmatrix},\quad \M:=\M(w_\ell,0).
	\end{equation}
	Also, by $\M^\vee$ we denote the intertwiner $\M$ attached to $w_\ell^{-1}$.

	\subsection{Spectral decomposition}

	We fix a Siegel set $\mathbb{S}$ of $G$ as in \cite[\S I.2.1, p.20]{MW}. In particular, on our chosen Siegel set we have that
	\begin{equation*}
		\alpha\left(H_{P_0}(g)\right) \gg 1,\quad \alpha\in\Delta_0,\quad g\in\mathbb{S},
	\end{equation*}
	where $\Delta_0$ is the set of positive roots with respect to $P_0$. We say that a function $\phi\in C^\infty([G])$ is \emph{Schwartz} if
	\begin{equation*}
		\phi(g)\ll_{\lambda,\phi} e^{\left\langle\lambda, H_{P_0}(g)\right\rangle}, \quad \lambda\in\a^\ast_{P_0},\quad g\in\mathbb{S},
	\end{equation*}
	and the same holds for all its derivatives; \emph{cf.}, \cite[\S I.2.12]{MW}. We denote the Schwartz class by $\mathcal{S}([G])$. A similar estimate immediately follows from the definition of the constant term $\phi_P$, namely:

	\begin{lem}\label{lem:general-bound-constant}
		Let $\phi\in\mathcal{S}([G])$. Then $\phi_P$ is a smooth function on
		$U(\A)M(F)\bs G(\A)$ with
		\begin{equation*}
			\phi_P(g) \ll_{\lambda,\phi} e^{\left\langle\lambda, H_{P_0}(g)\right\rangle},\quad g\in \mathbb{S},
		\end{equation*}
		for any $\lambda\in\a^\ast_{P_0}/\a^\ast_G$. Moreover, the same bounds hold for $\mathcal{D}\phi$, for any differential operator $\mathcal{D}$ of $G(F_\infty)$ (with implied constant allowed to depend on $\mathcal{D}$).
	\end{lem}

	Given any Hilbert space $\mathcal{H}$, by $\B(\mathcal{H})$ we denote an
	(a priori) arbitrary orthonormal basis of $\mathcal{H}$. Here and elsewhere,
	$\Pi(M)$ (resp.\ $\Pi_{\mathrm{c}}(M)$) denotes the isomorphism class of
	irreducible discrete (resp.\ cuspidal) automorphic representations of
	$M(\A)^1$. Here we record the Langlands spectral decomposition from
	\cite[section 7]{arthur2005introduction} (see also \cite[\S 2.2.1]{MV2010subconvexity}).
	For $\phi\in\mathcal{S}([G])$ we have
	\begin{equation}\label{eq:spectral-decomposition}
		\phi(g)=\sum_Pn_P^{-1}\sum_{\pi\in\Pi(M)}\intop_{i\a_P^\ast}\sum_{\varphi\in\B(\mathcal{I}(\pi))}\IP{\phi,\Eis(\varphi,\lambda)}_{[G]}\Eis(\varphi,\lambda)(g)\, d\lambda
	\end{equation}
	where the sum over $P$ runs over all standard parabolic subgroups, $n_P$ denotes the number of semi-standard parabolic subgroups of $G$ with Levi subgroup $M_P$, and $d\lambda$ is the Fourier-dual Haar measure on $i\a_P^\ast$ for the above measure on $\a_P$. The sum and integral above converge absolutely and the expression is independent of the choice of a basis. Finally, if $\phi\in\mathcal{S}([\bar{G}])$ then by Mellin inversion we see that in the spectral decomposition the coefficient is taken over $[\bar G]$, the $\lambda$-integral will be over $i\a_P^\ast/i\a_G^\ast$, and $\pi$ will run over irreducible discrete (resp. cuspidal) automorphic representations on which $Z_G$ acts trivially. We denote the last set by $\Pi^G(M)$ (resp.\ $\Pi^G_{\mathrm{c}}(M)$).

	\section{Auxiliary results for $\GL_n$}

	We denote the algebraic groups $\GL_n$ by $G_n$ and $\PGL_n$ by $\bar{G}_n$ over a number field $F$.
	We denote the Borel subgroup of upper triangular matrices in $G_n$ by $B_n$. We know that $B_n=T_nN_n$, where $N_n$ is the unipotent radical of $B_n$ consisting of all upper triangular unipotent matrices and $T_n$ is the maximal torus consisting of the diagonal matrices. Let $\psi_0$ be the character defined in \S \ref{subsec:chars-and-measures}. We define a character of $N_n$ by
	\begin{equation*}
		x\mapsto\psi_0\left(\sum_{i=1}^{n-1}x_{i,i+1}\right),\quad x\in N_n,
	\end{equation*}
	which we denote by $\psi$ as well as its restriction to any subgroup of $N_n$.

	\subsection{Schwartz space parametrization}\label{sec:Schwartz-space}

	Consider the parabolic subgroup of $G_n$ associated to the partition $n=n_1+\dots+n_k$. Let $\sigma:=\sigma_1\otimes\dots\otimes\sigma_k$ be a discrete automorphic representation of the Levi $G_{n_1}\times\cdots\times G_{n_k}$. We abbreviate the induced representation $\mathcal{I}(\sigma)$ by $\sigma_1\boxplus\dots\boxplus\sigma_k$.

	The case $k=2$ and $n=(n-1)+1$ is especially relevant to us. Consider $\sigma=\1_{n-1}\otimes\1_1$, where, here and elsewhere, $\1_r$ denotes the trivial representation of $G_r(\A)$ embedded in $L^2([G^1_r])$. For $s\in\C$ we define
	\begin{equation}\label{eq:def-lambda-E}
		\lambda_{\mathrm{E}}(s):=(s,-(n-1)s)\in\a^\ast_{\Q_{n-1},\C}.
	\end{equation}
	For $\Phi\in\mathcal{S}(\A^n)$ we define a vector $f_{s,\Phi}\in \mathcal{I}(\sigma,\lambda_{\mathrm{E}}(s))=\1_{n-1}\cdot|\det |^s\boxplus\1_1\cdot|\cdot|^{-(n-1)s}$
	by
	\begin{equation}\label{eq:def-f-from-phi}
		f_{s,\Phi}(g):=|\det g|^{\frac{1}{2}+s}\int_{\A^\times}\Phi(te_ng)|t|^{n(\frac{1}{2}+s)}\, d^\times t,
	\end{equation}
	where $e_n:=(0,\dots,0,1)\in\A^n$. The above converges absolutely for $\Re(s)>\frac{1}{n}-\frac{1}{2}$ and has meromorphic continuation to all of $\C$; see \cite[\S 2.3.1]{cogdell2007functions}.

	We also record a corresponding local variant of the above construction. For $\Phi_v\in\mathcal{S}(F_v^n)$ we define $f_{s,\Phi_v}$ where we replace $\Phi$ and $\A^\times$ in \eqref{eq:def-f-from-phi} by $\Phi_v$ and $F_v^\times$, respectively.
	\begin{rmk}
		We make the reader cautious that $f_{s,\Phi}$ and $(f_{0,\Phi})_{\lambda_E(s)}$ belong to the same representation space $\1_{n-1}\cdot|\det |^s\boxplus\1_1\cdot|\cdot|^{-(n-1)s}$ but are, in general, different.
	\end{rmk}

	\subsection{Whittaker functions and Fourier expansion}\label{sec:whittaker-function}

	Let $\phi$ be an element in $\mathcal{S}([G_n])$. Then we have a Fourier--Whittaker expansion of $\phi$, namely,
	\begin{equation}\label{Fourier-Whittaker-aut-forms}
		\phi(g)=\sum_{i=0}^{n-1} \sum_{\gamma \in \P_i(F)\backslash \P_{n-1}(F)}W_{\phi_{\Q_i}}^{\Q_i}(\gamma g),
	\end{equation}
	where \(\P_i\) and \(\Q_i\) are subgroups of $G_n$ of the form
	\begin{equation*}
		\Q_i:=\left\{\begin{pmatrix} G_i &\ast\\& G_{n-i} \end{pmatrix} \right\}, \quad \P_i:=\left\{\begin{pmatrix} G_i &\ast\\& N_{n-i} \end{pmatrix} \right\}
	\end{equation*}
	and
	\begin{equation*}
		W_{\phi_{\Q_i}}^{\Q_i}:=\int_{[N_{n-i}]}\int_{[U_{\Q_i}]}\phi\left[u\begin{pmatrix}\mathrm{I}_i&\\&n\end{pmatrix}\cdot\right]\overline{\psi}(n)\, du\, dn.
	\end{equation*}

	The sum in \eqref{Fourier-Whittaker-aut-forms} converges absolutely. A proof can be done similarly to \cite[Proposition 4.2]{IchinoYamana2016}. We abbreviate $W_{\phi_{\Q_0}}^{\Q_0}$ as $W_\phi$, the Whittaker function of $\phi$. Also, note that $W_{\phi_{\Q_{n-1}}}^{\Q_{n-1}}$ is nothing but the constant term of $\phi$ along $\Q_{n-1}$.

	Finally, it is easy to see that we have the following alternative formula for $W_{\phi_{\Q_i}}^{\Q_i}$:
	\begin{equation}\label{eq:second-formula-for-partial-whittaker}
		W_{\phi_{\Q_i}}^{\Q_i}=\int_{[\widetilde{U_i}]}\phi\left(u\cdot\right)\overline{\psi^{(i)}}(u)\, du,
	\end{equation}
	where $\widetilde{U_i}:=U_{\widetilde{\Q_i}}$ is the unipotent radical of the parabolic $\widetilde{\Q_i}$ associated to the partition $n=i+1+\dots+1$ and
	\begin{equation*}
		\psi^{(i)}(u):=\psi(u_{i+1,i+2}+\dots+u_{n-1,n}).
	\end{equation*}
	Hence, $W_{\phi_{\Q_i}}^{\Q_i}$ also satisfies a certain left $\tilde{U_i}$-equivariance, namely,
	\begin{equation}\label{eq:whittaker-equivariance}
		W_{\phi_{\Q_i}}^{\Q_i}(u) = \psi^{(i)}(u) W_{\phi_{\Q_i}}^{\Q_i}(1),\quad u\in \widetilde{U_i}(\A).
	\end{equation}

	Now we record a growth bound of the partial Whittaker functions.
	\begin{lem}\label{lem:whittaker-bound}
		Let $\phi\in\mathcal{S}([G_n])$. Then
		\begin{equation*}
			g\mapsto W_{\phi_{\Q_i}}^{\Q_i}\left[\begin{pmatrix}g&\\&h\end{pmatrix}\right]
		\end{equation*}
		is rapidly decreasing in the norm-one directions and as $|\det g|\to\infty$, and has at most polynomial growth as $|\det g|\to0$, for every $h\in G_{n-i}(\A)$, uniformly on compact sets of $h$. Moreover,
		\begin{equation*}
			W_{\phi_{\Q_i}}^{\Q_i}\left[\begin{pmatrix}g&\\&h\end{pmatrix}\right]\text{ decays rapidly along every positive root of $h$}
		\end{equation*}
		as a function of $h$ in a Siegel set, with constants of at most polynomial growth in $\|g\|$ and $|\det g|^{-1}$, depending on the decay order.
	\end{lem}

	\begin{proof}
		Recalling the definition, we write
		\begin{equation*}
			W_{\phi_{\Q_i}}^{\Q_i}\left[\begin{pmatrix}g&\\&h\end{pmatrix}\right]=\int_{[N_{n-i}]}\phi_{\Q_i}\left[\begin{pmatrix}g&\\&uh\end{pmatrix}\right]\overline{\psi}(u)\, du.
		\end{equation*}
		The first assertion follows from \cite[Lemma 2.5.13.1 and eq.~(2.4.1.4)]{BPCZ2022GGP}.

		To see the second assertion we write (\emph{cf}.\ \cite[Proof of Proposition 4.1]{JaNu2021reciprocity})
		\begin{equation*}
			W_{\phi_{\Q_i}}^{\Q_i}\left[\begin{pmatrix}g&\\&h\end{pmatrix}\right]=\int_{[N_{n-i}]}\left(\phi_{\Q_i}\left[\begin{pmatrix}g&\\&uh\end{pmatrix}\right]-\phi_{\widetilde{\Q_i}}\left[\begin{pmatrix}g&\\&uh\end{pmatrix}\right]\right)\overline{\psi}(u)\, du,
		\end{equation*}
		where $\widetilde{\Q_i}$ is the parabolic in $G_n$ associated to the partition $n=i+1+\dots+1$. Note that $\phi_{\widetilde{\Q_i}}\left[\begin{pmatrix}g&\\&h'\end{pmatrix}\right]$ is nothing but the constant term of $\phi_{{\Q_i}}\left[\begin{pmatrix}g&\\&h'\end{pmatrix}\right]$ along the minimal parabolic of $G_{n-i}$ containing $N_{n-i}$.
		Now we use \cite[Lemma I.2.10]{MW} on $\phi_{{\Q_i}}\left[\begin{pmatrix}g&\\&h'\end{pmatrix}\right]$ to conclude that the integrand, hence the integral, above decays rapidly along every positive root of $h$.
	\end{proof}

	\subsection{Newvectors}\label{sec:newvectors}

	Let $F$ be a non-archimedean local field and $\psi_0$ an unramified character of $F$. Let $\pi$ be an irreducible generic representation of $G_n(F)$, which we choose to realize in its Whittaker model with respect to the character $\psi_0$. Let
	\begin{equation*}
		L(s,\pi)=\prod_{i=1}^r\left(1-\alpha_\pi(i)N(\p)^{-s}\right)^{-1},\quad 0\le r\le n,
	\end{equation*}
	be the corresponding $L$-factor. Here $\alpha_\pi=\{\alpha_\pi(i)\}_{i=1}^r\in\C^{\times r}$ are called the Langlands parameters. For $\vartheta\ge 0$, we call $\pi$ $\vartheta$-tempered if $|\alpha_\pi(i)|\le N(\p)^\vartheta$ for $1\le i\le r$. Let $\pi_{\mathrm{ur}}$ denote the unramified principal series of $G_r(F)$ whose Satake parameters are $\{\alpha_\pi(i)\}_{i=1}^r$.

	If $\pi$ is unitary then we define a $G_n$-invariant inner product on it by
	\begin{equation}\label{eq:def-inner-prod-whittaker}
		\langle W_1,W_2\rangle:=\frac{\zeta(n)}{L(1,\pi_{\mathrm{ur}}\otimes\tilde{\pi}_{\mathrm{ur}})}\int_{N_{n-1}(F)\bs G_{n-1}(F)}W_1\cdot\overline{W_2}\left[\begin{pmatrix}h&\\&1\end{pmatrix}\right]\, dh.
	\end{equation}
	We remark here that if $F$ is archimedean and $\pi$ is a generic unitary representation of $G_n(F)$ we define the inner product on $\pi$ only by the above integral (that is without the normalization); \emph{cf.} \cite[eq.(6)]{JaNu2021reciprocity}.

	For the newvector discussion below, assume that the central character of $\pi$ is unramified, as in all our applications. Let $K_0(\p^N)$ be the Hecke congruence subgroup of $G_n(\o_v)$, consisting of matrices in $G_n(\o_v)$ whose last rows are congruent to $(0,\dots,0,\ast)\mod \p^N$. Let $c(\pi)$ be the smallest non-negative integer $N$ such that $\pi^{K_0(\p^{N})}\neq \{0\}$. Then there is a unique (up to scalars) vector $W$ in $\pi^{K_0(\p^{c(\pi)})}$. We call such a vector a \emph{newvector} of $\pi$. We also call $c(\pi)$ the \emph{conductor exponent} and $C(\pi):=N(\p)^{c(\pi)}$ the \emph{(finite) conductor} of $\pi$. We refer to \cite{JPSS1981conducteur} for details.

	If the underlying additive character $\psi$ is unramified then we denote the newvector in the $\psi$-Whittaker model of $\pi$, normalized so that $W(1)=1$, by $W_\pi$. Miyauchi \cite{Miyauchi2014Whittaker} produced, following previous work by Shintani \cite{Shintani1976explicit}, an explicit formula for $W_\pi$ given by
	\begin{equation}\label{shintani}
		W_\pi(\operatorname{diag}(\varpi^\nu,1))=
		\begin{cases}
			\delta_{B_n}^{1/2}(\operatorname{diag}(\varpi^\nu,1))\chi_\nu(\alpha_\pi),\quad & \text{if }\nu\in\mathbb{Z}^{n-1}\text{ is dominant and }\nu_{n-1}\ge0, \\
			0,\quad                                                                       & \text{otherwise},
		\end{cases}
	\end{equation}
	where $\chi_\nu$ is the Schur polynomial attached to $(\nu,0)$, evaluated at $\alpha_\pi$ padded with zeros to length $n$; see \cite{Shintani1976explicit} and \cite{Miyauchi2014Whittaker} for more details.

	Finally, if $F$ is archimedean and $\pi$ is an irreducible generic representation of $G_n(F)$, we also define the analytic conductor $C(\pi)$ of $\pi$ as in \cite{iwaniec2000perspectives}.

	\part{The Regularized Spectral Expansion}

	\section{Spectral Expansion for Large $\Re(s)$}

	We fix a Schwartz--Bruhat function $\Phi\in\mathcal{S}(\A^n)$ and for $s\in\C$ we construct an induced vector $f_{s,\Phi}$, as in \eqref{eq:def-f-from-phi}. Correspondingly,
	we construct Eisenstein series $\Eis(f_s)$ as in \eqref{eq:def-eis-from-f}. For the rest of this paper we will abbreviate $f_{0,\Phi}$ as $f$ and $f_{s,\Phi}$ as $f_s$. We also fix a $\phi\in\mathcal{S}([\bar{G}_n])$.

	\subsection{Partial unfolding}

	We start with a partial Rankin--Selberg type unfolding of the $\GL_n\times\GL_n$ zeta integral of $\phi, \Eis(f), \Eis(f_s)$.

	\begin{lem}\label{lem:partial-unfolding}
		Let $\phi$ and $\Eis(f_s)$ be as above with sufficiently large $\Re(s)$. Then we have
		\begin{equation*}
			\int_{[\bar{G}_n]}\phi\cdot\overline{\Eis(f)}\cdot\Eis(f_s)=\sum_{i=0}^{n-1}\int_{\P_i(F)\bs  G_n(\A)}W^{\Q_i}_{\phi_{\Q_i}}\overline{W^{\Q_i}_{\Eis(f)_{\Q_i}}}\Phi(e_n\cdot)|\det|^{1/2+s},
		\end{equation*}
		where the integrals on both sides converge absolutely.
	\end{lem}

	\begin{proof}
		The Schwartzness of $\phi$ immediately implies absolute convergence of the left hand side. The standard unfolding technique (see \emph{e.g.}, \cite[\S 2.3.2]{cogdell2007functions}) then leads to
		\begin{equation*}
			\int_{[\bar{G}_n]}\phi\cdot\overline{\Eis(f)}\cdot\Eis(f_s)=
			\int_{\P_{n-1}(F)\bs  G_n(\A)}\phi\cdot\overline{\Eis(f)}\Phi(e_n\cdot)|\det|^{1/2+s},
		\end{equation*}
		where $\P_i$ is as in \S\ref{sec:whittaker-function}. Now we insert the Fourier--Whittaker expansion of $\phi$, as in \eqref{Fourier-Whittaker-aut-forms}, to the above. That yields
		\begin{equation*}
			\int_{[\bar{G}_n]}\phi\cdot\overline{\Eis(f)}\cdot\Eis(f_s)=\sum_{i=0}^{n-1}
			\int_{\P_i(F)\bs  G_n(\A)}W^{\Q_i}_{\phi_{\Q_i}}\cdot\overline{\Eis(f)}\cdot
			\Phi(e_n\cdot)|\det|^{1/2+s}.
		\end{equation*}
		We first prove that the integrals on the right hand side converge absolutely (\emph{cf}.\ \cite[Proof of Lemma 4.1]{IchinoYamana2016}). To see that we parametrize $\P_i(F)\bs  G_n(\A)$ as
		\[
		u\begin{pmatrix}
			g \\&h
		\end{pmatrix}k,\qquad u\in [\tilde U_i],\,g\in [G_i],\, h\in N_{n-i}(\A)\backslash G_{n-i}(\A),\, k\in K_n.
		\]
		We note that by the first part of Lemma \ref{lem:whittaker-bound} and moderate growth of $\Eis(f)$
		together with the factor $|\det g|^{1/2+s}$, we obtain absolute convergence of the $g$-integral for sufficiently large $\Re(s)$. Similarly, by the second part of Lemma \ref{lem:whittaker-bound}, for $\Re(s)$ sufficiently large we get absolute convergence of the $h$-integral. Therefore the whole integral converges absolutely.

		We conclude after combining
		\eqref{eq:second-formula-for-partial-whittaker}, \eqref{eq:whittaker-equivariance} and the left $\widetilde{U_i}$-invariance of $\Phi$.
		\end{proof}

	\begin{lem}\label{lem:vanishing-whittaker-integral}
		The partial Whittaker function $W^{\Q_i}_{\Eis(f)_{\Q_i}}$ vanishes identically for $0\le i\le n-3$.
	\end{lem}

	\begin{proof}
		We prove the lemma first for $\Eis(f_s)$ for sufficiently large $\Re(s)$ which will yield the required statement for $s=0$ from the meromorphic continuation of $\Eis(f_s)$. Recalling the definition of $\Eis(f_s)$ from \eqref{eq:def-eis-from-f} and \eqref{eq:def-f-from-phi} we have the following absolutely convergent expression:
		\begin{equation*}
			\Eis(f_s)(g)=\sum_{\gamma\in\Q_{n-1}(F)\bs G_{n}(F)}\int_{\A^\times}\Phi(te_n\gamma g)|\det(t\gamma g)|^{1/2+s}\, d^\times t.
		\end{equation*}
		Note that $G_{n}(F)$ acts on $\mathbb{P}^{n-1}F :=F^n\setminus \{0\}/F^\times$ transitively with stabilizer of $e_n$ being $\Q_{n-1}(F)$. Thus we can realize the $\gamma$-sum above as a sum over $\mathbb{P}^{n-1}F$. On the other hand, writing $W^{\Q_i}_{\Eis(f_s)_{\Q_i}}$ as in \eqref{eq:second-formula-for-partial-whittaker} we obtain
		\begin{equation*}
			W^{\Q_i}_{\Eis(f_s)_{\Q_i}}=\int_{\A^\times}|t|^{n(1/2+s)}\int_{[\widetilde{U_i}]}\sum_{\mathbf{x}\in\mathbb{P}^{n-1}(F)}\Phi(t\mathbf{x} u)\overline{\psi^{(i)}(u)}\,du\, d^\times t.
		\end{equation*}
		It is enough to show that each orbit sum vanishes for $0\le i\le n-3$. We first group the projective points into $\widetilde{U_i}(F)$-orbits. Note that for $\mathbb{P}^{n-1}F\ni\mathbf{x}:=(x_1,\dots,x_n)$ we have
		\begin{equation*}
			\mathbf{x}u = \left(x_1,\dots,x_i,\sum_{j=1}^ix_ju_{j,i+1}+x_{i+1},\sum_{j=1}^{i+1}x_ju_{j,i+2}+x_{i+2},\dots,\sum_{j=1}^{n-1}x_ju_{j,n}+x_n\right).
		\end{equation*}
		As $\mathbf{x}\neq0$ we may assume $x_1=\dots=x_{k-1}=0$ and $x_k=1$ for some $1\le k\le n$. For $k<n$, summing over $x_n\in F$ unfolds $u_{k,n}$ from $F\bs\A$ to $\A$, with the other coordinates fixed. Let $k<n-1$ for now. After this unfolding, we change variable $u_{k,n}\mapsto u_{k,n}-\sum_{k\neq j=1}^{n-1}x_ju_{j,n}$. The argument in $\Phi(t\bullet)$ is then independent of $u_{n-1,n}$, while $\psi^{(i)}$ is independent of $u_{k,n}$. Thus the $u_{n-1,n}$-integral over $F\bs\A$ vanishes.

		Now if $k=n$ then we clearly have
		\begin{equation*}
			\int_{[\widetilde{U_i}]}\Phi(t\mathbf{x} u)\overline{\psi^{(i)}(u)}\,du=\int_{[\widetilde{U_i}]}\Phi(te_n u)\overline{\psi^{(i)}(u)}\,du=0.
		\end{equation*}
		Finally, if $k=n-1$, after unfolding $u_{n-1,n}$ the argument in $\Phi(t\bullet)$ is independent of $u_{n-2,n-1}$. For $i+1\le n-2$ the integral of $\overline{\psi_0(u_{n-2,n-1})}$ over $F\bs\A$ vanishes identically. Hence we conclude.
	\end{proof}

	Now combining Lemma \ref{lem:partial-unfolding} and Lemma
	\ref{lem:vanishing-whittaker-integral} we arrive at the following proposition.

	\begin{prop}\label{prop:whittaker-cleanup}
		Let $\phi$, $\Eis(f)$, $\Eis(f_s)$ be as above with sufficiently large $\Re(s)$. Then we have
		\begin{multline*}
			\int_{[\bar{G}_n]}\phi\cdot\overline{\Eis(f)}\cdot\Eis(f_s)=\int_{\P_{n-1}(F)\bs  G_n(\A)}\phi_{\Q_{n-1}}\overline{\Eis(f)_{\Q_{n-1}}}\Phi(e_n\cdot)|\det|^{1/2+s}\\+\int_{\P_{n-2}(F)\bs  G_n(\A)}W^{\Q_{n-2}}_{\phi_{\Q_{n-2}}}\overline{W^{\Q_{n-2}}_{\Eis(f)_{\Q_{n-2}}}}\Phi(e_n\cdot)|\det|^{1/2+s},
		\end{multline*}
		where the integrals on both sides converge absolutely.
	\end{prop}

	\subsection{The term with $i=n-1$}

	We start by looking at the first summand on the right hand side of the equation in Proposition \ref{prop:whittaker-cleanup}, namely,
	\begin{equation}\label{eq:def-I1}
		I_1(s):=\int_{\P_{n-1}(F)\bs  G_n(\A)}\phi_{\Q_{n-1}}\overline{\Eis(f)_{\Q_{n-1}}}\Phi(e_n\cdot)|\det|^{1/2+s}.
	\end{equation}
	We define
	\begin{equation}\label{eq:def-M00-M10}
		M_{00}(s):=\IP{\phi,\Eis(f\cdot \ol{f_s})}_{[\ol{G}_n]}\quad\text{and}\quad M_{10}(s):=\IP{\phi,\Eis(\M f\cdot \ol{f_s})}_{[\bar{G}_n]},
	\end{equation}
	where $\M$ is as in \eqref{eq:def-main-intertwiner}. Here we realize $\bar{f}\cdot f_s$ and $\M \bar{f}\cdot f_s$ as elements of $\1_{n-1}\cdot|\det|^{1/2+s}\boxplus\1_1\cdot|\cdot|^{-(n-1)(1/2+s)}$ and $\1_1\cdot|\cdot|^{\frac12 +s}\boxplus\1_{n-2}\cdot|\det|^s\boxplus\1_1\cdot|\cdot|^{-\frac12-(n-1)s}$, respectively. Now we record the main result of this section.

	\begin{prop}\label{prop:n-1-case}
		For $\Re(s)$ sufficiently large we have $I_1(s) = M_{00}(s)+M_{10}(s)$, where $I_1$ and $M_{00},M_{10}$ are as in \eqref{eq:def-I1} and \eqref{eq:def-M00-M10}, respectively.
	\end{prop}

	\begin{proof}
		We first evaluate $\Eis(f)_{\Q_{n-1}}$. We claim that $\{1,w_\ell\}$ is a
		complete set of representatives for $W^{\Q_{n-1}}\bs W/W^{\Q_{n-1}}$, where
		$w_\ell$ is as in \eqref{eq:def-main-intertwiner}. Indeed it follows from
		\cite[Lemma 6.7]{bernstein2024meromorphic} that
		$|W^{\Q_{n-1}}\bs W/W^{\Q_{n-1}}|=2$ and it is clear that $1$ and $w_\ell$
		are in different double classes. We check that $(\Q_{n-1})_{w_\ell}$ is the
		parabolic associated to the partition $n=1+(n-2)+1$. By definition, $f$ is left
		$N_n$-invariant, which implies that $f_{(\Q_{n-1})_{w_\ell}}=f$. Thus, from
		\eqref{eq:constant-term-expansion} we obtain that
		\begin{equation*}
			\Eis(f)_{\Q_{n-1}} = f+\Eis^{\Q_{n-1}}(\M f).
		\end{equation*}
		Plugging the above back into the expression of $I_1$  and folding the integral over $\A^\times$ we write
		\begin{equation*}
			I_1 =\int_{\Q_{n-1}(F)\bs \bar{G}_n(\A)}\phi_{\Q_{n-1}}\overline{f}f_s + \int_{\Q_{n-1}(F)\bs \bar{G}_n(\A)}\phi_{\Q_{n-1}}\Eis^{\Q_{n-1}}(\overline{\M f}) f_s.
		\end{equation*}
		The first summand above, for sufficiently large $\Re(s)$, upon folding the integral over a sum over $\Q_{n-1}(F)\bs\bar{G}_n(F)$, yields $M_{00}$. On the other hand, for sufficiently dominant $\lambda\in\a^\ast_{(\Q_{n-1})_{w_\ell},\C}$ writing
		\begin{equation*}
			\Eis^{\Q_{n-1}}\left(\overline{\M f},\lambda\right)
			=\sum_{\gamma \in (\Q_{n-1})_{w_\ell}(F)\bs \Q_{n-1}(F)}
			\left(\overline{\M f}\right)_\lambda(\gamma\cdot)
		\end{equation*}
		we similarly obtain
		\begin{align*}
			\int_{\Q_{n-1}(F)\bs \bar{G}_n(\A)}\phi_{\Q_{n-1}}
			\Eis^{\Q_{n-1}}(\overline{\M f},\lambda) f_s
			& =\int_{(\Q_{n-1})_{w_\ell}(F)\bs \bar{G}_n(\A)}\phi_{\Q_{n-1}}
			\left(\overline{\M f}\right)_\lambda f_s \\
			& =\int_{[\bar{G}_n]}\phi\,\Eis\left((\overline{\M f})_\lambda f_s\right).
		\end{align*}
		Upon invoking meromorphic continuation of both sides in $\lambda$ we specialize the above to $\lambda=0$, which yields $M_{10}$.
	\end{proof}

	\subsection{The term with $i=n-2$: preliminary clean-up}
	In this section we analyze the second term on the right hand side of the equation in Proposition \ref{prop:whittaker-cleanup}, namely,
	\begin{equation}\label{eq:def-I2}
		I_{2}(s):=\int_{\P_{n-2}(F)\bs  G_n(\A)}W^{\Q_{n-2}}_{\phi_{\Q_{n-2}}}\overline{W^{\Q_{n-2}}_{\Eis(f)_{\Q_{n-2}}}}\Phi(e_n\cdot)|\det|^{1/2+s}.
	\end{equation}
	We first look at $W^{\Q_{n-2}}_{\Eis(f)_{\Q_{n-2}}}$. Working as in the proof of Lemma \ref{lem:vanishing-whittaker-integral} for sufficiently large $\Re(s')$ we have
	\begin{align}\label{eq:whittaker-function-maximal-deg-eis}
		W^{\Q_{n-2}}_{\Eis(f_{s'})_{\Q_{n-2}}} & =\int_{\A^\times}|t|^{n(1/2+s')}\sum_{q\in F}\int_{F\bs\A}\Phi(t(0,\dots,0,1,x+q))\overline{\psi_0(x)}\, dx\, d^\times t\nonumber \\
		& =\int_{N_2(\A)}f_{s'}\left[\pmat{\mathrm{I}_{n-2}\\&w_2u}\right]\ol{\psi(u)}\,du,
	\end{align}
	where $w_2$ denotes the non-trivial Weyl element in $G_2$. The second equality above follows from unfolding the $q$-sum and $u$-integral and using the definition \eqref{eq:def-f-from-phi}. The integral above converges absolutely for sufficiently large $\Re(s')$ and has meromorphic continuation in $s'$, holomorphic near $s'=0$; see \cite[Theorem 15.4.1]{wallach1992real2}. In particular, by the left-equivariance of $f_{s'}$ and the aforementioned analytic continuation, we have
	\begin{equation*}
		W^{\Q_{n-2}}_{\Eis(f_{s'})_{\Q_{n-2}}}\left[\pmat{g&\\&\mathrm{I}_2}\cdot\right]=|\det g|^{1/2+s'}W^{\Q_{n-2}}_{\Eis(f_{s'})_{\Q_{n-2}}}.
	\end{equation*}
	A similar formula also holds for $\Phi(e_n\cdot)|\det|^{1/2+s}$.
	We parametrize $\P_{n-2}\backslash G_n(\A)$ in Iwasawa coordinates as in the proof of Lemma \ref{lem:partial-unfolding}. Noticing that, by our conventions, $\vol([U])=1$ for any unipotent group $U$, we arrive at the following result.
	\begin{lem}\label{lem:before-spectral-expansion}
		We have the following absolutely convergent expression
		\begin{multline*}
			I_2(s)=\intop_{K_n}\intop_{N_2(\A)\bs G_2(\A)}\left(\ \intop_{[G_{n-2}]}W^{\Q_{n-2}}_{\phi_{\Q_{n-2}}}\left[\begin{pmatrix}g&\\&h\end{pmatrix}k\right]|\det g|^{s-1}\, dg\ \right)\\
			\times\overline{W^{\Q_{n-2}}_{\Eis(f)_{\Q_{n-2}}}\left[\begin{pmatrix}\mathrm{I}_{n-2}&\\&h\end{pmatrix}k\right]}\Phi\left[e_n\begin{pmatrix}\mathrm{I}_{n-2}&\\&h\end{pmatrix}k\right]|\det h|^{n-3/2+s}\, dh\, dk
		\end{multline*}
		for $\Re(s)$ sufficiently large.
	\end{lem}

	In the next subsection, we do a spectral decomposition of the $[G_{n-2}]$-period above.

	\subsection{Inflation and a $\GL_2$ spectral expansion}

	First, for $s\in\C$ with $\Re(s)$ sufficiently positive and $h\in G_2(\A)$ we define
	\begin{equation}\label{eq:def-F-s}
		F_s(h):=|\det h|^{-\frac{(n-2)}{2}(s-1)}\int_{[G_{n-2}]}\phi_{\Q_{n-2}}\left[\pmat{g&\\&h}\right]|\det g|^{s-1}\, dg.
	\end{equation}
	Assuming the integral converges absolutely, notice that for all $z\in Z_2(\A)$ and $\gamma\in G_2(F)$ we have
	\begin{equation*}
		\int_{[G_{n-2}]}\phi_{\Q_{n-2}}\left[\pmat{g&\\&z\gamma h}\right]|\det g|^{s-1}\, dg = |z|^{(n-2)(s-1)}\int_{[G_{n-2}]}\phi_{\Q_{n-2}}\left[\pmat{g&\\&h}\right]|\det g|^{s-1}\, dg.
	\end{equation*}
	Thus $F_s\in C^\infty([\bar{G}_2])$. Next, we prove that \eqref{eq:def-F-s} indeed converges absolutely and simultaneously prove that $F_s(h)$ decays at the cusp of $[\bar{G}_2]$. More precisely, fix a Siegel domain of $[\bar{G}_2]$ consisting of the elements $h=\left[\pmat{r&\\& r^{-1}}\right]\Omega$ where $\Omega$ is a fixed compact set in $[\bar{G}_2]$ and $r>1$. We show that there is an absolute $d>0$ so that
	\begin{equation*}
		F_s(h)  \ll_\Omega r^{-d\,\Re(s)}.
	\end{equation*}
	It follows from \cite[Lemma 2.5.13.1 and eq.(2.4.1.4)]{BPCZ2022GGP} that there is a $c>0$ such that for every $N>0$ we have
	\begin{equation*}
		\phi_{\Q_{n-2}}\left[\pmat{g&\\&h}\right]\ll_{N} \left(|\det g|^c\cdot\max(\|g\|,\|h\|)\right)^{-N}.
	\end{equation*}
	Thus, taking $N=\frac{\Re(s)-1}{c}$, we have
	\begin{equation*}
		F_s(h)\ll \|h\|^{-N/2}\int_{[G_{n-2}]}\|g\|^{-N/2}\,dg,
	\end{equation*}
	using $\max(\|g\|,\|h\|)^{-N}\le\|g\|^{-N/2}\|h\|^{-N/2}$. For sufficiently large $\Re(s)$, \cite[Lemma 3.1]{fine-boi} bounds the integral. The same argument applies to derivatives, so $F_s$ has sufficiently rapid polynomial decay for \cite[Corollary 4.31]{fine-boi}.
	Applying the $\GL_2$ spectral decomposition from \eqref{eq:spectral-decomposition} we obtain
	\begin{multline}\label{eq:GL2-spec-decomp}
		F_s(h)=\sum_{\pi\in\Pi(\bar{G}_2)}\sum_{\varphi\in \B(\pi)} \left\langle F_s,\varphi\right\rangle_{[\bar{G}_2]} \varphi(h)\\
		+\frac{\vol([G_2]^1)}{\vol([\bar{G}_2])}\sum_{\chi\in\Pi(G_1)}\sum_{\varphi\in \B(\chi\boxplus\chi^{-1})} \int_{i\mathbb{R}}\left\langle F_s,\Eis(\varphi,z)\right\rangle_{[\bar{G}_2]} \Eis(\varphi,z)(h)\, dz,
	\end{multline}
	where $\B(\pi)$ denotes an orthonormal basis of $\pi$ in $L^2([\bar{G}_2])$. Here all pairings are taken on $[\bar{G}_2]$, which accounts for the factor in front of the continuous part: for a $Z_2(\A)$-invariant function $\Psi$ on $[G_2]$, the integrals $\int_{[G_2]^1}\Psi$ and $\int_{[\bar{G}_2]}\Psi$ are $G_2(\A)^1$-invariant functionals on the homogeneous space $[\bar{G}_2]$, hence proportional, with the constant $\vol([G_2]^1)/\vol([\bar{G}_2])$ obtained by taking $\Psi=1$.
	The above sum and integrals converge absolutely and uniformly for $h$ in a compact set.
	As a preparation, we start with the following proposition, which ``inflates'' a spectral expansion in $\sigma$ (an automorphic representation for $M$) to the same in $\mathcal{I}(\sigma)$ (an automorphic representation of $G$). The proposition works for any reductive group, so we write the statement and its proof in this general set-up with the hope that it might be useful elsewhere.

	\begin{prop}\label{prop:inflation-prop}
		Let $P$ be a standard parabolic of $G$ with Levi $M$ and let $\sigma\in\Pi^G(M)$. Let $\phi\in \mathcal{S}([\bar{G}])$ and let $\lambda\in\a_{P,\C}^\ast/\a^\ast_{G,\C}$, which we identify with the space of functionals in $\a_{P,\C}^\ast$ vanishing on $\a_G$, be sufficiently dominant. We equip $Z_G(\A)M(F)\backslash M(\A)$ and $[\bar{G}]=Z_G(\A)G(F)\backslash G(\A)$ with the quotients of the measures on $M(\A)$ and $G(\A)$ by a fixed Haar measure on $Z_G(\A)$; for $G=G_n$ these are the measures of \S\ref{subsec:chars-and-measures}. Then the following sum
		\begin{equation*}
			F_{\sigma,\lambda}(g):=\sum_{\eta\in \B(\sigma)}
			\left(\int_{Z_G(\A) M(F)\backslash M(\A)}\phi_P(mg)
			e^{\IP{\lambda-\rho_P,H_P(m)}}\overline{\eta}(m)\,dm\right)\eta,
		\end{equation*}
		converges absolutely in both $L^2$-topology and $L^\infty_{\mathrm{loc}}$-topology and we have
		\begin{equation*}
			F_{\sigma,\lambda}(g)=\sum_{\varphi\in\B\left(\mathcal{I}(\sigma)\right)}
			\IP{\phi,\Eis(\varphi, \overline{\lambda})}_{[\bar{G}]}\varphi_{-\lambda}(g).
		\end{equation*}
		Moreover, the above sums do not depend on the choice of bases.
	\end{prop}

	\begin{proof}
		We begin by proving the absolute convergence of the definition and the proposed expression of $F_{\sigma,\lambda}(g)$ which are realized as vectors in $\sigma$. The proof is similar to \cite[\S 5.1]{finisetal2011spectral}.

		Let $K_\infty$ be the chosen maximal compact of $G(F_\infty)$. We fix a Laplacian $\Delta$ on $G(F_\infty)$ by
		\begin{equation*}
			\Delta_G:=T\cdot\Id-\Omega_G+2\Omega_{K}
		\end{equation*}
		where $\Omega_G$ and $\Omega_K$ are the Casimir operators on $G(F_\infty)$ and $K_\infty$, respectively, and $T>0$ is a sufficiently large fixed real number. $\Delta_G$ is a positive definite self-adjoint operator. It is known that there is a sufficiently large $A>0$ such that $\Delta_G^{-A}$ is of trace class; see \cite[Cor 0.3]{muller1998trace} and \cite[\S 2.3.2]{MV2010subconvexity}. We will often work with a $\Delta_G$-eigenbasis of an irreducible admissible representation of $G$. Note that by construction, a $K$-type basis will automatically be a $\Delta_G$-eigenbasis.

		We fix a height $\|\cdot\|$ on $M(\A)$, as in \cite[\S1.2.2]{MW}. We start with a $\Delta_M$-eigenbasis $\B(\sigma)$. After integrating by parts the integral
		\begin{equation*}
			\int_{Z_G(\A) M(F)\backslash M(\A)}\phi_P(mg)
			e^{\IP{\lambda-\rho_P,H_P(m)}}\overline{\eta}(m)\,dm
		\end{equation*}
		with respect to $\Delta_M$ we write the above as
		\begin{equation*}
			\nu_{\eta_\lambda}^{-N}\int_{Z_G(\A) M(F)\backslash M(\A)}\Delta_M^N\left(\phi_P(\cdot g)
			e^{-\IP{\rho_P,H_P(\cdot)}}\right)(m)\overline{\eta}_\lambda(m)\,dm,
		\end{equation*}
		where $\nu_\eta$ is the $\Delta_M$-eigenvalue of $\eta$. Note that $\nu_{\eta_\lambda}\neq 0$ for sufficiently large $T$ depending on $\Re(\lambda)$. We use that $\eta\in\sigma$ are of uniform moderate growth and use the Sobolev embedding (see \cite[\S 2]{MV2010subconvexity} and \cite[\S I.2.2]{MW}) to write
		\begin{equation*}
			\eta(m) \ll \|m\|^{O(1)}\nu_\eta^{O(1)},\quad \|\eta\|_\sigma\ll \nu_\eta^{O(1)}.
		\end{equation*}
		On the other hand, it follows from \cite[Lemma 2.5.13.1 and eq.(2.4.1.4)]{BPCZ2022GGP} that there is a $c>0$ such that for every $N'>0$ we have
		\begin{equation*}
			\Delta_M^N\left(\phi_P(m g)\right)\ll_{g,N',N,\phi} \left(\delta_P^c(m)\|m\|\right)^{-N'}.
		\end{equation*}
		Consequently, \cite[Lemma 3.1]{fine-boi} implies that the integral
		\begin{equation*}
			\int_{Z_G(\A) M(F)\backslash M(\A)}\Delta_M^N\left(\phi_P(\cdot g)
			e^{-\IP{\rho_P,H_P(\cdot)}}\right)(m)\overline{\eta}_\lambda(m)\,dm
		\end{equation*}
		converges absolutely for sufficiently large $N'$ and sufficiently dominant $\lambda$ (depending on $N'$).
		Now applying the trace-class property of $\Delta_M^{-N}$ for sufficiently large $N>0$ we see that the sum defining $F_{\sigma,\lambda}$ is absolutely convergent in the $L^2$-topology and $L^\infty_{\mathrm{loc}}$-topology.

		A similar argument shows the absolute convergence of the other sum. The independence of the choice of basis is standard; see \cite[Appendix 4]{baruch2005bessel}.

		In order to establish the identity, we check that
		\begin{equation*}
			\tilde{F}_{\sigma,\lambda}(g):= e^{\IP{\lambda,H_P(g)}}F_{\sigma,\lambda}(g)
		\end{equation*}
		is an element of $\mathcal{I}(\sigma)$. Let $u_0m_0\in U(\A)M(\A) = P(\A)$ and write $mu_0m_0=(mu_0m^{-1})m_0$. Using the left-$U(\A)$ invariance of $\phi_P$ and changing variable $m\mapsto mm_0^{-1}$ we write
		\begin{equation*}
			F_{\sigma,\lambda}(u_0m_0g) = \sum_{\eta\in \B(\sigma)}
			\left(\int_{Z_G(\A) M(F)\backslash M(\A)}\phi_P(mg)
			e^{\IP{\lambda-\rho_P,H_P(mm_0^{-1})}}\overline{\eta}(mm_0^{-1})\,dm\right)\eta.
		\end{equation*}
		Changing the basis $\B(\sigma)$ we write the above as
		\begin{equation*}
			e^{\IP{\rho_P-\lambda,H_P(m_0)}}\sum_{\eta\in \B(\sigma)}
			\left(\int_{Z_G(\A) M(F)\backslash M(\A)}\phi_P(mg)
			e^{\IP{\lambda-\rho_P,H_P(m)}}\overline{\eta}(m)\,dm\right)\sigma(m_0)\eta
		\end{equation*}
		which, by linearity, can be rewritten as $e^{\IP{\rho_P-\lambda,H_P(m_0)}}\sigma(m_0)F_{\sigma,\lambda}(g)$. Finally, noting that
		\begin{equation*}
			e^{\IP{\lambda,H_P(u_0m_0g)}}=e^{\IP{\lambda,H_P(m_0)}}\cdot e^{\IP{\lambda,H_P(g)}},
		\end{equation*}
		we verify the claim.

		Now expanding over a basis of $\B\left(\mathcal{I}(\sigma)\right)$ we write
		\begin{equation*}
			F_{\sigma,\lambda}(g)={e^{\IP{-\lambda,H_P(g)}}}\sum_{\varphi\in \B\left(\mathcal{I}(\sigma)\right)} \left\langle \tilde{F}_{\sigma,\lambda},\varphi\right\rangle_{\Ind} \varphi(g)=\sum_{\varphi\in \B\left(\mathcal{I}(\sigma)\right)} \left\langle \tilde{F}_{\sigma,\lambda},\varphi\right\rangle_{\Ind} \varphi_{-\lambda}(g).
		\end{equation*}
		Thus it suffices to show that
		\begin{equation}\label{eq:coeff-equality}
			\left\langle \tilde{F}_{\sigma,\lambda},\varphi\right\rangle_{\Ind}=\IP{\phi, \Eis(\varphi,\bar{\lambda})}_{[\bar{G}]}
		\end{equation}
		for all sufficiently dominant $\lambda\in\a_{P,\C}^\ast$.

		We start with the right-hand side of \eqref{eq:coeff-equality}. First, using the definition of $\Eis(\varphi,\bar{\lambda})$ in the absolutely convergent region and unfolding we write
		\begin{equation*}
			\langle\phi, \Eis(\varphi,\bar\lambda)\rangle_{[\bar{G}]}=\int_{Z_G(\A) P(F)\backslash G(\A)}\phi(g) \bar{\varphi}_{\lambda}(g)(1)\, dg.
		\end{equation*}
		We use Iwasawa coordinates $g=umk$ with $u\in[U]$, $m\in[M]$, and $k\in K$. Using left $P(\A)$-equivariance of $\varphi_\lambda$, the integral above then evaluates to
		\begin{equation*}
			\int_{K}\int_{Z_G(\A) M(F)\bs M(\A)}\phi_P(mk)e^{\IP{\rho_P+{\lambda},H_P(m)}}\bar{\varphi}(k)(m)\, \frac{dm}{\delta_P(m)}\, dk.
		\end{equation*}
		Now we expand $\bar{\varphi}(k)\in\bar{\sigma}$ over $\B(\bar{\sigma})$, which rewrites the above as
		\begin{equation*}
			\int_{K}\int_{Z_G(\A) M(F)\bs M(\A)}\phi_P(mk)e^{\IP{{\lambda}-\rho_P,H_P(m)}}\sum_{\eta\in\B(\sigma)}\IP{\bar{\varphi}(k),\bar\eta}_{\bar\sigma}\bar\eta(m)\, dm\, dk
		\end{equation*}
		while
		\begin{equation*}
			\IP{\bar{\varphi}(k),\bar\eta}_{\bar{\sigma}}=\int_{[M^1]}\bar{\varphi}(k)(x){\eta}(x)\, dx.
		\end{equation*}
		On the other hand, using the definition of $F_{\sigma,\lambda}$ we expand the left hand side of \eqref{eq:coeff-equality} as
		\begin{equation*}
			\int_{K}\sum_{\eta\in \B(\sigma)}
			\left(\int_{Z_G(\A) M(F)\backslash M(\A)}\phi_P(mk)
			e^{\IP{\lambda-\rho_P,H_P(m)}}\overline{\eta}(m)\,dm\right)
			\int_{[M^1]}\eta(x)\bar{\varphi}(k)(x)\, dx\, dk.
		\end{equation*}
		Interchanging $x$-integral and $\eta$-sum (which can be justified as above using the rapid decay of $\IP{\varphi(k),\eta}$ in $\eta$) concludes the proof.
	\end{proof}

	\begin{rmk}\label{rmk:extra-decay-rep-lambda}
		We remark that \cite[Proposition 4.23 and Theorem 4.25]{fine-boi} show that the sum
		\begin{equation*}
			\sum_{\varphi\in\B(\mathcal{I}(\sigma))}
			\left|\IP{\phi,\Eis(\varphi, \overline{\lambda})}_{[\bar{G}]}\right| \nu_{\varphi}^{O(1)}
		\end{equation*}
		is supported on those $\sigma$ that have bounded (in terms of $\phi$) finite ramification and
		decays rapidly in $\Im(\lambda)$, on the closed positive chamber and the conductor-dependent regions of \cite[eq.~(3.13)]{fine-boi}, and in terms of the infinitesimal character (equivalently the archimedean conductor) of $\sigma$, after clearing the polar factors. The passage to the above pairings follows by integrating the weighted bounds against $\phi$ and moving powers of the invariant Laplacian onto $\phi$; see \cite[Proposition 4.27]{fine-boi}. These bounds justify the spectral sums and residue operations below, and dominated convergence to the boundary from the positive chamber.
	\end{rmk}

	\begin{rmk}\label{rmk:strong-inflation-lemma}
		We can slightly strengthen Proposition \ref{prop:inflation-prop} by weakening the assumption that $\phi\in\mathcal{S}([\bar{G}])$ to ``there exists a sufficiently large but fixed $N>0$ such that every derivative of $\phi$ decays as $\ll_N\|\cdot\|^{-N}$''. In the language of \cite[\S3.3.5]{fine-boi}, such $\phi$ lives in $\mathcal{T}_{-N}([\bar{G}])$.
	\end{rmk}

	\subsubsection{Discrete part}

	By the definition of $F_s$ as in \eqref{eq:def-F-s}, for any cusp form $\varphi\in\pi$ on $[\bar{G}_2]$ we have
	\begin{equation*}
		\left\langle F_s,\varphi\right\rangle_{[\bar{G}_2]}=\int_{[\bar{G}_2]}\int_{[G_{n-2}]}\phi_{\Q_{n-2}}\left[\pmat{g&\\&h}\right]
		\delta_{\Q_{n-2}}^{\frac{s-1}2}\left[\pmat{g&\\&h}\right]\overline{\varphi(h)}\,dg\,dh.
	\end{equation*}
	Upon identifying $M_{\Q_{n-2}}$ with $G_{n-2}\times G_2$,
	we may realize $\1_{n-2}\otimes\pi$ inside $\Pi^{G_n}\left(M_{\Q_{n-2}}\right)$, and the right-hand side above as the integral over $Z_n(\A)M_{\Q_{n-2}}(F)\backslash M_{\Q_{n-2}}(\A)$ of the integrand, which is $Z_n(\A)$-invariant. (Taking the quotient of $dg\,dh$ by $d^\times z$ along the diagonal $Z_n(\A)$ or along $Z_2(\A)$ gives the same measure, as $(z_1,z_2)\mapsto z_1z_2^{-1}$ identifies $Z_n(\A)\backslash Z_{n-2}(\A)Z_2(\A)$ with $Z_{n-2}(\A)$ preserving $d^\times z$.) Moreover, letting
	\begin{equation}\label{eq:def-lambda-discrete}
		\lambda_{\mathrm{D}}(s):=\left(s,-\tfrac{n-2}{2}s\right)\in\a^\ast_{\Q_{n-2},\C},
	\end{equation}
	we note that
	\begin{equation*}
		\delta_{\Q_{n-2}}^{\frac{s-1}{2}}(m)=\exp{\left\langle\lambda_{\mathrm{D}}(s)-\rho_{\Q_{n-2}},H_{\Q_{n-2}}(m)
			\right\rangle},\quad m\in M_{\Q_{n-2}}.
	\end{equation*}
	Thus recalling the definition of $F_{\sigma,\lambda}$ from Proposition \ref{prop:inflation-prop} we write
	\begin{equation*}
		\V^{-1}\sum_{\varphi\in \B(\pi)} \left\langle F_s,\varphi\right\rangle_{[\bar{G}_2]} \varphi(h) =  F_{\1_{n-2}\otimes\pi,\lambda_{\mathrm D}(s)}(1)(h),
	\end{equation*}
	where
	\begin{equation}\label{def-mathcal-V}
		\V:=\frac{\vol([M_{\Q_{n-2}}]^1)}{\vol([\bar{G}_2])}=\frac{\vol([G_{n-2}]^1)\vol([G_2]^1)}{\vol([\bar{G}_2])}
	\end{equation}
	is the squared norm of $\1_{n-2}\otimes\varphi$ in $L^2([M_{\Q_{n-2}}^1])$ for $\varphi\in\B(\pi)$: indeed, $\|\1_{n-2}\otimes\varphi\|^2=\vol([G_{n-2}]^1)\int_{[G_2]^1}|\varphi|^2$ and $\int_{[G_2]^1}|\varphi|^2=\vol([G_2]^1)/\vol([\bar{G}_2])$, as observed after \eqref{eq:GL2-spec-decomp}. We also record that $\vol([\bar{G}_2])=2\vol([G_2]^1)/\vol([G_1]^1)$, as one sees by integrating $\Psi(g)\eta(\log|\det g|)$, $\eta\in C_c(\R)$, over $[G_2]$ in two ways, using $dg=dg^1\,d\log|\det g|$ and $d^\times z=d^\times z^1\,d\log|z|$ together with $|\det z|=|z|^2$ for $z\in Z_2(\A)$.
	Thus invoking Proposition \ref{prop:inflation-prop} for sufficiently large $\Re(s)$ (which, in this case, is equivalent to $\lambda_{\mathrm{D}}(s)$ being sufficiently dominant) we obtain
	\begin{equation}\label{eq:disc-gl2-spectral}
		\V^{-1}\sum_{\varphi\in \B(\pi)} \left\langle F_s,\varphi\right\rangle_{[\bar{G}_2]} \varphi(h) =\sum_{\varphi\in\B\left(\1_{n-2}\boxplus\pi\right)}
		\IP{\phi,\Eis(\varphi, \ol{\lambda_{\mathrm{D}}(s)})}_{[\bar{G}_n]}\varphi_{-\lambda_{\mathrm{D}}(s)}(1)(h)
	\end{equation}
	for all $h\in G_2(\A)$.

	\subsubsection{Continuous part}
	For $\chi\in \Pi(G_1)$, we first assume that $\Re(s)$ is sufficiently positive
	and $\Re(z)$ is sufficiently negative.
	Using Proposition \ref{prop:inflation-prop}
	(also see Remark \ref{rmk:strong-inflation-lemma}) in the opposite direction,
	we obtain that
	\[
	\sum_{\varphi\in\B(\chi\boxplus\chi^{-1})}\left\langle F_s,\Eis(\varphi,-\overline{z})\right\rangle_{[\bar{G}_2]} \varphi_z(h)=\sum_{\eta\in\B(\chi\otimes\chi^{-1})}\int_{Z_2(\A) T_2(F)\backslash T_2(\A)}(F_s)_{B_2}(th)\delta^{-1/2-z}_B(t)\overline{\eta}(t)dt\,\eta(1).
	\]
	Observe that the basis $\B(\chi\otimes\chi^{-1})$ only contains one element which is, up to a volume factor, nothing but $\chi\otimes \chi^{-1}$. Recalling the definition \eqref{eq:def-F-s} of $F_s$, we see that
	\[
	(F_s)_{B_2}=|\det h|^{-\frac{(n-2)}{2}(s-1)}\int_{[G_{n-2}]}\phi_{\widetilde{\Q_{n-2}}}\left[\pmat{g\\&h}\right]|\det g|^{s-1}dg.
	\]
	Substituting in the above expression, observing that
	$\B(1_{n-2}\otimes\chi\otimes\chi^{-1})$ is also formed by one single element, whose volume factor is $\vol([G_{n-2}]^1)^{-1/2}$ times that of the element of $\B(\chi\otimes\chi^{-1})$, and arguing as in the discrete case,
	we can show that
	\begin{equation*}
		\frac{1}{\vol([G_{n-2}]^1)}\sum_{\varphi\in\B(\chi\boxplus\chi^{-1})}\left\langle F_s,\Eis(\varphi,-\overline{z})\right\rangle_{[\bar{G}_2]} \varphi_z(h)=|\det h|^{-\frac{(n-2)}{2}(s-1)}F_{\1_{n-2}\otimes\chi\otimes\chi^{-1},\lambda_{\mathrm{C}}(s,z)}(h),
	\end{equation*}
	where
	\begin{equation}\label{eq:def-lambda-cont}
		\lambda_{\mathrm{C}}(s,z):=\left(s,-\tfrac{n-2}{2}s-z,-\tfrac{n-2}{2}s+z\right)\in\a^\ast_{\widetilde{\Q_{n-2}},\C}.
	\end{equation}
	Applying Proposition \ref{prop:inflation-prop} once again, we arrive at
	\begin{multline*}
		|\det h|^{\frac{n-2}{2}(s-1)}\frac{1}{\vol([G_{n-2}]^1)}\sum_{\varphi\in \B(\chi\boxplus\chi^{-1})} \left\langle F_s,\Eis(\varphi,-\overline{z})\right\rangle_{[\bar{G}_2]}\varphi_z(h)\\
		=\sum_{\xi\in \B(\1_{n-2}\boxplus\chi\boxplus\chi^{-1})}\left\langle \phi,\Eis(\xi,\overline{\lambda_{\mathrm{C}}(s,z)})\right\rangle_{[\bar{G}_n]}\xi_{-\lambda_{\mathrm{C}}(s,z)}\left[\pmat{\mathrm{I}_{n-2}&\\&h}\right].
	\end{multline*}
	By meromorphic continuation, the above holds for all \(s\) and \(z\) in general position.

	Assume now momentarily that $\Re(z)$ is sufficiently large.
	Replacing $h$ by $\gamma h$ and summing over $\gamma\in B_2(F)\backslash G_2(F)$, we arrive at
	\begin{multline}\label{eq:cont-gl2-spectral}
		|\det h|^{\frac{n-2}{2}(s-1)}\frac{1}{\vol([G_{n-2}]^1)}\sum_{\varphi\in \B(\chi\boxplus\chi^{-1})} \left\langle F_s,\Eis(\varphi,-\overline{z})\right\rangle_{[\bar{G}_2]} \Eis(\varphi,z)(h)\\
		=\sum_{\xi\in \B(\1_{n-2}\boxplus\chi\boxplus\chi^{-1})}\left\langle \phi,\Eis(\xi,\overline{\lambda_{\mathrm{C}}(s,z)})\right\rangle_{[\bar{G}_n]}\Eis^{\Q_{n-2}}(\xi,-\lambda_{\mathrm{C}}(s,z))\left[\pmat{\mathrm{I}_{n-2}&\\&h}\right].
	\end{multline}

	Combining \eqref{eq:GL2-spec-decomp}, \eqref{eq:disc-gl2-spectral} and
  \eqref{eq:cont-gl2-spectral}, together with the fact that
	\begin{equation*}
		\varphi(1)(h)=\varphi_{-\lambda_{\mathrm{D}}(s)}(1)(h)= |\det h|^{-\frac{n-2}{2}(s-1)}\varphi_{-\lambda_{\mathrm{D}}(s)}\left[\begin{pmatrix}\mathrm{I}_{n-2}&\\&h\end{pmatrix}\right](1),
	\end{equation*}
 and invoking the meromorphic continuation of the $\Eis^{\Q_{n-2}}$ in $z$ we obtain
	\begin{multline}\label{eq:main-spec-decomp}
		|\det h|^{\frac{n-2}{2}(s-1)}\V^{-1}F_s(h)=\sum_{\pi\in\Pi(\bar{G}_2)}\sum_{\varphi\in\B\left(\1_{n-2}\boxplus\pi\right)}\IP{\phi,\Eis(\varphi, \ol{\lambda_{\mathrm{D}}(s)})}_{[\bar{G}_n]}\\
		\varphi_{-\lambda_{\mathrm{D}}(s)}\left[\begin{pmatrix}\mathrm{I}_{n-2}&\\&h\end{pmatrix}\right](1)
		+\sum_{\chi\in\Pi(G_1)}\sum_{\varphi\in\B(\1_{n-2}\boxplus\chi\boxplus\chi^{-1})}\intop_{i\mathbb{R}}\left\langle \phi,\Eis(\varphi,\overline{\lambda_{\mathrm{C}}(s,z)})\right\rangle_{[\bar{G}_n]}\\
		\Eis^{\Q_{n-2}}(\varphi,-\lambda_{\mathrm{C}}(s,z))\left[\begin{pmatrix}\mathrm{I}_{n-2}&\\&h\end{pmatrix}\right]\, dz.
	\end{multline}
	Finally, note that replacing $\phi\in \mathcal{S}([\bar{G}_n])$ by a $K_n\ni k$-translate of $\phi$ in \eqref{eq:def-F-s} and changing bases in \eqref{eq:main-spec-decomp} we arrive at the following lemma.

	\begin{lem}\label{lem:main-spec-decomp}
		Let $\phi\in \mathcal{S}([\bar G_n])$ and $k\in K_n$. Recall $\V$ from \eqref{def-mathcal-V}. For $s,z\in\C$ recall $\lambda_{\mathrm{D}}(s)$ and $\lambda_{\mathrm{C}}(s,z)$ from \eqref{eq:def-lambda-discrete} and \eqref{eq:def-lambda-cont}. Then for $h\in G_2(\A)$ we have
		\begin{multline*}
			\V^{-1}\int_{[G_{n-2}]}\phi_{\Q_{n-2}}\left[\pmat{g&\\&h}k\right]|\det g|^{s-1}\, dg=\sum_{\pi\in\Pi(\bar{G}_2)}\sum_{\varphi\in\B\left(\1_{n-2}\boxplus\pi\right)}\IP{\phi,\Eis(\varphi, \ol{\lambda_{\mathrm{D}}(s)})}_{[\bar{G}_n]}\\
			\varphi_{-\lambda_{\mathrm{D}}(s)}\left[\begin{pmatrix}\mathrm{I}_{n-2}&\\&h\end{pmatrix}k\right](1)+\sum_{\chi\in\Pi(G_1)}\sum_{\varphi\in\B(\1_{n-2}\boxplus\chi\boxplus\chi^{-1})}\intop_{i\mathbb{R}}\left\langle \phi,\Eis(\varphi,\overline{\lambda_{\mathrm{C}}(s,z)})\right\rangle_{[\bar{G}_n]}\\
			\Eis^{\Q_{n-2}}(\varphi,-\lambda_{\mathrm{C}}(s,z))\left[\begin{pmatrix}\mathrm{I}_{n-2}&\\&h\end{pmatrix}k\right]\, dz
		\end{multline*}
		if $\Re(s)$ is sufficiently large.
	\end{lem}

	\subsubsection{Partial Whittaker coefficient}

	Now we start with the equation in Lemma \ref{lem:main-spec-decomp}, replace $h$ by $uh$ where $u\in [N_2]$ and integrate both sides against the additive character $\overline{\psi(u)}$. The left hand side, by definition, yields
	\begin{equation*}
		\V^{-1}\int_{[G _{n-2}]}W^{\Q_{n-2}}_{\phi_{\Q_{n-2}}}\left[\pmat{g&\\&h}k\right] |\det g|^{s-1} dg,
	\end{equation*}
	which is the inner-most $[G_{n-2}]$-integral in Lemma \ref{lem:before-spectral-expansion}.
	On the other hand, the right-hand side, after interchanging the $u$-integral with the existing sums and integrals there (which is justified as the $u$-integral is compact and these sums and integrals converge absolutely), yields
	\begin{multline*}
		\sum_{\pi\in\Pi(\bar{G}_2)}\sum_{\varphi\in\B\left(\1_{n-2}\boxplus\pi\right)}\IP{\phi,\Eis(\varphi, \ol{\lambda_{\mathrm{D}}(s)})}_{[\bar{G}_n]}
		W^{\Q_{n-2}}_{\varphi_{-\lambda_{\mathrm{D}}(s)}}\left[\begin{pmatrix}\mathrm{I}_{n-2}&\\&h\end{pmatrix}k\right](1)\\
		+\sum_{\chi\in\Pi(G_1)}\sum_{\varphi\in\B(\1_{n-2}\boxplus\chi\boxplus\chi^{-1})}\intop_{i\mathbb{R}}\left\langle \phi,\Eis(\varphi,\overline{\lambda_{\mathrm{C}}(s,z)})\right\rangle_{[\bar{G}_n]}
		W^{\Q_{n-2}}_{\Eis^{\Q_{n-2}}(\varphi,-\lambda_{\mathrm{C}}(s,z))}\left[\begin{pmatrix}\mathrm{I}_{n-2}&\\&h\end{pmatrix}k\right]\, dz.
	\end{multline*}
	Moreover, we note that
	\begin{equation*}
		W^{\Q_{n-2}}_{\varphi_{-\lambda_{\mathrm{D}}(s)}}\left[\begin{pmatrix}\mathrm{I}_{n-2}&\\&h\end{pmatrix}k\right](1)=e^{\left\langle\rho_{\Q_{n-2}}-\lambda_{\mathrm D}(s),H_{\Q_{n-2}}(h)\right\rangle}\int_{[N_2]}\varphi(k)(uh)\overline{\psi(u)}\, du,
	\end{equation*}
	which vanishes identically if $\pi\ni\varphi(k)$ is non-generic. Thus we can restrict the above $\pi$-sum to the cuspidal spectrum $\Pi_{\mathrm{c}}(\bar{G}_2)$.

	Before stating the main result of this section, we define some useful expressions.

	\begin{rmk}\label{rmk:short-name-for-partial-Whittaker}
		Note that if $\varphi\in\1_{n-2}\boxplus\pi$ (or any vector that is induced from a parabolic associated to $n = (n-2)+2$) then, by the definition of $\Eis^{\Q_{n-2}}(\varphi,\lambda)$ as in \S\ref{sec:auto-rep}, it is the same as $\varphi_\lambda(\cdot)(1)$ for any $\lambda\in\a^\ast_{\Q_{n-2},\C}$. Thus we can realize $W^{\Q_{n-2}}_{\varphi_{-\lambda_{\mathrm{D}}(s)}}$ as $W^{\Q_{n-2}}_{\Eis^{\Q_{n-2}}(\varphi,-\lambda_{\mathrm{D}}(s))}$. Finally, to ease the notation we abbreviate $W^{\Q_{n-2}}_{\Eis^{\Q_{n-2}}(\varphi,\lambda)}$ as $\W(\varphi_\lambda)$.
	\end{rmk}

	For a vector $\varphi\in\1_{n-2}\boxplus\pi$ or
	$\varphi\in\1_{n-2}\boxplus\chi\boxplus\chi^{-1}$ or
	$\varphi \in \1_{n-1}\boxplus\1_1$,
	correspondingly, for $\lambda=-\lambda_\mathrm{D}(s)\in\a^\ast_{\Q_{n-2},\C}$,
	or $\lambda=-\lambda_\mathrm{C}(s,z)\in\a^\ast_{\widetilde{\Q_{n-2}},\C}$
	or $\lambda=-\lambda_\mathrm{E}(s-\tfrac 12)\in\a^\ast_{\Q_{n-1},\C}$,
	and $s,z\in\C$ we define our main period integral as
	\begin{equation}\label{eq:main-RS-period}
			\P(\lambda;\varphi,f,f_s):=\intop_{K_n}\intop_{Z_2(\A)N_2(\A)\bs G_2(\A)}
			\W(\varphi_\lambda)\cdot\overline{\W(f)}\cdot f_s
			\left[\begin{pmatrix}\mathrm{I}_{n-2}&\\&h\end{pmatrix}k\right]|\det h|^{n-2}\, dh\, dk,
	\end{equation}
	initially defined for $s$ such that \(\Re(s)\) is sufficiently large (in terms of $z$ when applicable).
	The first two cases already appear in the spectral expansion below. The third
	will be important for us later when we consider some residues of Eisenstein
	series.

	In the next section, in the proof of Lemma \ref{lem:factorization-period} we will show that the above integral converges absolutely for $\Re(s)$ sufficiently large.
	We also define
	\begin{equation}\label{eq:def-I2-disc}
		I_2^{\mathrm{D}}(s):=\V\sum_{\pi\in\Pi_{\mathrm{c}}(\bar{G}_2)}
		\sum_{\varphi\in\B\left(\1_{n-2}\boxplus\pi\right)}
		\left\langle\phi,\Eis(\varphi, \ol{\lambda_{\mathrm{D}}(s)})\right\rangle_{[\bar{G}_n]}
		\P(-\lambda_{\mathrm{D}}(s);\varphi,f,f_s)
		,
	\end{equation}
	for $\lambda_{\mathrm{D}}$ as in \eqref{eq:def-lambda-discrete} and
	\begin{equation}\label{eq:def-I2-cont}
		I_2^{\mathrm{C}}(s):=\V\sum_{\chi\in\Pi(G_1)}\sum_{\varphi\in\B(\1_{n-2}\boxplus\chi\boxplus\chi^{-1})}
		\intop_{i\mathbb{R}}\left\langle \phi,\Eis(\varphi,\overline{\lambda_{\mathrm{C}}(s,z)})\right\rangle_{[\bar{G}_n]}
		\P(-\lambda_{\mathrm{C}}(s,z);\varphi,f,f_s)
		\, dz,
	\end{equation}
	for $\lambda_{\mathrm{C}}$ as in \eqref{eq:def-lambda-cont}.

	Finally, combining the above definitions and applying Lemma \ref{lem:before-spectral-expansion} and Lemma \ref{lem:main-spec-decomp} we arrive at the following.

	\begin{prop}\label{prop:expansion-after-Whittaker}
		For $\Re(s)$ sufficiently large we have
		\begin{equation*}
			I_2(s) = I_2^{\mathrm{D}}(s) + I_2^{\mathrm{C}}(s)
		\end{equation*}
		where $I_2,I_2^{\mathrm{D}},I_2^{\mathrm{C}}$ are as in \eqref{eq:def-I2}, \eqref{eq:def-I2-disc}, \eqref{eq:def-I2-cont}, respectively.
	\end{prop}

	\section{Factorization of the Period and Unramified Computation}

	The main goal in this section is to show that the period $\P(\lambda;\varphi,f,f_s)$ from \eqref{eq:main-RS-period} factorizes as a product of local periods and to compute the local periods at the unramified places.

	\subsection{Factorization}

	We start by recalling Flath's theorem \cite{flath1979decomposition}. Let $M$ be a Levi of a proper parabolic $P$ of $G_n$ associated to the partition $n=n_1+\dots+n_k$ and $\sigma=\sigma_1\otimes\dots\otimes\sigma_k\in\Pi(M)$ where $\sigma_i$ are discrete series representations of $G_{n_i}(\A)$. On the other hand, for a place $v$ of $F$ we define the local representation $\Sigma_v:=\mathcal{I}\left(\otimes_{i=1}^k\sigma_{i,v}\right)$ where $\sigma_{i,v}$ are the $v$-adic factors of $\sigma_i$ under Flath's theorem. Then it follows that there is a $G_n(\A)$-equivariant isomorphism $\boxplus_{i=1}^k\sigma_i\cong \otimes'_v \Sigma_v$. We call a vector $\varphi\in\sigma_1\boxplus\dots\boxplus\sigma_k$ \emph{factorizable} if the image of $\varphi$ under the above isomorphism is a pure tensor $\otimes_v'\varphi_v$. In this case, under the above isomorphism $g\cdot\varphi\mapsto\otimes_v g_v\cdot\varphi_v$ for $g\in G_n(\A)$.

	Let $P=\Q_{n-2}$ and $\sigma=\1_{n-2}\otimes\pi$ where $\pi$ is a cuspidal representation of $G_2(\A)$. Recalling the notations from Remark \ref{rmk:short-name-for-partial-Whittaker} for $\varphi\in\mathcal{I}(\sigma)$ and $\lambda\in\a^\ast_{P,\C}$ we write
	\begin{equation*}
		\W(\varphi_\lambda)(1):=\int_{[N_2]}\varphi_\lambda\left[\pmat{\mathrm{I}_{n-2}\\&u}\right](1)\overline{\psi(u)}\, du = W_{\varphi(1)}(1)
	\end{equation*}
	where $W_{\varphi(1)}$ denotes the $\psi$-Whittaker function of $\varphi(1)\in\pi$. Fix a finite set of places $S$ containing all the archimedean and ramified places of $F$. If $\varphi$ is factorizable then following the above definition and using uniqueness of the Whittaker model we may write
	\begin{equation*}
		W_{\varphi(1)}=\frac{1}{\sqrt{L^S(1,\pi,\Ad)}}\prod_v\varphi_v(1)
	\end{equation*}
	where $\varphi_v(1)$ is an element in the $\psi_v$-Whittaker model of $\pi_v$ such that $\varphi_v(1)$ is the normalized spherical vector for all $v\notin S$. Correspondingly, we have
	\begin{equation}\label{eq:def-local-W-phi-cusp}
		\W(\varphi_\lambda)(g)=\frac{1}{\sqrt{L^S(1,\pi,\Ad)}}\prod_v\W(\varphi_{\lambda,v})(g_v),\quad \W(\varphi_{\lambda,v})(g_v):= e^{\left\langle \lambda,H_P(g_v)\right\rangle}\varphi_v(g_v)(1),
	\end{equation}
	for $G_n(\A)\ni g= (g_v)_v$.

	Now let $P=\widetilde{\Q_{n-2}}$ and $\sigma=\1_{n-2}\otimes\chi\otimes\chi^{-1}$ where $\chi\in\Pi(G_1)$. In this case, for $\varphi\in\mathcal{I}(\sigma)$ and $\lambda\in\a^\ast_{P,\C}$, we have
	\begin{align}\label{eq:bruhat-gl2}
		\W(\varphi_\lambda)
		& =\int_{[N_2]}\Eis^{\Q_{n-2}}(\varphi,\lambda)\left[\pmat{\mathrm{I}_{n-2}\\&u}\right]\overline{\psi(u)}\, du \nonumber \\
		& = \int_{N_2(\A)}\varphi_\lambda\left[\pmat{\mathrm{I}_{n-2}\\&w_2u}\right]\overline{\psi(u)}\, du
	\end{align}
	where the right hand side above converges for sufficiently dominant $\lambda$, and has analytic continuation for all $\lambda$. To see the equality \eqref{eq:bruhat-gl2} we first write $\Eis^{\Q_{n-2}}(\varphi,\lambda)(g)$ as
	\begin{equation*}
		\sum_{\gamma\in B(F)\bs G_2(F)}\varphi_\lambda\left[\pmat{\mathrm{I}_{n-2}\\&\gamma}g\right].
	\end{equation*}
	Replacing $g$ by $\pmat{\mathrm{I}_{n-2}\\&u}$ with $u\in [N_2]$, using the Bruhat decomposition for $B(F)\bs G_2(F)$, and using left $N_2(\A)$-invariance we deduce \eqref{eq:bruhat-gl2}. Now if $\varphi$ is factorizable then the above discussion allows us to write
	\begin{equation}\label{eq:def-local-W-phi-eis}
		\W(\varphi_\lambda)(g)
		=\prod_v\W(\varphi_{\lambda,v})(g_v),\quad
    \W(\varphi_{\lambda,v}):=
    \int_{N_2(F_v)}\varphi_{\lambda,v}\left[\pmat{\mathrm{I}_{n-2}\\&w_2u}\cdot\right]
    \overline{\psi_v(u)}\, du,
	\end{equation}
	where the right hand side of the above infinite product converges absolutely
  for sufficiently dominant $\lambda$. In fact, it suffices to require that
  $\Re(\alpha(\lambda))$ is sufficiently large where $\alpha$ is the simple
  root of the lower $G_2$ block in $\Q_{n-2}$.

	Finally, let $P=\Q_{n-1}$ and $\varphi\in \1_{n-1}\boxplus \1_1$. We put
	\begin{equation*}
		\W(f_{s'}):=W_{\Eis(f_{s'})_{\Q_{n-2}}}^{\Q_{n-2}}.
	\end{equation*}
	It follows from \eqref{eq:whittaker-function-maximal-deg-eis} that
	\[\W(f_{s'})=\int_{N_2(\A)}f_{s'}\left[\begin{pmatrix}
		\mathrm{I}_{n-2} \\&w_2u
	\end{pmatrix}\right]\overline{\psi(u)}du.\]
	Thus, if $f$ is factorizable, then
	\begin{equation}\label{eq:def-partial-jacquet-intertwiner-local}
		\W(f_{s'})=\prod_{v}\W(f_{v,s'}),\quad\W(f_{v,s'}):=\int_{N_2(F_v)}f_{v,s'}\left[\pmat{\mathrm{I}_{n-2}\\&w_2u}\right]\overline{\psi_v(u)}\, du,
	\end{equation}
	where the product and the integral converge absolutely for sufficiently large $\Re(s')$ and are understood in general by analytic continuation. We also use the shorthands $\W(f)=\W(f_0)$ and $\W(f_v)=\W(f_{v,0})$.

	We now define a local Rankin--Selberg period as follows:
	\begin{multline}\label{eq:main-RS-period-local}
			\P_v(\lambda;\varphi_v,f_v,f_{v,s}):=\intop_{K_{n,v}}\intop_{Z_2(F_v)N_2(F_v)\bs G_2(F_v)}
			\W(\varphi_{\lambda,v})\cdot\overline{\W(f_v)}\cdot
			f_{v,s}\left[\begin{pmatrix}\mathrm{I}_{n-2}&\\&h\end{pmatrix}k\right]\\
			|\det h|^{n-2}\, dh\, dk.
	\end{multline}

	Here, the conventions are the local counterparts of those used in the
	definition of the global period \eqref{eq:main-RS-period}.
	More precisely, we assume that
	$\vphi$ belongs to $\1_{n-2}\boxplus\pi_v$ or
	$\1_{n-2}\boxplus\chi_v\boxplus\chi_v^{-1}$ or
	$\1_{n-1}\boxplus\1_1$ and $\lambda$ is either $-\lambda_{\mathrm{D}}(s)$ or
	$-\lambda_{\mathrm{C}}(s,z)$ or $-\lambda_\mathrm{E}(s-\tfrac 12)$, respectively.
	The local period converges absolutely for sufficiently large $\Re(s)$ (in terms
	of $z$, in the second case).

	\begin{lem}\label{lem:factorization-period}
		Let $\Phi=\otimes'_v\Phi_v$ be a factorizable Schwartz--Bruhat function.
		Let the vector $\varphi$ be a factorizable vector (in the above sense)
		in $\1_{n-2}\boxplus\pi$, or $\1_{n-2}\boxplus\chi\boxplus\chi^{-1}$,
		or $\1_{n-1}\boxplus\1_{1}$,
		respectively, and let $\lambda$ be $-\lambda_{\mathrm{D}}(s)$,
		$-\lambda_{\mathrm{C}}(s,z)$ or $-\lambda_\mathrm{E}(s-\tfrac 12)$, respectively.
		Then, for sufficiently large $\Re(s)$ (in terms of $z$ in the second case),
		we have that both the local period in \eqref{eq:main-RS-period-local} and
		the global period in \eqref{eq:main-RS-period} converge absolutely. Also, in the first case
		we have the factorization
		\begin{equation*}
				\P(\lambda;\varphi,f,f_s)=\frac{1}{\sqrt{L^S(1,\pi,\Ad)}}\prod_v\P_v(\lambda;\varphi_v,f_v,f_{v,s}).
		\end{equation*}
		In the other two cases, the factor $L^S(1,\pi,\Ad)^{-1/2}$ is omitted.

	\end{lem}
	\begin{proof}
		We first show the integral \eqref{eq:main-RS-period-local} converges absolutely.
		In fact, by unfolding, it suffices to prove it for the integral
			$$\intop_{K_{n,v}}\intop_{N_2(F_v)\bs G_2(F_v)}\W(\varphi_{\lambda,v})\cdot\overline{\W(f_v)}\cdot\Phi_v(e_n\bullet)\left[\begin{pmatrix}\mathrm{I}_{n-2}&\\&h\end{pmatrix}k\right]\\
			|\det h|^{n-3/2+s}\, dh\, dk.$$
		Note that
		$\W(\varphi_{\lambda,v})\left[\begin{pmatrix}\mathrm{I}_{n-2}&\\&\bullet\end{pmatrix}k\right]$
		is a $\GL_2$-Whittaker function belonging to the representation with $L$-function parameters $\beta,\gamma$, which are certain affine complex-valued functions of
		$\lambda$. It follows from
		\cite[Proof of Proposition 3.2.3]{MV2010subconvexity} that
		\begin{equation*}
			\W(\varphi_{\lambda,v})\left[\begin{pmatrix}\mathrm{I}_{n-2}&\\&h\end{pmatrix}k\right]
			\ll_{\varphi_v,\lambda} |\det h|_v^{\frac{\Re(\beta+\gamma)}{2}}
			\delta_{B_2}(h)^{\frac{1-|\Re(\beta-\gamma)|}{2}}
			\min\left(1,\delta_{B_2}(h)\right)^{-N},
		\end{equation*}
		uniformly in $k\in K_{n,v}$.
		Similarly, $\W(f_v)\left[\begin{pmatrix}\mathrm{I}_{n-2}&\\&\bullet\end{pmatrix}k\right]$
		is in the Whittaker model of the $\GL_2$-principal series
		$\1_{1,v}\boxplus|\cdot|_v^{-\frac{n-2}{2}}$, and
		$\Phi_v\left[e_n\begin{pmatrix} \mathrm{I}_{n-2}&\\&\bullet\end{pmatrix}k\right]
		\in\mathcal{S}(F_v^2)$
		and satisfies a similar estimate. Thus from the classical theory of local
		$\GL_2\times\GL_2$ Rankin--Selberg zeta integrals (see \emph{e.g.},
		\cite[\S 2.3, Theorem 2.2]{cogdell2007functions}) we see that the integral
		in consideration converges absolutely for
		\begin{equation*}
			\Re(s)\ge -\min\left(\Re(\beta),\Re(\gamma)\right)-\tfrac{n-1}{2}
		\end{equation*}
		and has a meromorphic continuation to all $\C$.
		Moreover, the same also implies the absolute convergence for large
		$\Re(s)$ and meromorphic continuation to all $\C$
		of the global period $\P(\lambda;\varphi,f,f_s)$.
	\end{proof}
	\subsection{The unramified computation}\label{sec:unramified-local-choice}

	In what follows let $v$ be a finite place of $F$ that is coprime to the different ideal of $F$. We also assume that $\pi$ and $\chi$ are unramified at $v$. Let us choose $\Phi$ so that $\Phi_v:=\mathbbm{1}_{\o_v^n}$. Thus we compute
	\begin{equation}\label{eq:schwartz-restriction}
		\Phi_v\left[e_n\begin{pmatrix}\mathrm{I}_{n-2}&\\&\bullet\end{pmatrix}k\right] = \mathbbm{1}_{\o_v^2}(e_2\cdot)
	\end{equation}
	as $k\in K_{n,v}$. Consequently, from the relation of $f_{v,s'}$ and $\Phi_v$ as in \S\ref{sec:Schwartz-space} and the definition of $\W(f_{v,s'})$ (see \eqref{eq:def-partial-jacquet-intertwiner-local}) we deduce that $\W(f_v)\left[\begin{pmatrix}\mathrm{I}_{n-2}&\\&\bullet\end{pmatrix}\right]$ is a spherical vector in the $\psi_v$-Whittaker model of $\1_{1,v}\boxplus |\cdot|_v^{-\frac{n-2}{2}}$.
	Once again from the relation of $f_{v,s'}$ and $\Phi_v$ as in \S\ref{sec:Schwartz-space}, we compute
	\begin{align*}
		\W(f_{v,s'})(1)
		& = \int_{F_v^\times}|t|_v^{n(1/2+s')}\int_{F_v}\mathbbm{1}_{\o_v^2}(t(1,x))\ol{\psi_v(x)}\, dx\, d^\times t \\
		& =\int_{\o_v}|t|_v^{n(1/2+s')-1}\int_{\o_v}\ol{\psi_v(x/t)}\,dx\, d^\times t=1
	\end{align*}
	for $\Re(s')$ sufficiently large. Hence, by analytic continuation in $s'$ we conclude that
	\begin{equation}\label{eq:induced-schwartz-normalization}
		\W(f_v)\left[\begin{pmatrix}\mathrm{I}_{n-2}&\\&\bullet\end{pmatrix}k\right] = W_{\tau_v},\quad \tau_v=\1_{1}\boxplus|\cdot|_v^{-\frac{n-2}{2}}.
	\end{equation}
	This is the normalized spherical vector in the Whittaker model of $\tau_v$. In what follows, we choose the vectors $\varphi$ at $v$ and compute the local period $\P_v$ as defined in \eqref{eq:main-RS-period-local}.

	\subsubsection{Cuspidal case}\label{sec:cuspidal-local-choice}

	Let $\varphi\in\1_{n-2}\boxplus\pi$ be such that $\varphi_v$ is the spherical vector with the normalization $\W(\varphi_{0,v}) = 1$. Recall $\lambda_{\mathrm{D}}(s)$ from \eqref{eq:def-lambda-discrete}. By sphericality we have
	\begin{equation*}
		\W(\varphi_{-\lambda_{\mathrm{D}}(s),v})\left[\begin{pmatrix}\mathrm{I}_{n-2}&\\&hk'\end{pmatrix}k\right] = |\det h|_v^{\frac{n-2}{2}s}\W(\varphi_{0,v})\left[\begin{pmatrix}\mathrm{I}_{n-2}&\\&h\end{pmatrix}\right],\quad k'\in K_{2,v}.
	\end{equation*}
	Thus by the above normalization we obtain
	\begin{equation}\label{eq:cuspidal-local-choice}
		\W\left(\varphi_{-\lambda_{\mathrm{D}}(s),v}\right)\left[\begin{pmatrix}\mathrm{I}_{n-2}&\\&\bullet\end{pmatrix}k\right]
		= W_{\pi_v}\cdot|\det|_v^{\frac{n-2}{2}(s-1)}.
	\end{equation}
	Hence, combining \eqref{eq:cuspidal-local-choice}, \eqref{eq:induced-schwartz-normalization}, and \eqref{eq:schwartz-restriction} and recalling \eqref{eq:main-RS-period-local} we compute
	\begin{align}\label{eq:unramified-cusp}
		\P_v\left(-\lambda_{\mathrm{D}}(s);\varphi_v,f_v,f_{v,s}\right)
		& =L_v\left(\tfrac{n-1}{2}+\tfrac{n}{2}s,\tau_v\,\otimes\,\pi_v\right)                                          \\
		& =L_v\left(\tfrac{n-1}{2}+\tfrac{n}{2}s,\pi_v\right)L_v\left(\tfrac{1}{2}+\tfrac{n}{2}s,\pi_v\right),\nonumber
	\end{align}
	which follows from \cite[Theorem 3.3]{cogdell2007functions} for sufficiently large $\Re(s)$.

	\subsubsection{Eisenstein case}\label{sec:eisenstein-local-choice}

	Let $\varphi\in\1_{n-2}\boxplus\chi\boxplus\chi^{-1}$ be such that $\varphi_v$ is the spherical vector with the normalization $\varphi_v(1) = 1$. Recall $\lambda_{\mathrm{C}}(s,z)$ from \eqref{eq:def-lambda-cont}. By sphericality we have
	\begin{equation*}
		\W(\varphi_{-\lambda_{\mathrm{C}}(s,z),v})\left[\begin{pmatrix}\mathrm{I}_{n-2}&\\&hk'\end{pmatrix}k\right]
		= e^{\left\langle (z,-z), H_B( h)\right\rangle}|\det h|_v^{\frac{n-2}{2}s}\W(\varphi_{0,v})\left[\begin{pmatrix}\mathrm{I}_{n-2}&\\&h\end{pmatrix}\right],\quad k'\in K_{2,v}.
	\end{equation*}
	On the other hand, a similar computation as in the deduction of \cite[eq.(6.11)]{bump19927automorphic} yields
	\begin{equation*}
		\W(\varphi_{-\lambda_{\mathrm{C}}(s,z),v})(1) = L_v(1+2z,\chi_v^2)^{-1}
	\end{equation*}
	for $\Re(z)$ sufficiently large. Thus by the above normalization we obtain
	\begin{equation}\label{eq:eisenstein-local-choice}
		\W(\varphi_{-\lambda_{\mathrm{C}}(s,z),v})\left[\begin{pmatrix}\mathrm{I}_{n-2}&\\&\bullet\end{pmatrix}k\right]
		= L_v(1+2z,\chi_v^2)^{-1}W_{\left(\chi_v\cdot|\cdot|_v^{z}\,\boxplus\,\chi_v^{-1}\cdot|\cdot|_v^{-z}\right)}\cdot
		|\det|^{\frac{n-2}{2}(s-1)}.
	\end{equation}
	Hence, combining \eqref{eq:eisenstein-local-choice}, \eqref{eq:induced-schwartz-normalization}, and \eqref{eq:schwartz-restriction} and recalling \eqref{eq:main-RS-period-local} we compute
	\begin{align}\label{eq:unramified-eis}
		\P_v\left(-\lambda_{\mathrm{C}}(s,z);\varphi_v,f_v,f_{v,s}\right)
		& =\frac{L_v\left(\tfrac{n-1}{2}+\tfrac{n}{2}s,\tau_v\,\otimes\,
			\left(\chi_v\cdot|\cdot|_v^{z}\,\boxplus\,\chi_v^{-1}\cdot|\cdot|_v^{-z}\right)\right)}{L_v(1+2z,\chi_v^2)}\nonumber \\
		& =\frac{\prod_{\pm}L_v\left(\tfrac{n-1}{2}+\tfrac{n}{2}s\pm z,\chi^\pm_v\right)
			L_v\left(\tfrac{1}{2}+\tfrac{n}{2}s\pm z,\chi^\pm_v\right)}{L_v(1+2z,\chi_v^2)},
	\end{align}
	which follows from \cite[Theorem 3.3]{cogdell2007functions} for sufficiently large $\Re(s)$.

	Now we state the main result of this section.

	\begin{prop}\label{prop:before-meromorphic-continuation}
		Let $S$ be a finite set of places of $F$ containing all the archimedean and ramified places of $F$ such that $\phi\in\mathcal{S}([\bar{G}_n])$ is right invariant by $K_{n,v}$ for each $v\notin S$. Moreover, we choose a $\Phi\in\mathcal{S}(\A^n)$ such that $\Phi_v=\mathbbm{1}_{\o_v^n}$ for all $v\notin S$, as in \S\ref{sec:unramified-local-choice}. Then, for sufficiently large $\Re(s)$ we have
		\begin{multline*}
			I_2^{\mathrm{D}}(s)=\V\sum_{\pi\in\Pi_{\mathrm{c}}(\bar{G}_2)}\sum_{\varphi\in\B\left(\1_{n-2}\boxplus\pi\right)}\left\langle\phi,\Eis(\varphi, \ol{\lambda_{\mathrm{D}}(s)})\right\rangle_{[\bar{G}_n]}\\
			\frac{L^S\left(\tfrac{n-1}{2}+\tfrac{n}{2}s,\pi\right)L^S\left(\tfrac{1}{2}+\tfrac{n}{2}s,\pi\right)}{\sqrt{L^S(1,\pi,\Ad)}}
			\P_S(-\lambda_{\mathrm{D}}(s);\varphi,f,f_s)
			,
		\end{multline*}
		and
		\begin{multline*}
			I_2^{\mathrm{C}}(s)=\V\sum_{\chi\in\Pi(G_1)}\sum_{\varphi\in\B(\1_{n-2}\boxplus\chi\boxplus\chi^{-1})}\intop_{i\mathbb{R}}\left\langle \phi,\Eis(\varphi,\overline{\lambda_{\mathrm{C}}(s,z)})\right\rangle_{[\bar{G}_n]}\\
			\frac{\prod_{\pm}L^S\left(\tfrac{n-1}{2}+\tfrac{n}{2}s\pm z,\chi^\pm\right)L^S\left(\tfrac{1}{2}+\tfrac{n}{2}s\pm z,\chi^\pm\right)}{L^S(1+2z,\chi^2)}
			\P_S(-\lambda_{\mathrm{C}}(s,z);\varphi,f,f_s)
			\, dz,
		\end{multline*}
		where $I_2^{\mathrm{D}}$ and $I_2^{\mathrm{C}}$ are as in \eqref{eq:def-I2-disc} and \eqref{eq:def-I2-cont}, respectively;
		and $\P_S:=\prod_{v\in S}\P_v$ where $\P_v$ is as in \eqref{eq:main-RS-period-local}.
	\end{prop}

	\begin{proof}
		Since $\phi$ is right invariant under $K_{n,v}$ for $v\notin S$, it implies that $\1_{n-2}\boxplus\pi$ and $\1_{n-2}\boxplus\chi\boxplus\chi^{-1}$  are unramified at $v$ and hence so are $\pi$ and $\chi$. We choose $\varphi\in\B$ in either case so that $\varphi_v$ are normalized as in \S\ref{sec:cuspidal-local-choice} and \S\ref{sec:eisenstein-local-choice}. The proof now follows by combining Lemma \ref{lem:factorization-period}, \eqref{eq:unramified-cusp}, \eqref{eq:unramified-eis}.
	\end{proof}

	\section{Meromorphic Continuation and Residues}\label{sec:meromorphic-continuation}

	The primary goal of this section is to meromorphically continue the term $I_2(s)$,
	which is defined for $\Re(s)$ sufficiently large in \eqref{eq:def-I2}. Our
	approach is via Proposition \ref{prop:expansion-after-Whittaker}. Namely,
	we will consider $I_2^{\mathrm{D}}(s)$ and $I_2^{\mathrm{C}}(s)$ separately,
	using their respective expressions given in Proposition
	\ref{prop:before-meromorphic-continuation}.

	\subsection{Meromorphic continuation of $I_2^{\mathrm{D}}(s)$}

	First, we discuss the meromorphic behavior of the local period $\P_v$ as defined in \eqref{eq:main-RS-period-local},
	which will be used to prove meromorphic continuation of $I_2^{\mathrm{D}}(s)$.

	\begin{lem}\label{lem:holomorphicity-local-factor-cuspidal}
		Let $\pi\in\Pi_{\mathrm{c}}(\bar{G}_2)$ and $\varphi\in\1_{n-2}\boxplus\pi$ be $s$-independent. Also, recall $\P_v(\lambda;\varphi_v,f_v,f_{v,s})$ from \eqref{eq:main-RS-period-local} and $\lambda_{\mathrm{D}}(s)$ from \eqref{eq:def-lambda-discrete}. Then
		\begin{equation*}
				\frac{\P_v(-\lambda_{\mathrm{D}}(s);\varphi_v,f_v,f_{v,s})}{L_v\left(\tfrac{n-1}{2}+\tfrac{n}{2}s,\pi_v\right)L_v\left(\tfrac{1}{2}+\tfrac{n}{2}s,\pi_v\right)}
		\end{equation*}
		can be analytically continued in $s\in\C$ as an entire function.
	\end{lem}

	\begin{proof}
		Let $s'\in\C$. First, from the definition \eqref{eq:def-local-W-phi-cusp}
		we note that $\W(\varphi_{\lambda_\mathrm{D}(s'),v})$ is entire in $s'$.
		Now, as discussed in the proof of Lemma \ref{lem:factorization-period}
		and \S\ref{sec:cuspidal-local-choice},
		$\P_v(-\lambda_{\mathrm{D}}(s);\varphi_v,f_v,f_{v,s})$
		can be realized as the restriction to $s=-s'$ of
		$$\intop_{K_{n,v}}\intop_{N_2(F_v)\bs G_2(F_v)}
		\W(\varphi_{\lambda_\mathrm{D}(s'),v})\cdot
		\overline{\W(f_v)}\cdot
		\Phi_v(e_n\bullet)\left[\begin{pmatrix}\mathrm{I}_{n-2}&\\&h\end{pmatrix}k\right]\\
		|\det h|^{n-3/2+s}\, dh\, dk.$$
		The above is a $K_n$-average of a $\GL_2\times\GL_2$ local
		Rankin--Selberg integral of an $s$-independent Whittaker vector
		(that is entire in $s'$) in $\pi_v\cdot|\det|_v^{\frac{n-2}{2}(s'-1)}$,
		an $s,s'$-independent Whittaker vector in
		$|\cdot|_v^{0}\boxplus|\cdot|_v^{-\frac{n-2}{2}}$, and an
		$s,s'$-independent Schwartz function on $F_v^2$ at $n-\tfrac{3}{2}+s$.
		The $k$-average being a compact integral does not alter the analytic properties.
		Thus using the analytic theory of local zeta integrals (\emph{e.g.}, see
		\cite[Theorem 3.5]{cogdell2007functions}) and restricting to $s'=-s$ we conclude that
		\[
		s\mapsto
		 \tfrac{\P_v(\lambda_{\mathrm{D}}(s');\varphi_v,f_v,f_{v,s})}
			{L_v\left(\tfrac{n-1}{2}+\tfrac{n}{2}s,\pi_v\right)
				L_v\left(\tfrac{1}{2}+\tfrac{n}{2}s,\pi_v\right)}
		\]
		is entire in $s$.
	\end{proof}

	\begin{prop}\label{prop:cont-I2D}
		The expression of $I_2^{\mathrm{D}}(s)$ as given in Proposition \ref{prop:before-meromorphic-continuation} extends holomorphically to $\Re(s)>0$ and continuously to $\Re(s)\ge0$.
	\end{prop}

	\begin{proof}
		First, we write
		\begin{multline*}
			\frac{L^S\left(\tfrac{n-1}{2}+\tfrac{n}{2}s,\pi\right)L^S\left(\tfrac{1}{2}+\tfrac{n}{2}s,\pi\right)}{\sqrt{L^S(1,\pi,\Ad)}}\P_S(-\lambda_{\mathrm{D}}(s);\varphi,f,f_s)\\
			=\frac{\Lambda\left(\tfrac{n-1}{2}+\tfrac{n}{2}s,\pi\right)\Lambda\left(\tfrac{1}{2}+\tfrac{n}{2}s,\pi\right)}{\sqrt{\Lambda(1,\pi,\Ad)}}\prod_{v\in S}
			\frac{\P_v(-\lambda_{\mathrm{D}}(s),\varphi_v,f_v,f_{v,s})\sqrt{L_v(1,\pi_v,\Ad)}}
				{L_v\left(\tfrac{n-1}{2}+\tfrac{n}{2}s,\pi_v\right)L_v\left(\tfrac{1}{2}+\tfrac{n}{2}s,\pi_v\right)}.
		\end{multline*}
		Using \cite[Theorem 4.2]{cogdell2007functions} along with the cuspidality of $\pi$ and Lemma \ref{lem:holomorphicity-local-factor-cuspidal} with the choice of $\varphi$ there, we conclude that the left hand side above has analytic continuation in $s$ and is entire.

		By \cite[Theorem 4.17 (1)]{fine-boi}, $\Eis\left(\ol{\varphi},\lambda_{\mathrm{D}}(s)\right)$ is holomorphic for $\Re(s)\ge0$. Finally, Remark \ref{rmk:extra-decay-rep-lambda} gives absolute and locally uniform convergence of the double sum defining $I_2^{\mathrm{D}}(s)$ on this closed half-plane. This yields holomorphicity for $\Re(s)>0$ and continuity on $\Re(s)\ge0$.
	\end{proof}

	\subsection{Meromorphic continuation of $I_2^{\mathrm{C}}(s)$}
	\label{sec:mero_cont_for_continuous_spec}

	The process in this case is similar to that of $I_2^{\mathrm{D}}(s)$.
	We will focus on the differences.

	\begin{lem}\label{lem:holomorphicity-local-factor-eisenstein}
		Let $\chi\in\Pi(G_1)$ and $\varphi\in\1_{n-2}\boxplus\chi\boxplus\chi^{-1}$ be $(s,z)$-independent.
		Also, recall
		$\P_v(\lambda;\varphi_v,f_v,f_{v,s})$
		 from \eqref{eq:main-RS-period-local} and $\lambda_{\mathrm{C}}(s,z)$ from \eqref{eq:def-lambda-cont}. Then
		\begin{equation*}
			\frac{\P_v(-\lambda_{\mathrm{C}}(s,z);\varphi_v,f_v,f_{v,s})L_v(1+2z,\chi_v^2)}
				{\prod_{\pm}L_v\left(\tfrac{n-1}{2}+\tfrac{n}{2}s\pm z,\chi^\pm_v\right)
					L_v\left(\tfrac{1}{2}+\tfrac{n}{2}s\pm z,\chi^\pm_v\right)}
		\end{equation*}
		extends meromorphically to $(s,z)\in\C^2$, with possible poles only at poles of $L_v(1+2z,\chi_v^2)$. In particular, it is holomorphic for $|\Re(z)|<\tfrac12$.
	\end{lem}

	\begin{proof}
		We first confirm that
		\begin{equation*}
			(s',z)\mapsto\W\left(\varphi_{\lambda_{\mathrm{C}}(s',z),v}\right)
		\end{equation*}
		is entire for $\varphi_v$ being an $(s',z)$-independent vector. We do this by appealing to the analytic continuation of the Jacquet intertwiner; see \cite[Theorem 15.4.1]{wallach1992real2}. The argument of Lemma \ref{lem:holomorphicity-local-factor-cuspidal} shows that the displayed quotient without the numerator factor $L_v(1+2z,\chi_v^2)$ is entire. Multiplication by this factor gives the asserted meromorphic continuation; since $\chi_v$ is unitary, this factor is holomorphic for $\Re(z)>-\tfrac12$.
	\end{proof}

	\begin{prop}\label{prop:cont-I2C}
		The term $I_2^{\mathrm{C}}(s)$, initially defined for large $\Re(s)$
    in Proposition \ref{prop:before-meromorphic-continuation}, admits a
    meromorphic continuation to $\Re(s)>0$; for $0<\Re(s)<1/n$ it is given by
		\begin{multline*}
			M_{01}(s)+M_{11}(s)+\V\sum_{\chi\in\Pi(G_1)}
      \sum_{\varphi\in\B(\1_{n-2}\boxplus\chi\boxplus\chi^{-1})}
			\intop_{i\mathbb{R}}\left\langle \phi,
      \Eis(\varphi,\overline{\lambda_{\mathrm{C}}(s,z)})\right\rangle_{[\bar{G}_n]}\\
			\frac{\prod_{\pm}L^S\left(\tfrac{n-1}{2}+\tfrac{n}{2}s\pm z,\chi^\pm\right)
			L^S\left(\tfrac{1}{2}+\tfrac{n}{2}s\pm z,\chi^\pm\right)}{L^S(1+2z,\chi^2)}
			\P_S(-\lambda_{\mathrm{C}}(s,z);\varphi,f,f_s)
			 \, dz,
		\end{multline*}
		where $M_{01}$ and $M_{11}$ are given in \eqref{eq:M01-before-magic} and \eqref{eq:M11-before-magic}, respectively. Moreover, the third summand is holomorphic for $0<\Re(s)<1/n$ and continuous at $\Re(s)=0$, where it is defined by its limit from the right.
	\end{prop}
	\begin{rmk}\label{rmk:abuse-def-I2C}
		By an abuse of notation, we denote the third summand in the expression in Proposition \ref{prop:cont-I2C} also by $I_2^{\mathrm{C}}(s)$ whenever it converges. In particular we do not denote the meromorphic continuation of $I_2^{\mathrm{C}}(s)$ by the same name as is usual.
	\end{rmk}
	First, we record a result regarding zeros and poles of Eisenstein series that appear above.

	\begin{lem}\label{lem:pole-zero-eis}
		Let $\chi\in\Pi(G_1)$ and recall $\lambda_{\mathrm{C}}(s,z)$ from \eqref{eq:def-lambda-cont}.
		\begin{enumerate}
			\item Let $\chi^2\neq\1_1$ and $\varphi\in\1_{n-2}\boxplus\chi\boxplus\chi^{-1}$. Then $\Eis(\bar{\varphi},\lambda_{\mathrm{C}}(s,z))$ is holomorphic in the region $\Re\left(\tfrac{n}{2}s\right)\ge \Re(-z)\ge 0$.
			\item Let $\chi^2=\1_1\neq \chi$ and $\varphi\in\1_{n-2}\boxplus\chi\boxplus\chi^{-1}$. Then $\Eis(\bar{\varphi},\lambda_{\mathrm{C}}(s,z))$ is meromorphic in the region $\Re\left(\tfrac{n}{2}s\right)\ge \Re(-z)\ge 0$ with a simple polar divisor only at $z=-\tfrac{1}{2}$.
			\item Let $\chi=\1_1$ and $\varphi\in\1_{n-2}\boxplus\chi\boxplus\chi^{-1}$. Then $\Eis(\bar{\varphi},\lambda_{\mathrm{C}}(s,z))$ is meromorphic in the region $\Re\left(\tfrac{n}{2}s\right)\ge \Re(-z)\ge 0$ with simple polar divisors only at $z=-\tfrac{1}{2}$ and $z=\pm\left(\tfrac{n}{2}s-\tfrac{n-1}{2}\right)$.
		\end{enumerate}
	\end{lem}

	\begin{proof}
		We extensively use the notations from \cite[\S 4.4.1]{fine-boi}.

		Let $\lambda\in\a^\ast_{\widetilde{\Q_{n-2}},\C}$. For any $\chi\in\Pi(G_1)$ we have
		\begin{equation*}
			\sigma_{\1_{n-2}\otimes\chi\otimes\chi^{-1}} = \1_1^{\boxtimes n-2}\boxtimes\chi\boxtimes\chi^{-1}.
		\end{equation*}
		Thus if $\chi^2\neq \1_1$ then $\{\1_1,\chi,\chi^{-1}\}$ are pairwise non-isomorphic. Thus $L_{\1_{n-2}\otimes\chi\otimes\chi^{-1},E}(\lambda)$, as defined in \cite[eq.(4.29)]{fine-boi}, is identically $1$. Hence, \cite[Theorem 4.17 (1)]{fine-boi} implies that $\Eis(\varphi,\lambda)$ is holomorphic in (see \cite[\S3.4.10]{fine-boi} for the relevant definitions)
		\begin{equation}\label{eq:holomorphic-region}
			\Re(\lambda_1-\lambda_2),\Re(\lambda_1-\lambda_3),\Re(\lambda_2-\lambda_3) \ge 0.
		\end{equation}
		Similarly, if $\chi^2=\1_1\neq\chi$ then working similarly we get that $(\lambda_2-\lambda_3-1)\Eis(\varphi,\lambda)$ is holomorphic in the region \eqref{eq:holomorphic-region}. Finally, if $\chi=\1_1$ then
		\begin{equation*}
			L_{\1_{n-2}\otimes\1_1\otimes\1_1,E}(\lambda)=\left(\lambda_1-\lambda_2-\tfrac{n-1}{2}\right) \left(\lambda_1-\lambda_3-\tfrac{n-1}{2}\right)(\lambda_2-\lambda_3-1),
		\end{equation*}
		and $L_{\1_{n-2}\otimes\1_1\otimes\1_1,E}$ is holomorphic in \eqref{eq:holomorphic-region}. Recalling \eqref{eq:def-lambda-cont} we conclude the proof.
	\end{proof}

	Let $\chi\neq\1_1$. We write
	\begin{multline*}
		\frac{\prod_{\pm}L^S\left(\tfrac{n-1}{2}+\tfrac{n}{2}s\pm z,\chi^\pm\right)L^S\left(\tfrac{1}{2}+\tfrac{n}{2}s\pm z,\chi^\pm\right)}
		{L^S(1+2z,\chi^2)}
		\P_S(-\lambda_{\mathrm{C}}(s,z);\varphi,f,f_s)
		\\
		=\frac{\prod_{\pm}\Lambda\left(\tfrac{n-1}{2}+\tfrac{n}{2}s\pm z,\chi^\pm\right)
			\Lambda\left(\tfrac{1}{2}+\tfrac{n}{2}s\pm z,\chi^\pm\right)}{\Lambda(1+2z,\chi^2)}\\
		\times\prod_{v\in S}
		\frac{\P_v(-\lambda_{\mathrm{C}}(s,z);\varphi_v,f_v,f_{v,s})L_v(1+2z,\chi_v^2)}
			{\prod_{\pm}L_v\left(\tfrac{n-1}{2}+\tfrac{n}{2}s\pm z,\chi_v^\pm\right)
				L_v\left(\tfrac{1}{2}+\tfrac{n}{2}s\pm z,\chi_v^\pm\right)}.
	\end{multline*}
	We record that $\Lambda(\cdot,\chi)$ is entire and $\Lambda(1+2z,\chi^2)$ has no zero on the line $\Re(z)=0$.
	Working in the strip $|\Re(z)|<\tfrac12$ of Lemma \ref{lem:holomorphicity-local-factor-eisenstein}, as in the proof of Proposition \ref{prop:cont-I2D}, and using Lemma \ref{lem:pole-zero-eis} (1)-(2) we conclude that the $\chi\neq\1_1$ part of the sum in the expression of $I_2^{\mathrm{C}}(s)$ in Proposition \ref{prop:before-meromorphic-continuation} is holomorphic for $\Re(s)>0$ and extends continuously to $\Re(s)=0$, by the bounds in Remark \ref{rmk:extra-decay-rep-lambda}.

	Now we consider the $\chi=\1_1$ case. We write $L(\cdot,\chi)=\zeta$ and
	$\Lambda(\cdot,\chi)=\xi$, and denote by $\tilde \P_S(s,z)$ the map
	\begin{equation}\label{def-PS-tilde}
		\vphi\mapsto
		\prod_{v\in S}\frac{\P_v(-\lambda_{\mathrm{C}}(s,z);\varphi_v,f_v,f_{v,s})
				\zeta_v(1+2z)}{\prod_{\pm}\zeta_v\left(\frac{n-1}{2}+\frac{n}{2}s\pm z\right)
				\zeta_v\left(\frac{1}{2}+\frac{n}{2}s\pm z\right)},
	\end{equation}
	which we regard as a meromorphic family, holomorphic for $|\Re(z)|<\tfrac12$, of linear functionals on
	$C^\infty(K_n)$, using the compact realization of the induced vectors.
	We first record a useful lemma.

	\begin{lem}\label{lem:evenness-I2C}
		The integrand of
		\begin{equation*}
			I^1(s):=\intop_{i\mathbb{R}}\sum_{\varphi\in\B(\1_{n-2}\boxplus\1_1\boxplus\1_1)}\left\langle \phi,\Eis(\varphi,\overline{\lambda_{\mathrm{C}}(s,z)})\right\rangle_{[\bar{G}_n]}
			\frac{\prod_{\pm}\xi\left(\frac{n-1}{2}+\frac{n}{2}s\pm z\right)\xi\left(\frac{1}{2}+\frac{n}{2}s\pm z\right)}{\xi(1+2z)}\widetilde{\P}_S(s,z)(\vphi)\, dz
		\end{equation*}
		is even as a function of $z$. Moreover, as a function of $s$, the integral has analytic continuation to $\Re(s)>\frac{n-1}{n}$.
	\end{lem}
	\begin{proof}
		Since the integrand is meromorphic in $(s,z)$, it suffices to prove evenness
		for $\Re(s)=0=\Re(z)$. Reverse engineering via \emph{e.g.}, Proposition
		\ref{prop:before-meromorphic-continuation} the integrand equals
		\begin{equation*}
			\sum_{\varphi\in\B(\1_{n-2}\boxplus\1_1\boxplus\1_1)}\left\langle \phi,\Eis(\varphi,\overline{\lambda_{\mathrm{C}}(s,z)})\right\rangle_{[\bar{G}_n]}
			\P(-\lambda_{\mathrm{C}}(s,z);\varphi,f,f_s).
		\end{equation*}
		Temporarily, let $\M:=\M_{\1_{n-2}\times\1_1\times\1_1}\left(\left(\begin{smallmatrix}\mathrm{I}_{n-2}&&\\&&1\\&1&\end{smallmatrix}\right),\lambda_{\mathrm{C}}(s,z)\right)$. Then $\M$ is unitary as $\Re(s)=0=\Re(z)$; see \cite[Theorem 7.2(a)]{arthur2005introduction}. Using the invariance of the choices of the orthonormal basis in the sum in the above integrand we replace the $\varphi$-sum by a sum over
		\begin{equation*}
			\{\M\varphi\mid \varphi\in\B(\1_{n-2}\boxplus\1_1\boxplus\1_1)\}.
		\end{equation*}
		Now applying the functional equation of the Eisenstein series \eqref{eq:FE-for-Eis} we obtain
		\begin{equation*}
			\Eis(\M\bar{\varphi},\lambda_{\mathrm{C}}(s,-z))=\Eis(\bar{\varphi},\lambda_{\mathrm{C}}(s,z))
		\end{equation*}
		and applying the same for $\GL_2$ Eisenstein series $\Eis^{\Q_{n-2}}$ we obtain
		\begin{equation*}
			\Eis^{\Q_{n-2}}(\M\varphi,\lambda_{\mathrm{C}}(s,-z))=\Eis^{\Q_{n-2}}(\varphi,\lambda_{\mathrm{C}}(s,z)).
		\end{equation*}
		Recalling the definition of $\P$ from \eqref{eq:main-RS-period}, the first assertion follows.

		To prove the second assertion, note that for $\Re(s)\ge 0$ and $|\Re(z)|$ sufficiently small $$\prod_{\pm}\xi\left(\frac{n-1}{2}+\frac{n}{2}s\pm z\right)\xi\left(\frac{1}{2}+\frac{n}{2}s\pm z\right)$$ as a function of $(s,z)$ has at most simple polar divisors at the hyperplanes defined by
		\begin{equation}\label{eq:possible-poles}
			\pm z=\begin{cases}\tfrac{n}{2}s-\tfrac{1}{2},\quad                  & n>3; \\
				\tfrac{3}{2}s-\tfrac{1}{2},\, \tfrac{3}{2}s,\quad & n=3.\end{cases}
		\end{equation}
		On the other hand, for $n=3$ and $\pm z =\tfrac{3}{2}s$ we have $\lambda_{\mathrm{C}}(s,z)$ equals either $(s,-2s,s)$ or $(s,s,-2s)$, where $\Eis(\varphi,\cdot)$ has a zero; see \cite[Theorem 4.17 (2)]{fine-boi}.
		Thus for $n\ge 3$ it follows from Lemma \ref{lem:pole-zero-eis} (3) that as a function of $s$ the integral in the lemma has holomorphic continuation in $\Re(s)>\tfrac{n-1}{n}$.
	\end{proof}

	As an immediate corollary of the first assertion of Lemma \ref{lem:evenness-I2C} we see that if $z=s_0$ is a simple pole of the integrand in Lemma \ref{lem:evenness-I2C} then so is $-s_0$ and
	\begin{equation}\label{eq:residue-even}
		\Res_{z=s_0} = -\Res_{z=-s_0},
	\end{equation}
	where here and elsewhere $\Res$ denotes the residue operator.

	\begin{lem}\label{lem:res-deg-eis-original}
		Let $d\ge 2$ and for $\lambda:=(\lambda_1,\lambda_2)\in\a^\ast_{\Q_{d-1},\C}$ let $\varphi_\lambda\in\1_{d-1}\cdot|\det|^{\lambda_1}\boxplus\1_1\cdot|\cdot|^{\lambda_2}$ be a factorizable vector. If $\varphi_\lambda$ is holomorphic in a sufficiently small neighborhood of the hyperplane $\lambda_1-\lambda_2=\tfrac{d}{2}$ then $\Eis(\varphi_\lambda)$ has at most a simple polar divisor along the same hyperplane, where the residue with respect to $\lambda_1-\lambda_2$ is given by
		\begin{equation*}
			|\det|^{\lambda_1-\frac{1}{2}} \Delta^{-\frac{d-1}{2}}\frac{\xi^\ast(1)}{\xi(d)}\int_{K_d}\varphi_\lambda(k)\, dk\in\1_d\cdot|\det|^{\frac{(d-1)\lambda_1+\lambda_2}{d}},
		\end{equation*}
		for $\lambda_1-\lambda_2=\tfrac{d}{2}$.
	\end{lem}
	The above lemma is known; for instance, it can be extracted from \cite[Theorem 1-5]{hanzer2015degenerate}. Here we give a short proof in our particular case of interest.

	\begin{proof}
		By the surjectivity of the intertwiner between the Schwartz-space model and the induced model, we find a factorizable Schwartz--Bruhat function $\Phi_\lambda\in\mathcal{S}(\A^d)$ holomorphic in $\lambda$ so that
		$\Phi_{\lambda,v}=\zeta_v(\tfrac{d}{2}+\lambda_1-\lambda_2)^{-1}\mathbbm{1}_{\o_v^d}$ for almost every $v$ and
		\begin{equation*}
			\varphi_\lambda = |\det|^{\frac{1}{2}+\lambda_1}\int_{\A^\times}\Phi_\lambda(te_d\bullet)|t|^{\frac{d}{2}+\lambda_1-\lambda_2}\,d^\times t.
		\end{equation*}

		The above converges absolutely for $\Re(\lambda_1-\lambda_2)>1-\tfrac{d}{2}$ and has analytic continuation to all $\lambda$. This follows via Tate's thesis; see \cite[Proposition 3.1.6]{bump19927automorphic}. Now it follows from \cite[Proof of Proposition 2.1]{cogdell2007functions} that $\Eis(\varphi_\lambda)$ has a simple polar divisor along the hyperplane $\lambda_1-\lambda_2=\tfrac{d}{2}$ where the residue is given by
		\begin{equation*}
			\frac{\vol([G_1^1])}{d}\Res_{\lambda_1-\lambda_2=\frac{d}{2}}\frac{|\det|^{\lambda_1-\frac{1}{2}}}{\frac{\lambda_1-\lambda_2}{d}-\tfrac{1}{2}}\widehat{\Phi}_\lambda(0)=\Delta^{\frac 12}\xi^\ast(1)\frac{|\det|^{\lambda_1-\frac{1}{2}}}{\xi(d)}\prod_v\zeta_v(d)\hat{\Phi}_{\lambda,v}\mid_{\lambda_1-\lambda_2=\frac{d}{2}}(0).
		\end{equation*}
		Note that the $v$-adic factor above equals $1$ for almost every $v$.

		On the other hand, for $\lambda_1-\lambda_2=\tfrac{d}{2}$ we write equalities of absolutely convergent integrals
		\begin{equation*}
			\int_{K_{d,v}}\varphi_{\lambda,v}\vert_{\lambda_1-\lambda_2=\frac{d}{2}}(k)\, dk
			=\int_{K_{d,v}}\int_{F_v^\times}\Phi_{\lambda,v}(te_d k)|t|^d\,d^\times t\, dk=\Delta_v^{d/2}\zeta_v(d)\hat{\Phi}_{\lambda,v}(0).
		\end{equation*}
		This concludes the proof.
	\end{proof}

	A similar proof yields the analogous result for the Eisenstein series with respect to the associate parabolic, which is the content of the following lemma.

	\begin{lem}\label{lem:res-deg-eis-dual}
		Let $d\ge 2$ and for $\lambda:=(\lambda_1,\lambda_2)\in\a^\ast_{\Q_1,\C}$ let $\varphi_\lambda\in\1_1\cdot|\cdot|^{\lambda_1}\boxplus\1_{d-1}\cdot|\det|^{\lambda_2}$ be a factorizable vector. If $\varphi_\lambda$ is holomorphic in a sufficiently small neighborhood of the hyperplane $\lambda_1-\lambda_2=\tfrac{d}{2}$ then $\Eis(\varphi_\lambda)$ has at most a simple polar divisor along the same hyperplane where the residue with respect to $\lambda_1-\lambda_2$ is given by
		\begin{equation*}
			|\det|^{\frac{1}{2}+\lambda_2}\cdot \Delta^{-\frac{d-1}{2}}\frac{\xi^\ast(1)}{\xi(d)}\int_{K_d}\varphi_\lambda(k)\, dk\in\1_d\cdot|\det|^{\frac{\lambda_1+(d-1)\lambda_2}{d}},
		\end{equation*}
		for $\lambda_1-\lambda_2=\tfrac{d}{2}$.
	\end{lem}

	\begin{proof}
		Temporarily, let $w=w_\ell^d$ as in \eqref{eq:def-main-intertwiner} and $\varphi^1_\lambda:=\varphi_\lambda\left(w\bullet^{-\top}\right)$ which lies in $\1_{d-1}\cdot|\det|^{-\lambda_2}\boxplus\1_1\cdot|\cdot|^{-\lambda_1}$. Consequently, for sufficiently dominant $\lambda$ we have
		\begin{equation*}
			\Eis(\varphi^1_\lambda)(g^{-\top})=\sum_{\gamma\in\Q_{d-1}(F)\bs G_d(F)}\varphi_\lambda(w\gamma^{-\top}w^{-1}g)=\sum_{\gamma\in\Q_1(F)\bs G_d(F)}\varphi_\lambda(\gamma g)=\Eis(\varphi_\lambda)(g).
		\end{equation*}
		Thus employing Lemma \ref{lem:res-deg-eis-original} we obtain that the residue equals
		\begin{equation*}
			|\det g|^{\frac{1}{2}+\lambda_2}\Delta^{-\frac{d-1}{2}}\frac{\xi^\ast(1)}{\xi(d)}\int_{K_d}\varphi_\lambda(wk^{-\top})\, dk\Big\vert_{\lambda_1-\lambda_2=\frac{d}{2}},
		\end{equation*}
		which yields the claim after the change of variable $k\mapsto wk^{-\top}$.
	\end{proof}

	\begin{lem}\label{lem:eis-residue}
		Let $\varphi\in\1_{n-2}\boxplus\1_1\boxplus\1_1$ be $(s,z)$-independent. Then there exists $\beta_\varphi\in\1_{n-1}\boxplus\1_1$ so that
		\begin{equation*}
			\Res_{z=\frac{n-1}{2}-\frac{n}{2}s}\Eis(\varphi,\lambda_{\mathrm{C}}(s,z))=\Eis(\beta_\varphi,\lambda_{\mathrm{E}}(s-\tfrac{1}{2})),
		\end{equation*}
		where $\lambda_{\mathrm{E}}$ is as in \eqref{eq:def-lambda-E},
		for $s$ in general position.
	\end{lem}

	\begin{proof}
		Let, temporarily, $\Re(s)$ and $-\Re(z)$ be sufficiently positive (equivalently, $\lambda_{\mathrm{C}}(s,z)$ is sufficiently dominant). We thus write
		\begin{align}\label{eq:eisenstein-in-stages}
			\Eis(\varphi,\lambda_{\mathrm{C}}(s,z))
			& =\sum_{\gamma\in\Q_{n-1}(F)\bs G_n(F)}\sum_{\gamma'\in\widetilde{\Q_{n-2}}(F)\bs\Q_{n-1}(F)}\varphi_{\lambda_{\mathrm{C}}(s,z)}(\gamma'\gamma\cdot)\nonumber \\
			& =\sum_{\gamma\in\Q_{n-1}(F)\bs G_n(F)}\Eis^{\Q_{n-1}}(\varphi,\lambda_{\mathrm{C}}(s,z))(\gamma\cdot).
		\end{align}
		Using induction in stages and writing $\Q_{n-1}':=\widetilde{\Q_{n-2}}\cap G_{n-1}$, so $\Q_{n-1}'\bs G_{n-1}\cong\widetilde{\Q_{n-2}}\bs\Q_{n-1}$ we obtain
		\begin{equation*}
			\delta_{\Q_{n-1}}^{-s}\cdot\varphi_{\lambda_{\mathrm{C}}(s,z)}\left[\begin{pmatrix}\bullet&\\&1\end{pmatrix}g\right]\in\1_{n-2}\cdot|\det|^{\frac{1}{2}}\boxplus\1_1\cdot|\cdot|^{-z+\frac{1}{2}-\frac{n}{2}s}
		\end{equation*}
		for $g\in G_n(\A)$. Moreover, we realize
		\begin{equation*}
			\Eis^{\Q_{n-1}}(\varphi,\lambda_{\mathrm{C}}(s,z))(g) = \Eis_{\Q'_{n-1}}\left(\delta_{\Q_{n-1}}^{s}\cdot\varphi_{\lambda^n_{\mathrm{C}}(s,z)}\left[\begin{pmatrix}\bullet&\\&1\end{pmatrix}g\right],\lambda^{n-1}_{\mathrm{C}}(s,z)\right)(1)
		\end{equation*}
		as a $G_{n-1}$ maximal degenerate Eisenstein series, where
		\begin{equation*}
			\lambda^{n-1}_{\mathrm{C}}(s,z):=(\tfrac{1}{2},\tfrac{1}{2}-z-\tfrac{n}{2}s)\in\a^\ast_{\Q'_{n-1},\C};\quad\lambda^n_{\mathrm{C}}(s,z):=(0,0,z+\tfrac{n}{2}s)\in\a^\ast_{\widetilde{\Q_{n-2}},\C}.
		\end{equation*}
		Thus from Lemma \ref{lem:res-deg-eis-original} it follows that the above degenerate Eisenstein series has a simple polar divisor along the hyperplane
		\begin{equation*}
			z+\tfrac{n}{2}s=\tfrac{n-1}{2}\quad\iff\quad z=\tfrac{n-1}{2}-\tfrac{n}{2}s
		\end{equation*}
		and that there exists $\beta_\varphi\in\1_{n-1}\boxplus\1_1$ so that \begin{equation}\label{eq:residue-degenerate-eis}
			\Res_{z=\frac{n-1}{2}-\frac{n}{2}s}\Eis^{\Q_{n-1}}(\varphi,\lambda_{\mathrm{C}}(s,z))(g)=\left(\beta_\varphi\right)_{\lambda_{\mathrm{E}}(s-\frac{1}{2})}(g),
		\end{equation}
		where $\lambda_{\mathrm{E}}(s)$ is as in the statement of the lemma.

		Now we compute the required residue using the expression \eqref{eq:eisenstein-in-stages}. As the sum in \eqref{eq:eisenstein-in-stages} is absolutely convergent uniformly in $s,z$ lying in compact sets in a sufficiently dominant cone, we can interchange the sum and the residue operator. To compute the residue of the summands we realize them as $z$-limits staying in a sufficiently dominant cone. In other words, we compute the limit as ${z\to \frac{n-1}{2}-\frac{n}{2}s}$ with $\Re(z)<0$ and $\Re(s)>\frac{n-1}{n}$. Thus using \eqref{eq:residue-degenerate-eis} we obtain
		\begin{equation*}
			\Res_{z=\frac{n-1}{2}-\frac{n}{2}s}\Eis(\varphi,\lambda_{\mathrm{C}}(s,z))=\sum_{\gamma\in\Q_{n-1}(F)\bs G_n(F)}(\beta_\varphi)_{\lambda_{\mathrm{E}}(s-\frac{1}{2})}(\gamma\cdot)=\Eis\left(\beta_{\varphi},\lambda_{\mathrm{E}}(s-\tfrac{1}{2})\right)
		\end{equation*}
		for $\Re(s)$ sufficiently large. We conclude by meromorphic continuation.
	\end{proof}

	\begin{lem}\label{lem:residue-projection}
		Let $\mathcal{L}(s,z)$ be a linear functional on $C^\infty(K_n)$, viewing induced vectors through their restrictions to $K_n$, that is holomorphic and satisfies $\mathcal{L}(s,z)(\varphi)\ll_{s,z}\nu_\varphi^{O(1)}$ continuously in $(s,z)$, in a sufficiently small neighborhood of the hyperplane $z+\tfrac{n}{2}s=\tfrac{n-1}{2}$. Then
		\begin{multline*}
			\Res_{z=\frac{n-1}{2}-\frac{n}{2}s}\sum_{\varphi\in\B(\1_{n-2}\boxplus\1_1\boxplus\1_1)}\left\langle\phi,\Eis\left(\varphi,\overline{\lambda_{\mathrm{C}}(s,z)}\right)\right\rangle_{[\bar{G}_n]}\mathcal{L}(s,z)(\varphi)\\
			=\sum_{\beta\in\B(\1_{n-1}\boxplus\1_1)}\left\langle\phi,\Eis\left(\beta,\overline{\lambda_{\mathrm{E}}(s-\tfrac{1}{2})}\right)\right\rangle_{[\bar{G}_n]}\mathcal{L}\left(s,\tfrac{n-1}{2}-\tfrac{n}{2}s\right)(\beta)
		\end{multline*}
		where $\lambda_{\mathrm{E}}(s)$ is as in \eqref{eq:def-lambda-E}. On the residue hyperplane, the twists of $\beta|_{K_n}$ by $-\lambda_{\mathrm C}(s,z)$ and by $-\lambda_{\mathrm E}(s-\tfrac12)$ agree.
	\end{lem}

	\begin{proof}
		As $\phi$ is Schwartz, $\langle\phi,\Eis\left(\varphi,\overline{\lambda_{\mathrm{C}}(s,z)}\right)\rangle$ decays rapidly in $\varphi$ uniformly in $z$ in a compact set away from the poles of the Eisenstein series. So, we can and will interchange the sum with the residue operator. Applying Lemma \ref{lem:eis-residue} we obtain that the residue in the lemma equals
		\begin{equation*}
			\sum_{\varphi\in\B(\1_{n-2}\boxplus\1_1\boxplus\1_1)}\left\langle\phi,\Eis\left(\beta_\varphi,\overline{\lambda_{\mathrm{E}}(s-\tfrac{1}{2})}\right)\right\rangle_{[\bar{G}_n]}\mathcal{L}\left(s,\tfrac{n-1}{2}-\tfrac{n}{2}s\right)(\varphi).
		\end{equation*}
		Temporarily assume that $\Re(s)$ is sufficiently large. Then writing $\Eis\left(\beta_\varphi,\overline{\lambda_{\mathrm{E}}(s-\tfrac{1}{2})}\right)$ as an absolutely convergent sum and expanding $\beta_\varphi$ over an orthonormal basis of $\1_{n-1}\boxplus\1_1$ we write the above as
		\begin{equation*}
			\sum_{\beta\in\B(\1_{n-1}\boxplus\1_1)}\left\langle\phi,\Eis\left(\beta,\overline{\lambda_{\mathrm{E}}(s-\tfrac{1}{2})}\right)\right\rangle_{[\bar{G}_n]}\sum_{\varphi\in\B(\1_{n-2}\boxplus\1_1\boxplus\1_1)}\langle\beta,\beta_\varphi\rangle_{\1_{n-1}\boxplus\1_1}\mathcal{L}\left(s,\tfrac{n-1}{2}-\tfrac{n}{2}s\right)(\varphi)
		\end{equation*}
		where the interchange of the sums is justified because of the rapid decay of the inner products in $\beta$. We claim that
		\begin{equation*}
			\langle\beta,\beta_\varphi\rangle_{\1_{n-1}\boxplus\1_1} = \langle\beta,\varphi\rangle_{\1_{n-2}\boxplus\1_1\boxplus\1_1},
		\end{equation*}
		which immediately yields the lemma.

		Now we prove the claim. Note that by the classification of the residual representations by M{\oe}glin--Waldspurger (see \cite[\S 4.1.2]{fine-boi}) we can find a non-zero element $\alpha$ in the cuspidal representation $\boxplus_{i=1}^n\1_1$ such that
		\begin{equation*}
			\varphi=\Res_{\lambda=\left(\rho^{B_{n-2}},0,0\right)}\Eis^{\widetilde{\Q_{n-2}}}(\alpha,\lambda),\quad\beta_\varphi=\Res_{\lambda=\left(\rho^{B_{n-1}},0\right)}\Eis^{\Q_{n-1}}(\alpha,\lambda),
		\end{equation*}
		where the above $\Res$ denotes iterated residues in the sense of \cite{MW}.
		It follows from \cite[Proposition 4.4]{fine-boi} that
		\begin{equation*}
			\langle\beta,\beta_\varphi\rangle_{\1_{n-1}\boxplus\1_1} =\langle\beta_{B_n},\alpha\rangle_{\boxplus_{i=1}^n\1_1},\quad\left\langle\beta_{\widetilde{\Q_{n-2}}},\varphi\right\rangle_{\1_{n-2}\boxplus\1_1\boxplus\1_1}=\left\langle(\beta_{\widetilde{\Q_{n-2}}})_{B_n},\alpha\right\rangle_{\boxplus_{i=1}^n\1_1}.
		\end{equation*}
		Noting that $\beta_P=\beta$ for any standard parabolic $P$ we obtain
		\begin{equation*}
			\langle\beta,\beta_\varphi\rangle_{\1_{n-1}\boxplus\1_1}=\langle\beta,\varphi\rangle_{\1_{n-2}\boxplus\1_1\boxplus\1_1}
		\end{equation*}
		concluding the proof.
	\end{proof}

	\begin{rmk}
		We emphasize that \cite[Proposition 4.4]{fine-boi} is applicable here since we are adopting the same normalizations as in \cite{fine-boi}, in particular, the compatibility of measures given by \eqref{dg-for-iwasawa-PK}.
	\end{rmk}

	\begin{rmk}\label{rmk:stronger-holomorphic-region}
		We will need to know the analytic behavior of
    $\Eis(\varphi,\lambda_{\mathrm{C}}(s,z))$ for
    $\varphi\in\1_{n-2}\boxplus\1_1\boxplus\1_1$ in a slightly larger region
    than the one described in Lemma \ref{lem:pole-zero-eis}. Let $0<\epsilon<\tfrac14$
    be fixed but arbitrarily small. Using \cite[Theorem 4.17 (1)]{fine-boi}
    and the standard zero-free region of $\xi(1+2z)$, together with Lemma \ref{lem:holomorphicity-local-factor-eisenstein} in $|\Re(z)|<\tfrac12$, we define an even
    continuous function $\kappa:\R\to (0,\epsilon)$ so that the integrand
    in Lemma \ref{lem:evenness-I2C} remains holomorphic in
    $|\Re(z)|<2\kappa(\Im(z))$ except for the possible simple polar hyperplanes
    described in \eqref{eq:possible-poles}.
    \end{rmk}

	\begin{proof}[Proof of Proposition \ref{prop:cont-I2C}]
		As discussed after the proof of Lemma \ref{lem:pole-zero-eis} it suffices to continue $I^1(s)$, the $\chi=\1_1$ summand of the sum in the expression of $I_2^{\mathrm{C}}(s)$, as defined in Lemma \ref{lem:evenness-I2C}. We continue $I^1(s)$ arbitrarily to the right of $\Re(s)=\tfrac{n-1}{n}$ using the second assertion of Lemma \ref{lem:evenness-I2C}. Let $0<\tfrac n2(\Re(s)-\tfrac{n-1}{n})<\kappa(\tfrac n2\Im(s))$ where $\kappa$ is from Remark \ref{rmk:stronger-holomorphic-region}. We shift the contour of the above to $\Re(z)=-\kappa(\Im(z))$. In this process, we pick up the pole at $z=-\tfrac{n}{2}s+\tfrac{n-1}{2}$, where the residue is given by
		\begin{multline}\label{eq:residue-1}
			\Res_{z=\frac{n-1}{2}-\frac{n}{2}s}\sum_{\varphi\in\B(\1_{n-2}\boxplus\1_1\boxplus\1_1)}\left\langle \phi,\Eis(\varphi,\overline{\lambda_{\mathrm{C}}(s,z)})\right\rangle_{[\bar{G}_n]}\\
			\frac{\xi(n-1)\xi(\tfrac{n}{2})\xi(ns)\xi(ns-\tfrac{n-2}{2})}{\xi(n-ns)}\widetilde{\P}_S\left(s,\tfrac{n-1}{2}-\tfrac{n}{2}s\right)(\vphi).
		\end{multline}
		We notice that the integral over the deformed contour as a function of $s$ is holomorphic in $\tfrac n2|\Re(s)-\tfrac{n-1}{n}|<\kappa(\tfrac n2\Im(s))$. Now we let $0<\tfrac n2(\tfrac{n-1}{n}-\Re(s))<\kappa(\tfrac n2\Im(s))$ and deform the contour of this new integral back to $\Re(z)=0$. In this process, we pick up the pole at $z=\tfrac{n}{2}s-\tfrac{n-1}{2}$, where the residue is also given by \eqref{eq:residue-1} thanks to \eqref{eq:residue-even}. Thus we have continued $I^1(s)$ in the region $\tfrac{1}{n}<\Re(s)<\tfrac{n-1}{n}$, where it is given by the sum of the integral in Lemma \ref{lem:evenness-I2C} and twice the residue given in \eqref{eq:residue-1}.

		We briefly take an interlude to describe the residues that appear in \eqref{eq:residue-1}.
		Reverse engineering we realize the sum in \eqref{eq:residue-1} as
		\begin{equation*}
			\sum_{\varphi\in\B(\1_{n-2}\boxplus\1_1\boxplus\1_1)}\left\langle \phi,\Eis(\varphi,\overline{\lambda_{\mathrm{C}}(s,z)})\right\rangle_{[\bar{G}_n]}
			\P(-\lambda_{\mathrm{C}}(s,z);\varphi,f,f_s)
		\end{equation*}
		and then we apply Lemma \ref{lem:residue-projection}. Thus we obtain that $2\V$ times the expression in \eqref{eq:residue-1} equals
		\begin{multline}\label{eq:M11-before-magic}
			M_{11}(s):=2\xi(n-1)\xi(\tfrac{n}2)\V\sum_{\varphi\in\B(\1_{n-1}\boxplus\1_1)}\left\langle\phi,\Eis\left(\varphi,\overline{\lambda_{\mathrm{E}}(s-\tfrac{1}{2})}\right)\right\rangle_{[\bar{G}_n]}\\
			\times\frac{\xi(ns)\xi(ns-\tfrac{n-2}{2})}{\xi(n-ns)}\widetilde{\P}_S(s,\tfrac{n-1}2-\tfrac n2s)(\varphi),
		\end{multline}
		with $\lambda_{\mathrm{E}}(s)$ as in \eqref{eq:def-lambda-E} and $\tilde \P_S$ given by \eqref{def-PS-tilde}.
		Now we go back to $I^1(s)$ for $\tfrac{1}{n}<\Re(s)<\tfrac{n-1}{n}$. We perform a similar $z$-contour shifting
		maneuver as above to meromorphically continue it to $0<\Re(s)<\tfrac{1}{n}$, and take the continuous boundary value of the integral from the right using Remark \ref{rmk:extra-decay-rep-lambda}. We notice that in this process we pick up two poles at $z=\pm\left(\tfrac{n}{2}s-\tfrac{1}{2}\right)$. Once again applying \eqref{eq:residue-even} we obtain that the residue equals
		\begin{multline}\label{eq:M01-before-magic}
			M_{01}(s):=2\xi^\ast(1)\xi(\tfrac{n}{2})\V\sum_{\varphi\in\B(\1_{n-2}\boxplus\1_1\boxplus\1_1)}\left\langle\phi,\Eis\left(\varphi,\ol{\lambda_{\mathrm{L}}(s)}\right)\right\rangle_{[\bar{G}_n]}\\
			\times\xi(\tfrac {n-2}2+ns)\widetilde{\P}_S(s,\tfrac n2s-\tfrac 12)(\varphi)
		\end{multline}
		where, again, $\tilde \P_S$ is given by \eqref{def-PS-tilde} and
		\begin{equation}\label{eq:def-lambda-L}
			\lambda_{\mathrm{L}}(s):=\lambda_{\mathrm{C}}(s,\tfrac{n}{2}s-\tfrac{1}{2})
			=\left(s,\tfrac12-(n-1)s,s-\tfrac{1}{2}\right)\in\a^\ast_{\widetilde{\Q_{n-2}},\C}.
		\end{equation}
		On the hyperplane $z=\tfrac n2s-\tfrac12$, the factor $\zeta_v(1+2z)=\zeta_v(ns)$ in \eqref{def-PS-tilde} cancels the denominator factor $\zeta_v(\tfrac12+\tfrac n2s+z)$. The restriction in \eqref{eq:M01-before-magic} is understood after this cancellation and by meromorphic continuation. This concludes the proof.
	\end{proof}

	\subsection{Comparing $M_{11}$ to an inner product}

	The goal of this section is to show the following result.
	\begin{prop}\label{prop:M11-magic}
		Recall $M_{11}(s)$ from \eqref{eq:M11-before-magic}. We have
		\begin{equation*}
			M_{11}(s)=
			\left\langle \phi, \Eis\left(\M f\cdot \M\ol{f_s}\right)\right\rangle_{[\bar{G}_n]}
		\end{equation*}
		where $\M$ is as in \eqref{eq:def-main-intertwiner}.
	\end{prop}

	\begin{proof}
		Note that the adjoint $\M^\vee$ is the intertwiner $\M_{\1_1\times\1_{n-1}}\left(\left(\begin{smallmatrix}&\mathrm{I}_{n-1}\\1&\end{smallmatrix}\right),0\right)$.
		First, we check that
		\begin{equation*}
			\beta(f,s):=\exp\left\langle-\lambda_{\mathrm{E}}(s-\tfrac{1}{2}),
			H_{\Q_{n-1}}(\cdot)\right\rangle\times
			\M^\vee\left(\M\bar{f}\cdot \M f_s\right)\in\1_{n-1}\boxplus\1_1
		\end{equation*}
		where $\lambda_{\mathrm{E}}$ is as in \eqref{eq:def-lambda-E}.
		By \eqref{eq:FE-for-Eis}, we have
		\begin{equation*}
			\Eis\left(\M\bar{f}\cdot \M f_s\right)=\Eis\left(\M^\vee\left(\M\bar{f}\cdot \M f_s\right)\right).
		\end{equation*}
		Then for $\Re(s)$ sufficiently positive we write the Eisenstein series on the right hand side above as
		\begin{equation*}
			\sum_{\gamma\in\Q_{n-1}(F)\bs G_n(F)}\exp\left\langle\lambda_{\mathrm{E}}(s-\tfrac{1}{2}),H_{\Q_{n-1}}(\cdot)\right\rangle\beta(f,s)(\gamma),
		\end{equation*}
		which is absolutely convergent.
		Now expanding $\beta(f,s)$ over an orthonormal basis of $\1_{n-1}\boxplus\1_1$ and interchanging sums (justified by their absolute convergence) we obtain
		\begin{equation*}
			\Eis\left(\M\bar{f}\cdot \M f_s\right)=\sum_{\varphi\in\B(\1_{n-1}\boxplus\1_1)}\Eis\left(\bar{\varphi},\lambda_{\mathrm{E}}(s-\tfrac{1}{2})\right)\IP{\beta(f,s),\bar{\varphi}}_{\1_{n-1}\boxplus\1_1}.
		\end{equation*}
		By meromorphic continuation, the above holds for $s$ in general position.
		We know from \cite[Theorem 7.2 (a)]{arthur2005introduction} that $\M$ acts unitarily on unitary representations. Thus on $\Re(s)=\tfrac{1}{2}$ we have
		\begin{align*}
			\IP{\beta(f,s),\bar{\varphi}}_{\1_{n-1}\boxplus\1_1} & =\IP{\M\varphi_{\lambda_{\mathrm{E}}(\frac{1}{2}-s)},\M f\cdot\M \ol{f_s}}_{\1_1\cdot|\cdot|^{(n-1)(\frac{1}{2}-s)}\boxplus\1_{n-1}\cdot|\det|^{s-\frac{1}{2}}} \\
			& =\V'\int_{K_n}\M\varphi_{\lambda_{\mathrm{E}}(\frac{1}{2}-s)}\cdot\M\bar{f}\cdot \M f_{s},
		\end{align*}
		where
		\begin{equation}\label{def-mathcal-V-prime}
			\V'=\vol([{M}_{\Q_{n-1}}]^1).
		\end{equation}
		We similarly obtain that the right-hand side above is the meromorphic continuation of
		\begin{equation*}
			\IP{\M^\vee\left(\M\varphi_{\lambda_{\mathrm{E}}(\frac{1}{2}-s)}\cdot\M\bar{f}\right),
				\ol{f_s}}_{\1_{n-1}\cdot|\det|^s\boxplus\1_1\cdot|\cdot|^{-(n-1)s}}
			=\V'\int_{K_{n}}
			\M^\vee\left(\M\varphi_{\lambda_{\mathrm{E}}(\frac{1}{2}-s)}\cdot\M\bar{f}\right)
			\cdot f_{s}
		\end{equation*}
		defined on $\Re(s)=0$.
		Thus, recalling \eqref{eq:M11-before-magic}, we reduce to showing that,
		for $s$ in general position, we have
		\begin{equation*}
			\int_{K_{n}}\M^\vee\left(\M\varphi_{\lambda_{\mathrm{E}}(\frac{1}{2}-s)}\cdot\M\bar{f}\right)\cdot f_s
			=\frac{2\V}{\V'}
			\P(\lambda_{\mathrm{E}}(\tfrac{1}{2}-s);\varphi,f,f_s),
		\end{equation*}
		where, we recall, $\mathcal{V}$ is given by \eqref{def-mathcal-V}.
		We assume that $\varphi\in\B$ above are factorizable. The proof now
		follows by applying Lemma \ref{lem:global-period-equality-M11} and the identity
		\[
		\frac{2\V}{\V'}=\frac{\vol([G_{n-2}]^1)}{\vol([G_{n-1}]^1)}
		=\Delta^{-\frac{n-1}2}\xi(n-1)^{-1},
		\]
	which is a direct consequence of \eqref{def-mathcal-V} and \eqref{volume-Gk1}.
\end{proof}

\begin{lem}\label{lem:global-period-equality-M11}
	We have
	\begin{equation*}
		\int_{K_{n}}\M^\vee\left(\M\varphi_{\lambda_{\mathrm{E}}(\frac{1}{2}-s)}\cdot\M\bar{f}\right)\cdot f_{s}\\
		=\Delta^{-\frac{n-1}{2}}\xi(n-1)^{-1}\P\left(\lambda_{\mathrm{E}}(\tfrac{1}{2}-s);\varphi,f,f_s\right)
	\end{equation*}
	for $s$ in general position.
\end{lem}

\begin{proof}
	We abbreviate $\lambda_{\mathrm{E}}(\tfrac{1}{2}-s)$ as $\lambda$. We recall the definition of $\P$ from \eqref{eq:main-RS-period}. We fold the central average and use Iwasawa coordinates for $h\in Z_2N_2\backslash G_2$, so that we can write the period
	\begin{equation*}
		\int_{K_n}\int_{\A^\times}\W\left(\varphi_\lambda\right)\cdot\overline{\W(f)}\cdot f_{s}\left[\begin{pmatrix}
			\mathrm{I}_{n-2} \\&y\\&&1
		\end{pmatrix}k\right]\,|y|^{n-3}d^\times y\,dk.
	\end{equation*}
	Note that from the proof of Lemma \ref{lem:factorization-period} it follows that for $\Re(s)$ sufficiently large the above converges absolutely.
	We claim that the above equals
	\begin{equation*}
		\Delta^{\frac{n-1}{2}}\xi(n-1)\int_{K_n}\M^\vee\left(\M\varphi_\lambda\cdot\M\bar{f}\right)\cdot f_{s}.
	\end{equation*}
	We deform $f$ to $f_{s'}$ with $\Re(s')$ sufficiently negative and prove the claim for the deformed period. The original claim follows by meromorphicity at $s'=0$.

	For $\Re(s)$ and $-\Re(s')$ sufficiently large we write
	\begin{equation*}
		\M^\vee\left(\M\varphi_\lambda\cdot\M\bar{f}_{s'}\right)=\int_{\A^{n-1}}\M \varphi_\lambda\cdot \M \bar{f}_{s'}\left[\begin{pmatrix}&1\\\mathrm{I}_{n-1}&\end{pmatrix}\begin{pmatrix}\mathrm{I}_{n-1}&x\\&1\end{pmatrix}\cdot\right]\, dx
	\end{equation*}
	as an absolutely convergent integral. Applying Plancherel
	for $L^2(\A^{n-1})$ we write the above as
	\begin{multline*}
		\int_{(\A^{n-1})^\top}\left(\int_{\A^{n-1}}\M\varphi_\lambda
		\left[\begin{pmatrix}&1\\\mathrm{I}_{n-1}&\end{pmatrix}
		\begin{pmatrix}\mathrm{I}_{n-1}&x_1\\&1\end{pmatrix}\cdot\right]
		\overline{\psi_0(x'x_1)}\, dx_1\right)\\
		\times\left(\int_{\A^{n-1}}\M\bar{f}_{s'}\left[\begin{pmatrix}&1\\\mathrm{I}_{n-1}&\end{pmatrix}
		\begin{pmatrix}\mathrm{I}_{n-1}&x_2\\&1\end{pmatrix}\cdot\right]
		{\psi_0(x'x_2)}\, dx_2\right)\, dx'.
	\end{multline*}
	Note that the $x_i$-integrals converge absolutely as $\Re(s)$ and $-\Re(s')$ are sufficiently large. On the other hand, the $x'$-integrand decays rapidly due to the smoothness of $\varphi_\lambda$ and $f_{s'}$.

	Let $\Phi(x')$ be the integrand above. Using the last formula in the proof of
	Lemma \ref{lem:res-deg-eis-original}, we see that
	\[
	\Delta^{\frac{n-1}{2}}\xi(n-1)\int_{(\A^{n-1})^\top}\Phi(x')\,dx'=\int_{K_{n-1}}\int_{\A^\times}\Phi(ye_{n-1}k)\,|y|^{n-1}d^\times y\,dk.
	\]
	Thus changing variables $x_j\mapsto \left(\begin{smallmatrix}
		\mathrm{I}_{n-2}\\&y
	\end{smallmatrix}\right) x_j$ and employing left equivariance of $\M\varphi_\lambda$ and $\M f_{s'}$ we obtain the above equals
	\begin{equation*}
		\int_{K_{n-1}}\int_{\A^\times}\W_1(\M \varphi_\lambda)
		\ol{\W_1(\M f_{s'})}\left[\begin{pmatrix}\mathrm{I}_{n-2}\\&y\\&&1\end{pmatrix}\begin{pmatrix}k\\&1\end{pmatrix}\cdot\right]\, |y|^{n-\frac{5}{2}+s-s'}d^\times y,
	\end{equation*}
	where for $\varphi$ being $\M \varphi_\lambda$ or $\M f_{s'}$ we define
	\begin{equation*}
		\W_1(\varphi):=\int_{\A^{n-1}}\varphi\left[\begin{pmatrix}&1\\\mathrm{I}_{n-1}&\end{pmatrix}\begin{pmatrix}\mathrm{I}_{n-1}&x\\&1\end{pmatrix}\cdot\right]\overline{\psi_0(e_{n-1}x)}\, dx.
	\end{equation*}
	We evaluate the above at $k\in K_n$, multiply it by $f_s(k)$, use the left equivariance of $f_s$ and integrate over $K_n$. Appealing once again to the unimodularity of $K_n$ and since $\vol(K_{n-1})=1$, we see that
	\begin{equation*}
		\int_{K_n}\M^\vee\left(\M \varphi_\lambda\cdot\M \bar{f}_{s'}\right)\cdot f_{s}=\int_{K_n}\int_{\A^\times}\W_1\left(\M \varphi_\lambda\right)\cdot\overline{\W_1(\M f_{s'})}\cdot f_{s}\left[\begin{pmatrix}
			\mathrm{I}_{n-2} \\&y\\&&1
		\end{pmatrix}k\right]\,|y|^{n-3}d^\times y\,dk.
	\end{equation*}
	Thus it suffices to show that $\W_1(\M \varphi_\lambda)=\W(\varphi_\lambda)$ and $\W_1(\M f_{s'})=\W(f_{s'})$. Since the proofs are the same in both cases we focus on $\varphi_\lambda$, which we abbreviate as $\varphi$ in the rest of the proof.

	Denote $\M_1:=\M_{\1_1\times\1_{n-2}\times\1_1}\left(\left(\begin{smallmatrix}&\mathrm{I}_{n-2}&\\1&&\\&&1\end{smallmatrix}\right),0\right)$ and $\M_2:=\M_{\1_{n-2}\times\1_1\times\1_1}\left(\left(\begin{smallmatrix}\mathrm{I}_{n-2}&&\\&&1\\&1&\end{smallmatrix}\right),0\right)$, that is (locally or globally defined),
	\begin{multline}\label{eq:def-two-intertwiners}
		\M_1\varphi:=\int_{\A^{n-2}}\varphi\left[\begin{pmatrix}&1&\\\mathrm{I}_{n-2}&&\\&&1\end{pmatrix}\begin{pmatrix}\mathrm{I}_{n-2}&x&\\&1&\\&&1\end{pmatrix}\cdot\right]\, dx,\\ \M_2\varphi:=\int_\A\varphi\left[\begin{pmatrix}\mathrm{I}_{n-2}&&\\&&1\\&1&\end{pmatrix}\begin{pmatrix}\mathrm{I}_{n-2}&&\\&1&x\\&&1\end{pmatrix}\cdot\right]\, dx,
	\end{multline}
	whenever they converge absolutely, otherwise by meromorphic continuation. Let
	$$
	\tilde\vphi:=\varphi\left[\begin{pmatrix}\mathrm{I}_{n-2}&\\&\bullet\end{pmatrix}\right].
	$$
	Then it is clear that $\tilde{\varphi}$ can be seen as an induced vector in a deformation of $\1_1\boxplus\1_1$ and by definition $\W(\vphi)$ equals the $\GL_2$ Jacquet intertwiner applied to $\tilde\vphi$, which, by classical results on $\GL_2$ Eisenstein series, coincides with the Whittaker period applied to the Eisenstein series $\Eis(\tilde\vphi)$. By the functional equation of Eisenstein series, we have that $\Eis(\tilde \vphi)=\Eis(\M_{\1_1\times\1_1}\tilde \vphi)$. Observing that $\M_{\1_1\times\1_1}\tilde \vphi=\M_2\varphi\left[\left(\begin{smallmatrix}\mathrm{I}_{n-2}&\\&\bullet\end{smallmatrix}\right)\right]$ and reversing the steps we obtain that $\W(\vphi)=\W(\M_2\vphi)$. Now using that
	\begin{equation*}
		\begin{pmatrix}&1\\\mathrm{I}_{n-1}&\end{pmatrix}\begin{pmatrix}\mathrm{I}_{n-1}&x\\&1\end{pmatrix}\\
		=\begin{pmatrix}&1&\\\mathrm{I}_{n-2}&&\\&&1\end{pmatrix}\begin{pmatrix}\mathrm{I}_{n-2}&\tilde{x}&\\&1&\\&&1\end{pmatrix}\begin{pmatrix}\mathrm{I}_{n-2}&&\\&&1\\&1&\end{pmatrix}\begin{pmatrix}\mathrm{I}_{n-2}&&\\&1&x_{n-1}\\&&1\end{pmatrix}
	\end{equation*}
	for $x=(\tilde{x},x_{n-1})\in \A^{n-2}\times \A$ we write
	\begin{align*}
		\W_1(\M \varphi) & = \int_\A\M_1\left(\M \varphi\right)\left[\begin{pmatrix}\mathrm{I}_{n-2}&&\\&&1\\&1&\end{pmatrix}\begin{pmatrix}\mathrm{I}_{n-2}&&\\&1&x\\&&1\end{pmatrix}\right]\overline{\psi_0(x)}\, dx \\
		& =  \W(\M_1\!\circ\!\M\,\vphi).
	\end{align*}
	Now by the functional equation of intertwining operators (see \eqref{eq:FE-for-intertwiner}), we see that $\M_1\circ\M=\M_2$. This concludes the proof.
\end{proof}

\begin{rmk}\label{rmk:local-period-equality-M11}
	First, on the line $\Re(s)=0$ using unitarity of $\M$ and then by meromorphic continuation for general $s$ we write Lemma \ref{lem:global-period-equality-M11} as
	\begin{equation*}
		\int_{K_{n}}\M\varphi_{\lambda_{\mathrm{E}}(\frac{1}{2}-s)}\cdot\M\bar{f}\cdot \M f_{s}\\
		=\Delta^{-\frac{n-1}{2}}\xi(n-1)^{-1}\P\left(\lambda_{\mathrm{E}}(\tfrac{1}{2}-s);\varphi,f,f_s\right)
	\end{equation*}
	for $s$ in general position. Now let $v$ be a fixed place of $F$. Choosing $f$ and $\varphi$ to be normalized spherical everywhere but $v$ and using the unramified computation in \S\ref{sec:unramified-local-choice} we obtain that
	\begin{equation*}
		\int_{K_{n,v}}\M_v^\vee\left(\M_v\varphi_{\lambda_{\mathrm{E}}(\frac{1}{2}-s),v}\cdot\M\bar{f}_v\right)\cdot f_{v,s}\\
		=c_s\P_v\left(\lambda_{\mathrm{E}}(\tfrac{1}{2}-s);\varphi,f_v,f_{v,s}\right)
	\end{equation*}
	where $c_s\neq 0$ is an $s$-dependent scalar.
\end{rmk}

\subsection{Comparing $M_{01}$ to an inner product}
Similarly to the previous section, we prove the following result.
\begin{prop}\label{prop:M01-magic}
	Recall $M_{01}(s)$ from \eqref{eq:M01-before-magic}. We have
	\begin{equation*}
		M_{01}(s)=\left\langle \phi, \Eis\left( f\cdot \mathbf{M} \ol{f_s}\right)\right\rangle_{[\bar{G}_n]}
	\end{equation*}
	where $\M$ is as in \eqref{eq:def-main-intertwiner}.
\end{prop}

\begin{proof}
	We proceed similarly to the proof of Proposition \ref{prop:M11-magic}. Recall the intertwiner $\M_1$ from \eqref{eq:def-two-intertwiners}. Note that
	\begin{equation*}
		\beta(f,s):=\exp\left\langle-\lambda_{\mathrm{L}}(s),H_{\widetilde{\Q_{n-2}}}(\cdot)\right\rangle\times \M_1\left(\bar{f}\cdot \M f_s\right)\in\1_{n-2}\boxplus\1_1\boxplus\1_1
	\end{equation*}
	where $\lambda_{\mathrm{L}}(s)$ is as in \eqref{eq:def-lambda-L}. Thus spectrally expanding and using functional equation of the Eisenstein series \eqref{eq:FE-for-Eis} we write
	\begin{equation*}
		\Eis\left( \bar{f}\cdot\M {f_s}\right)=\sum_{\varphi\in\B(\1_{n-2}\boxplus\1_1\boxplus\1_1)}\Eis(\bar{\varphi},\lambda_{\mathrm{L}}(s))\left\langle\beta(f,s),\bar{\varphi}\right\rangle_{\1_{n-2}\boxplus\1_1\boxplus\1_1}.
	\end{equation*}
	For $\Re(s)$ sufficiently negative, it follows that by the invariance of $\bar{f}$ and the functional equation of intertwining operators (see \eqref{eq:FE-for-intertwiner}), one has
	\begin{equation*}
		\M_1\left(\bar{f}\cdot\M f_s\right)=\bar{f}\cdot\left(\M_1\circ\M \right)f_s=\bar{f}\cdot\M_2 f_s,
	\end{equation*}
	where $\M_2$ is the intertwiner in \eqref{eq:def-two-intertwiners}. Now by analytic continuation the above holds for $s$ in general position. By adjointness of $\M_2$ for the bilinear pairing $\int_{K_n}$, valid for all $s$ by meromorphic continuation,
	$$
	\langle \beta(f,s),\overline{\vphi}\rangle_{\1_{n-2}\boxplus\1_1\boxplus\1_1}=\V''\int_{K_n} \vphi_{-\lambda_{\mathrm{L}}(s)}\cdot\bar{f}\cdot \M_2 f_s=\V''\int_{K_n} \M_2\left(\vphi_{-\lambda_{\mathrm{L}}(s)}\cdot\bar{f}\right)\cdot  f_s,
	$$
	where $\V''=\vol([M_{\widetilde{\Q_{n-2}}}]^1)$.
	Thus, just like in the proof of Proposition \ref{prop:M11-magic}, it suffices to show that
	\begin{equation*}
		\left\langle\beta(f,s),\bar{\varphi}\right\rangle_{\1_{n-2}\boxplus\1_1\boxplus\1_1}=2\V\xi^\ast(1)\xi(\tfrac{n}{2})\xi(\tfrac {n-2}2+ns)\widetilde{\P}_S(s,\tfrac n2s-\tfrac12)(\varphi).
	\end{equation*}
	We apply Lemma \ref{lem:global-period-equality-M01} below
	and recall the definition \eqref{eq:def-lambda-L} of $\lambda_{\mathrm L}(s)$.
	We conclude by observing that
	$$
	2\V\times \xi^\ast(1)
	=\Delta^{-1/2}\V'',
	$$
	which is a direct consequence of \eqref{def-mathcal-V} and \eqref{volume-Gk1}.
\end{proof}

\begin{lem}\label{lem:global-period-equality-M01}
	Assume $\vphi$ factorizable. We have
	\begin{equation*}
			\int_{K_{n}}\M_2\left(\varphi_{-\lambda_{\mathrm{L}}(s)}\cdot\bar{f}\right)\cdot f_{s}
			=\Delta^{-1/2}\xi(\tfrac n2)\xi(\tfrac {n-2}2+ns)\tilde\P_S(s,\tfrac n2s-\tfrac12)(\vphi)
	\end{equation*}
	for $s$ in general position.
\end{lem}

\begin{proof}
	For $\Re(s)$ sufficiently positive, we use the factorization of intertwiners
	to show that
	\begin{equation*}
		\int_{K_n}\M_2\left(\varphi_{-\lambda_{\mathrm{L}}(s)}\cdot\bar{f}\right)\cdot f_s                         \\
		=\prod_v\int_{K_{n,v}}\M_{2,v}\left(\varphi_{-\lambda_{\mathrm{L}}(s),v}\cdot\bar{f_v}\right)\cdot f_{v,s}.
	\end{equation*}
	Recalling the definition of $\M_{2,v}$, we see that
	\begin{equation*}
		\M_{2,v}\left(\varphi_{-\lambda_{\mathrm{L}}(s),v}\cdot\bar{f_v}\right)=\int_{F_v}\varphi_{-\lambda_{\mathrm{L}}(s),v}\cdot\bar{f_v}\left[\begin{pmatrix}\mathrm{I}_{n-2}&&\\&&1\\&1&\end{pmatrix}\begin{pmatrix}\mathrm{I}_{n-2}&&\\&1&x\\&&1\end{pmatrix}\cdot\right]\, dx.
	\end{equation*}
	As in the proof of Lemma \ref{lem:global-period-equality-M11}, we apply Parseval,
	this time in $L^2(F_v)$, and make a change of variables inside the integrals
	defining the Fourier transforms. Combining these steps, observing our
	normalization of measures from \S \ref{subsec:chars-and-measures} and
	using equivariance of the involved vectors, we write the above as
	\begin{equation*}
		\Delta_v^{-1/2}\zeta_v(1)^{-1}\int_{F
			_v^\times}\W\left(\varphi_{-\lambda_{\mathrm{L}}(s),v}\right)\cdot\overline{\W(f_v)}\left[\begin{pmatrix}\mathrm{I}_{n-2}&&\\&y&\\&&1\end{pmatrix}\cdot\right]|y|^{n-\frac 52 +s}\, d^\times y.
	\end{equation*}
	The result follows by plugging the above expression at $k\in K_{n,v}$,
	integrating and invoking the unramified computation from \eqref{eq:unramified-eis}.
\end{proof}

\section{Proof of Theorem \ref{thm:spectral-expansion}}\label{sec:proof-of-first-main-thm}

We now have most of the ingredients to prove Theorem \ref{thm:spectral-expansion}. In this section we complete the proof of this theorem.

\begin{prop}\label{prop:reg-main-term}
	Let $f_s$ be as before and let $\M$ be as in \eqref{eq:def-main-intertwiner}. Then $\Eis(\bar{f}\cdot f_s)+\Eis(\M\bar{f}\cdot \M f_s)$ is holomorphic at $s=0$.
\end{prop}

\begin{proof}
	Note that $\bar{f}\cdot f_s$ is a holomorphic section in $\1_{n-1}\cdot|\det|^{\frac{1}{2}+s}\boxplus\1_1\cdot|\cdot|^{-(n-1)(\frac{1}{2}+s)}$. Invoking Lemma \ref{lem:res-deg-eis-original} we see that $\Eis(\bar{f}\cdot f_s)$ has a simple pole at $s=0$ and its residue there is given by
	\begin{equation*}
		\frac{\Delta^{-\frac{n-1}{2}}}{n}\frac{\xi^\ast(1)}{\xi(n)}\cdot\int_{K_n}|f|^2(k)\, dk=c\|f\|^2,
	\end{equation*}
	where $c$ is some constant independent of $f$.
	On the other hand, $\M\bar{f}\cdot\M f_s$ is a holomorphic section around $s=0$ lying in
	$\1_1\cdot|\cdot|^{(n-1)(\frac{1}{2}-s)}\boxplus\1_{n-1}\cdot|\det|^{s-\frac{1}{2}}$. Invoking Lemma \ref{lem:res-deg-eis-dual} we see that $\Eis(\M\bar{f}\cdot \M f_s)$ has a simple pole at $s=0$ and its residue there is given by
	\begin{equation*}
		- c\|\M f\|^2 = -c\| f\|^2,
	\end{equation*}
	where $c$ is the same constant as above and we used that $\M$ is unitary on unitary representations. Hence, the residues cancel each other, which completes the proof.
\end{proof}

\begin{lem}\label{lem:res-eis-middle}
	Let $\Q:=\Q_1\cap\Q_{n-1}$, $\lambda:=(\lambda_1,\lambda_2,\lambda_3)\in\a^\ast_{\Q,\C}$, and $\varphi_\lambda$ be a holomorphic section in $\1_1\cdot|\cdot|^{\lambda_1}\boxplus\1_{n-2}\cdot|\det|^{\lambda_2}\boxplus\1_1\cdot|\cdot|^{\lambda_3}$. Then for $\lambda_2$ in a generic position $\Eis(\varphi_\lambda)$ has a simple polar divisor along the hyperplane $\lambda_1-\lambda_3=1$ where its residue is given by
	\begin{equation*}
		{\Delta^{-\frac{1}{2}}}\frac{\xi^\ast(1)}{\xi(2)}\cdot\Eis\left(\beta_{\varphi_{\lambda}}\right)
	\end{equation*}
	where $\beta_{\varphi_\lambda}\in\1_{n-2}\cdot|\det|^{\lambda_2}\boxplus\1_2\cdot|\cdot|^{\frac{\lambda_1+\lambda_3}{2}}$ is given by
	\begin{equation*}
		\beta_{\varphi_\lambda}(g)=\int_{K_2}{\M}_1\varphi_\lambda\left[\begin{pmatrix}\mathrm{I}_{n-2}&\\&k\end{pmatrix}g\right]\, dk
	\end{equation*}
	for $\lambda_1-\lambda_3=1$ and $\M_1$ is as in \eqref{eq:def-two-intertwiners}.
\end{lem}

\begin{proof}
	We follow a similar strategy to that in the proof of Lemma \ref{lem:eis-residue}. Temporarily, assume that $\lambda_2$ is in generic position and $(\lambda_2,\lambda_1,\lambda_3)\in\a^\ast_{\widetilde{\Q_{n-2}},\C}$ is sufficiently dominant. Then using the functional equation \eqref{eq:FE-for-Eis} we write
	\begin{equation*}
		\Eis(\varphi_\lambda) = \Eis(\M_1\varphi_\lambda)=\sum_{\gamma\in\widetilde{\Q_{n-2}}(F)\bs G_n(F)}\M_1\varphi_\lambda(\gamma\cdot)
	\end{equation*}
	as an absolutely convergent sum, where $\M_1$ is as in \eqref{eq:def-two-intertwiners}. Writing $\widetilde{\Q_{n-2}}\bs G_n$ as $\widetilde{\Q_{n-2}}\bs\Q_{n-2}\times\Q_{n-2}\bs G_n$ we write the right hand side above as
	\begin{equation*}
		\sum_{\gamma\in\Q_{n-2}(F)\bs G_n(F)}\Eis^{\Q_{n-2}}(\M_1\varphi_\lambda)(\gamma\cdot).
	\end{equation*}
	Moreover, we realize
	\begin{equation*}
		\Eis^{\Q_{n-2}}(\M_1\varphi_\lambda)(g)= \Eis_{B_2}\left(\M_1\varphi_\lambda\left[\begin{pmatrix}\mathrm{I}_{n-2}&\\&\bullet\end{pmatrix}g\right]\right)(1),
	\end{equation*}
	and note
	\begin{equation*}
		\M_1\varphi_\lambda\left[\begin{pmatrix}\mathrm{I}_{n-2}&\\&\bullet\end{pmatrix}g\right]\in|\cdot|^{\lambda_1-\frac{n-2}{2}}\boxplus|\cdot|^{\lambda_3-\frac{n-2}{2}}.
	\end{equation*}
	We compute using \cite[eq.(4.7)-(4.10)]{fine-boi} that
	\begin{equation*}
		\M_1\varphi_\lambda= \frac{\xi(\frac{n-1}{2}+\lambda_2-\lambda_1)}{\xi(\frac{n-1}{2}+\lambda_1-\lambda_2)}\tilde{\M}_1\varphi_\lambda
	\end{equation*}
	where $\tilde{\M}_1\varphi_\lambda$ is holomorphic for $\Re(\lambda_1-\lambda_2)\ge 0$; see \cite[Theorem 4.2 (3)]{fine-boi}. As $\lambda_2$ is in generic position $\M_1\varphi_\lambda$ is holomorphic around $\lambda_1-\lambda_3=1$. Thus from the standard $\GL_2$ theory or using Lemma \ref{lem:res-deg-eis-original} we know that the above Eisenstein series has a simple pole along the hyperplane $\lambda_1-\lambda_3=1$ where its residue is given by ${\Delta^{-\frac{1}{2}}}\frac{\xi^\ast(1)}{\xi(2)}\cdot\frac{\xi(\frac{n-1}{2}+\lambda_2-\lambda_1)}{\xi(\frac{n-1}{2}+\lambda_1-\lambda_2)}\cdot\tilde{\beta}_{\varphi_\lambda}(g)$ where $\tilde{\beta}_{\varphi_\lambda}\in\1_{n-2}\cdot|\det|^{\lambda_2}\boxplus\1_2\cdot|\cdot|^{\frac{\lambda_1+\lambda_3}{2}}\big\vert_{\lambda_1-\lambda_3=1}$ is given by
	\begin{equation*}
		\tilde{\beta}_{\varphi_\lambda}(g)=\int_{K_2}\tilde{\M}_1\varphi_\lambda\left[\begin{pmatrix}\mathrm{I}_{n-2}&\\&k\end{pmatrix}g\right]\, dk\Bigg\vert_{\lambda_1-\lambda_3=1}.
	\end{equation*}
	We conclude by meromorphic continuation.
\end{proof}

\begin{rmk}\label{rem:upgrade-n=4}
	In practice, we need a slight upgrade of Lemma \ref{lem:res-eis-middle}. Namely, instead of keeping $\lambda_2$ at a generic position we need to compute the residue there as $\lambda\to(\tfrac{1}{2},0,-\tfrac{1}{2})$. For $n\neq 4$ using the continuity of the $\Eis$ intertwiner we obtain that the residue in Lemma \ref{lem:res-eis-middle} equals
	\begin{equation*}
		\Delta^{-1/2}\frac{\xi^\ast(1)\xi(\frac{n-2}{2})}{\xi(2)\xi(\frac{n}{2})}\cdot\Eis\left(\tilde\beta_{\varphi_{(\frac{1}{2},0,-\frac{1}{2})}}\right).
	\end{equation*}
	However, for $n=4$ we see that $\M_1\varphi_\lambda$ has a simple polar
	divisor along the hyperplane $\lambda_1-\lambda_2=\tfrac{1}{2}$. On the
	other hand, $\tilde\beta_{\varphi_\lambda}$ is a holomorphic section around
	$(\tfrac{1}{2},0,-\tfrac{1}{2})$ and thus by \cite[Theorem 4.17 (2)]{fine-boi}
	we deduce that $\Eis\left(\tilde\beta_{\varphi_\lambda}\right)$ has a zero when
	$\lambda_2=\tfrac{\lambda_1+\lambda_3}{2}$. Hence, the residue in Lemma
	\ref{lem:res-eis-middle} is regular at $(\tfrac{1}{2},0,-\tfrac{1}{2})$
	also in this case.
\end{rmk}

\begin{prop}\label{prop:reg-second-main-term}
	Let $f_s$ be as before. Then $\Eis(\M\bar{f}\cdot f_s)+\Eis(\bar{f}\cdot \M f_s)$ is holomorphic at $s=0$.
\end{prop}

\begin{proof}
	Let $s'\in\C$ be at a generic position. We readily check that
	\begin{equation*}
		\bar{f}_{s'}\cdot\M f_s\in\1_1\cdot|\cdot|^{\frac{1}{2}+s'-(n-1)s}\boxplus\1_{n-2}\cdot|\det|^{s'+s}\boxplus\1_1\cdot|\cdot|^{-\frac{1}{2}-(n-1)s'+s}
	\end{equation*}
	and apply Lemma \ref{lem:res-eis-middle} and the discussion afterwards to obtain that $\Eis(\bar{f}_{s'}\cdot\M f_s)$ has a simple pole at $s=s'$ and its residue is given by $-d\cdot \Eis(\beta_{s'})$ for some $d>0$ and
	\begin{equation*}
		\beta_{s'}:=\int_{K_2}{\M}_1\left(\bar{f}_{s'}\cdot \M f_{s'}\right)\left[\begin{pmatrix}\mathrm{I}_{n-2}&\\&k\end{pmatrix}\cdot\right]\, dk.
	\end{equation*}
	Note that due to the left invariance of $f$ and functional equation of intertwiners \eqref{eq:FE-for-intertwiner}
	we have
	\begin{equation*}
		\M_1\left(\bar{f}_{s'}\cdot\M f_{s'}\right)=\bar{f}_{s'}\cdot(\M_1\circ\M) f_{s'} = \bar{f}_{s'}\cdot\M_2 f_{s'},
	\end{equation*}
	where $\M_2$ is as defined in \eqref{eq:def-two-intertwiners}.
	Thus from \cite[eq.(4.7)-eq.(4.10)]{fine-boi} it follows that the above integral evaluates to
	\begin{equation*}
		\beta_{s'}=\frac{\xi(\frac{n-2}{2}+ns')}{\xi(\frac{n}{2}+ns')}\tilde{\beta}_{s'},\quad \tilde{\beta}_{s'}:=\int_{K_2}\bar{f}_{s'}\cdot\tilde{\M}_2 f_{s'}\left[\begin{pmatrix}\mathrm{I}_{n-2}&\\&k\end{pmatrix}\cdot\right]\, dk.
	\end{equation*}
	Similarly,
	\begin{equation*}
		\M\bar{f}_{s'}\cdot f_s\in\1_1\cdot|\cdot|^{\frac{1}{2}-(n-1)s'+s}\boxplus\1_{n-2}\cdot|\det|^{s'+s}\boxplus\1_1\cdot|\cdot|^{-\frac{1}{2}+s'-(n-1)s}
	\end{equation*}
	is a holomorphic section around $s=s'$ and thus $\Eis(\M\bar{f}_{s'}\cdot f_s)$ has a simple
	pole at $s=s'$ with residue $d\cdot\Eis(\beta^\vee_{s'})$ where
	\begin{equation*}
		\beta^\vee_{s'}=\frac{\xi(\frac{n-2}{2}+ns')}{\xi(\frac{n}{2}+ns')}\tilde{\beta}^\vee_{s'},\quad\tilde{\beta}^\vee_{s'}:=\int_{K_2}\tilde{\M}_2\bar{f}_{s'}\cdot f_{s'}\left[\begin{pmatrix}\mathrm{I}_{n-2}&\\&k\end{pmatrix}\cdot\right]\, dk.
	\end{equation*}
	Hence, it suffices to show that $\Eis\left(\beta_{s'}\right)
	=\Eis\left(\beta^\vee_{s'}\right)$ at $s'=0$, which will follow from the
	regularity of the two sides at $s'=0$ (see Remark \ref{rem:upgrade-n=4})
	and $\tilde{\beta}_{s'}=\tilde{\beta}^\vee_{s'}$. We will show
	the last equality locally. Indeed, for any place $v$ we check that

	\begin{align*}
		\tilde{\beta}_{s',v}(g) & =\int_{K_{2,v}}\bar{f}_v\cdot \tilde{\M}_{2,v}
		f_v\left[\begin{pmatrix}\mathrm{I}_{n-2}&\\&k\end{pmatrix}g\right]
		|\det k|^{-\frac{n-2}{2}-(n-2)s'}\, dk\\
		& =\int_{B_{2,v}\bs G_{2,v}}\left(|\det|^{-\frac{n-2}{4}-\frac{n-2}{2}s'}
		\cdot g\bar{f}_v\right)
		\cdot \tilde{\M}_{2,v}\left(|\det|^{-\frac{n-2}{4}-\frac{n-2}{2}s'}\cdot gf_v\right)
		\left[\begin{pmatrix}\mathrm{I}_{n-2}&\\&\bullet\end{pmatrix}\right].
	\end{align*}
	The right hand side above defines a $G_{2,v}$-equivariant bilinear pairing between
	$|\cdot|_v^{\frac{n-2}{4}+\frac n2s'}\boxplus|\cdot|_v^{-\frac{n-2}{4}-\frac n2s'}$
	and its dual. Similarly, $\tilde{\beta}^\vee_{s',v}(g)$
	also defines a $G_{2,v}$-equivariant bilinear pairing on the same pair of vectors.
	For $n\neq 4$ and $s'=0$, we have,
	by Schur's lemma, that $\tilde{\beta}_{0,v}(g)$ equals $\tilde{\beta}^\vee_{0,v}(g)$ up to a multiplicative $f,g$-independent scalar. When $n=4$,
	since $|\cdot|_v^{\frac{n-2}{4}}\boxplus|\cdot|_v^{-\frac{n-2}{4}}$ is not
	irreducible, we cannot directly apply Schur's lemma at $s'=0$. However,
	applying it for $s'\neq 0$ and letting $s'\to 0$ we deduce the equality
	up to constants by continuity.
	Checking the equality when $g=1$ and $f_v$ is normalized spherical, we conclude the proof.
\end{proof}

\begin{proof}[Proof of Theorem \ref{thm:spectral-expansion}]
	Let $s\in\C$ with sufficiently large $\Re(s)$. We start with Proposition \ref{prop:whittaker-cleanup}, and the definitions \eqref{eq:def-I1} and \eqref{eq:def-I2}. We obtain
	\begin{equation*}
		\int_{[\bar{G}_n]}\phi\cdot\overline{\Eis(f)}\cdot\Eis(f_s) = I_1(s)+I_2(s) = M_{00}(s)+M_{10}(s) + I_2^{\mathrm{D}}(s) + I_2^{\mathrm{C}}(s)
	\end{equation*}
	where the second equality follows from Proposition \ref{prop:n-1-case} and Proposition \ref{prop:expansion-after-Whittaker}.  Next we do meromorphic continuation of $I_2^{\mathrm{D}}(s)$ and $I_2^{\mathrm{C}}(s)$ to $0<\Re(s)<1/n$ using Proposition \ref{prop:cont-I2D} and Proposition \ref{prop:cont-I2C} and write
	\begin{multline*}
		\int_{[\bar{G}_n]}\phi\cdot\overline{\Eis(f)}\cdot\Eis(f_s) = M_{00}(s)+M_{10}(s) +M_{01}(s)+M_{11}(s)\\
		+\V\sum_{\pi\in\Pi_{\mathrm{c}}(\bar{G}_2)}\frac{L^S\left(\tfrac{n-1}{2}+\tfrac{n}{2}s,\pi\right)L^S\left(\tfrac{1}{2}+\tfrac{n}{2}s,\pi\right)}
		{\sqrt{L^S(1,\pi,\Ad)}}\\
		\sum_{\varphi\in\B\left(\1_{n-2}\boxplus\pi\right)}
		\left\langle\phi,\Eis(\varphi, \ol{\lambda_{\mathrm{D}}(s)})\right\rangle_{[\bar{G}_n]}
		\P_S(-\lambda_{\mathrm{D}}(s);\varphi,f,f_s)
		\\
		+\V\sum_{\chi\in\Pi(G_1)}\frac{\prod_{\pm}L^S\left(\tfrac{n-1}{2}+\tfrac{n}{2}s\pm z,\chi^\pm\right)
			L^S\left(\tfrac{1}{2}+\tfrac{n}{2}s\pm z,\chi^\pm\right)}{L^S(1+2z,\chi^2)}\\\sum_{\varphi\in\B(\1_{n-2}\boxplus\chi\boxplus\chi^{-1})}\intop_{i\mathbb{R}}
		\left\langle \phi,\Eis(\varphi,\overline{\lambda_{\mathrm{C}}(s,z)})\right\rangle_{[\bar{G}_n]}
		\P_S(-\lambda_{\mathrm{C}}(s,z);\varphi,f,f_s)
		\, dz.
	\end{multline*}
	Finally, we apply Proposition \ref{prop:M11-magic} and Proposition \ref{prop:reg-main-term} to conclude that $M_{00}+M_{11}$ is regular at $s=0$. Similarly, Proposition \ref{prop:M01-magic} and Proposition \ref{prop:reg-second-main-term} yield that $M_{01}+M_{10}$ is regular at $s=0$. The remaining terms are continuous as $s\to0$ from $\Re(s)>0$. Taking this limit, we conclude by abbreviating
	\begin{equation}\label{def-main-PSf-disc}
		\P_S(\vphi,f):=\P_S(0;\varphi,f,f)=\prod_{v\in S}\P_v(0;\varphi_v,f_v,f_v),
	\end{equation}
	and
	\begin{equation}\label{def-main-PSf-cont}
		\P_S(z,\vphi,f):=\P_S((0,z,-z);\varphi,f,f)=\prod_{v\in S}\P_v((0,z,-z);\varphi_v,f_v,f_v),
	\end{equation}
	where $\P_v(\cdots)$ is given by \eqref{eq:main-RS-period-local}.
	\end{proof}

\part{Moment Asymptotics}

\section{Local Choices}\label{sec:choices}

We recall the number field $F$ and the additive character $\psi_0$ of $F$ as before. For the rest of the paper, let $\q=\prod_{v<\infty} \p_v^{r_v}$ be an integral ideal coprime to the different ideal of $F$.  For each $v\mid\infty$ by $\tau_v$ we denote an arbitrarily small but fixed positive real number. The actual value of $\tau_v$ may differ from line to line.

We choose a factorizable $\Phi=\otimes_v \Phi_v\in\mathcal{S}(\A^n)$ as follows:
\begin{itemize}
	\item For all finite $v\nmid \q$, we choose $\Phi_{v}:=\mathbbm{1}_{\o_v^n}$.
	\item For all $v\mid \q$, we let
    \[
      \Phi_v(u_1,\dots,u_{n-1}, u_n):=N(\p_v^{r_v})^{\frac{n-1}{2}}
      \mathbbm{1}_{\p_v^{r_v}}(u_1)
      \cdots \mathbbm{1}_{\p_v^{r_v}}(u_{n-1})\mathbbm{1}_{\o_v^\times}(u_n).
    \]
	\item Finally, for each $v\mid\infty$ we let $\Phi_v$ be a smooth non-negative function with compact support inside the ball of radius $\tau_v$ around $(0,\dots,0,1)$ in $F_v^n$ which is identically $1$ in the ball of radius $\tau_v/2$. We normalize $\Phi_v$ according to \eqref{eq:arch-normalization-f} below.
\end{itemize}
From this $\Phi$ (resp.\ $\Phi_v$) we construct a vector $f_s=f_{s,\Phi}$ (resp.\ $f_{v,s}$) as defined in equation \eqref{eq:def-f-from-phi}. As before, we abbreviate $f_0$ and $f_{v,0}$ as $f$ and $f_v$, respectively.

Note that with this choice, we have an explicit description of $f_{v,s}(k_v)$ for $k_v\in K_{n,v}$ and $v\mid \q$. Writing the bottom row of $k_v$ as $(k_{n1},\dots,k_{nn})$, one sees that
\[
  f_{v,s}(k_v)=N(\p_v^{r_v})^{\frac{n-1}{2}}\int_{F_v^\times}\prod_{j=1}^{n-1}\mathbbm{1}_{\p_v^{r_v}}
  (tk_{nj})\mathbbm{1}_{\o_v^\times}(tk_{nn})|t|^{\tfrac n2+ns}\,d^\times t.
\]
Notice that the last indicator function tells us that one must have
$|t|=|k_{nn}|^{-1}$, which is $\ge 1$. If $|t|>1$ then $|k_{nn}|<1$
and therefore there must exist $i< n$ such that $|k_{ni}|=1$,
but this implies that $|tk_{ni}|=|t|>1$, which means that the integrand vanishes. Hence the integral is supported on $|t|=1$ and from this we deduce that
\begin{equation}\label{eq:test_vector_on_kx}
	f_{v,s}(k_v) =  N(\p_v^{r_v})^{\frac{n-1}{2}}\mathbbm{1}_{K_0\left(\p_v^{r_v}\right)}(k_v).
\end{equation}
We note that a similar computation shows that $f_{v,s}$ is $K_{n,v}$-invariant
for all finite $v\nmid\q$.

On the other hand, if $v\mid\infty$ the choice of $\Phi_v$ ensures that
\begin{equation}\label{eq:supp-f-arch}
	f_v\left[\pmat{\mathrm{I}_{n-1}&\\ \bullet&1}\right]\text{ is supported on a $\tau_v$-radius ball around the origin},
\end{equation}
and
\begin{equation}\label{eq:away-from-zero}
	\int_{F_v^{n-1}}f_v\left[\pmat{\mathrm{I}_{n-1}&\\ c&1}\right]\, dc \gg_{\tau_v} 1
\end{equation}
which follows from the local version of \eqref{eq:def-f-from-phi} and the choice of $\Phi$ above. Finally, we normalize $\Phi_v$ above so that
\begin{equation}\label{eq:arch-normalization-f}
	\int_{F_v^{n-1}} \left \lvert f_v\left[\pmat{\mathrm{I}_{n-1}&\\ c&1}\right]\right\rvert^2 \, dc=\frac{\zeta_v(1)}{\zeta_v(n)}.
\end{equation}
The above is possible as $\Phi_v$ has a positive $L^2$-mass. With our Haar measures this is equivalent to $\int_{K_{n,v}}|f_v(k)|^2\,dk=1$.

For later purposes, we also define $\Phi^0$ to be the factorizable Schwartz--Bruhat function such that $\Phi^0_v=\mathbbm{1}_{\o_v^n}$ at all $v<\infty$ and $\Phi^0_v=\Phi_{v}$, as chosen above, for $v\mid\infty$. Let $f^0_s\defeq f_{s,\Phi^0}$ be the induced vector, constructed from $\Phi^0$ as in \eqref{eq:def-f-from-phi}. Once again, for simplicity we write $f^0$ for $f^0_0$.

We fix $\pi_0$ a cuspidal automorphic representation for $\bar{G}_n(F)$ that is unramified at every finite place and tempered at every archimedean place.
Let $\phi_0\in\pi_0$ be a cusp form with Whittaker function $W_0=\bigotimes_v W_{0,v}$, such that $W_{0,v}$ is the normalized spherical vector $W_{\pi_{0,v}}$ for all $v<\infty$. For each $v\mid\infty$ we choose $W_{0,v}$ so that
\begin{equation}\label{eq:arch-normalization-W}
	\lVert W_{0,v} \rVert_{\pi_{0,v}}=1
\end{equation}
and
\begin{equation}\label{eq:support-arch-whittaker}
	W_{0,v}\left[\pmat{\bullet&\\&1}\right]\text{ is supported on a $\tau_v$-radius ball centered at the identity}
\end{equation}
as an element in $C_c^{\infty} \left(N_{n-1}(F_v)\backslash G_{n-1}(F_v),\psi_v\right)$. Here and in the rest of the paper, for a generic representation $\sigma$ of $G_n(F_v)$ we adopt the unitary inner product on it as defined in \eqref{eq:def-inner-prod-whittaker}.

\section{The First Spectral Expansion}\label{sec:firstspecexp}

Recall the choices of $\phi_0$ and $\Eis(f)$ from \S \ref{sec:choices}. The goal of this section is to prove Proposition \ref{prop:spec_exp_second_moment}, namely, the quantitative spectral expansion of
\begin{equation*}
	\mathrm{vol}^{-1}(K_0(\q))\frac{\zeta_\q(1)\zeta_\q(\frac{n}{2})^2}{\zeta_\q(n)}\int_{[\bar{G}_{n}]}|\phi_0|^2|\Eis(f)|^2.
\end{equation*}
As $\phi_0$ is cuspidal the above converges absolutely.

We first prove Lemma \ref{lem:abstract-rankin-selberg} below. Its proof exists in \cite[\S 4.2]{Jana2020RS}, however with a different set of notations. We recall the relevant notations here and give a bare sketch of the proof.

We realize the above integral as an $[\bar{G}_n]$-inner product of $\phi_0\cdot\Eis(f)$ with itself. Using Parseval (\emph{e.g.}, from \eqref{eq:spectral-decomposition}) we obtain that the above equals
\begin{equation*}
	\mathrm{vol}^{-1}(K_0(\q))\frac{\zeta_\q(1)\zeta_\q(\frac{n}{2})^2}{\zeta_\q(n)}\sum_{P}n_P^{-1}\sum_{\pi\in\Pi^{G_n}(M)}\intop_{i\a^\ast_P/i\a^\ast_G}\sum_{\varphi\in\B(\mathcal{I}(\pi))}\left|\left\langle\phi_0\cdot\Eis(f),\Eis(\varphi,\lambda)\right\rangle\right|^2\,d\lambda.
\end{equation*}
Working as in \cite[Lemma 4.1]{Jana2020RS} we reduce the above automorphic Plancherel integral to the generic subspectrum, equivalently by Langlands classification, to $\pi\in\Pi_{\mathrm{c}}^{G_n}(M)$. Then we appeal to the standard $\GL_n\times\GL_n$ Rankin--Selberg theory of Jacquet--Piatetski-Shapiro--Shalika (see \cite[\S 2.3]{cogdell2007functions}) and sphericality of $\phi_0$ to write
\begin{multline*}
	\mathrm{vol}^{-1}(K_0(\q))\frac{\zeta_\q(1)\zeta_\q(\frac{n}{2})^2}{\zeta_\q(n)}\sum_{\varphi\in\B(\mathcal{I}(\pi))}\left|\left\langle\phi_0\cdot\Eis(f),\Eis(\varphi,\lambda)\right\rangle\right|^2\\=
	\frac{\left|L\left(\frac{1}{2},\pi_0\otimes\tilde{\mathcal{I}(\pi,\lambda)}\right)\right|^2}{\ell(\mathcal{I}(\pi,\lambda))}\prod_{v\mid\q}H_v\left(\mathcal{I}(\pi_v,\lambda)\right)\prod_{v\mid\infty}h_v\left(\mathcal{I}(\pi_v,\lambda)\right)\
\end{multline*}
where for $v\mid \q$ and an irreducible generic unitary representation $\sigma$ of $G_n(F_v)$ we define
\begin{equation}\label{eq:def-H-sigma}
	H_v(\sigma):=\mathrm{vol}^{-1}(K_0(\p_v^{r_v}))\frac{\zeta_v(1)\zeta_v(\frac{n}{2})^2}{\zeta_v(n)}\sum_{W\in\B(\sigma)}\frac{|Z_v(W_{0,v},\ol{W},f_v)|^2}{|L_v\left(\frac{1}{2},\pi_{0,v}\otimes\tilde{\sigma}\right)|^2},
\end{equation}
and for $v\mid\infty$ we define
\begin{equation}\label{eq:def-h-sigma}
	h_v(\sigma):=\sum_{W\in\B(\sigma)}\left|Z_v(W_{0,v},\ol{W},f_v)\right|^2.
\end{equation}
Here $Z_v(W_1,W_2,f_{v,s}):=\Psi_v(\tfrac{1}{2}+s,W_1,W_2,\Phi_v)$ and $\Psi_v$ is the local zeta integral as in \cite[\S 2.3]{cogdell2007functions}. Finally, the harmonic weight $\ell(\pi)$ is defined by the proportionality between the automorphic norm and the product of our local Whittaker norms. By \eqref{eq:def-inner-prod-whittaker} and \cite[eqs.~(6), (24), Lemma 4.1]{JaNu2021reciprocity}, it equals
\[
	\mathcal{L}(\pi)\prod_{v<\infty}\frac{L_v(1,\pi_{v,\mathrm{ur}}\otimes\tilde{\pi}_{v,\mathrm{ur}})}{L_v(1,\pi_v\otimes\tilde{\pi}_v)},
\]
up to a power of the discriminant $\Delta$ depending only on $n$. The local weights below use these same norms. Thus we obtain the following lemma.

\begin{lem}\label{lem:abstract-rankin-selberg}
	Recall the choices of $\phi_0$ and $\Eis(f)$ from \S \ref{sec:choices}. Then
	\begin{multline*}
		\mathrm{vol}^{-1}(K_0(\q))\frac{\zeta_\q(1)\zeta_\q(\frac{n}{2})^2}{\zeta_\q(n)}\int_{[\bar{G}_{n}]}|\phi_0\cdot\Eis(f)|^2 = \sum_{P}n_P^{-1}\sum_{\pi\in\Pi^G_{\mathrm{c}}(M)}\intop_{i\a^\ast_P/i\a^\ast_G}\frac{|L(\frac{1}{2},\mathcal{I}(\pi,\lambda)\otimes\tilde{\pi}_0)|^2}{\ell(\mathcal{I}(\pi,\lambda))}\mathcal{H}\left(\mathcal{I}(\pi,\lambda)\right)\, d\lambda
	\end{multline*}
	where $\mathcal{H}:=\prod_{v\mid\q}H_v\prod_{v\mid\infty}h_v$ with $H_v$ and $h_v$ as in \eqref{eq:def-H-sigma} and \eqref{eq:def-h-sigma}, respectively.
\end{lem}

\subsection{Finite ramified computations}
Our goal is to show the following result.

\begin{lem}\label{lem:spec_wt_ramified}
	Fix $v\mid\q$. Let $\sigma$ be an irreducible generic unitary representation of $\bar{G}_n(F_v)$. Recall $H_v$ from \eqref{eq:def-H-sigma}. If $c(\sigma)\le r_v$ then we have the lower bound
	\[H_v(\sigma)\geq 1\]
	Otherwise, $H_v(\sigma)$ vanishes identically.
\end{lem}

\begin{proof}

	Note that $v\nmid\Delta$, as $\q$ is coprime to the different. We drop subscripts $v$ for the rest of the proof. For $W\in\sigma$ we have, by meromorphic continuation from $f_s$ for sufficiently large $\Re(s)$,
	\[
	Z\left(W_0,\overline{W},f\right)
	=\int_{Z_n(F)N_n(F)\backslash {G}_n(F)}W_0(g)
	\overline{W(g)}f(g)\, dg,
	\]
	where the normalized quotient $Z(W_0,\overline W,f_s)/L(\tfrac12+s,\pi_0\otimes\overline\sigma)$ is a Laurent polynomial in $N(\p)^s$. The following identities are first proved for sufficiently large $\Re(s)$ and then evaluated at $s=0$ after cancelling the common $L$-factor. Writing $g$ in Iwasawa coordinates, $g=\left(\begin{smallmatrix}h&\\&1\end{smallmatrix}\right)k$ with $h\in N_{n-1}(F)\backslash G_{n-1}(F)$ and $k\in K_n$, and applying \eqref{eq:test_vector_on_kx} and sphericality of $W_0$ we write the above as
	\[
	N(\p^r)^{\frac{n-1}{2}}\int_{N_{n-1}(F)\backslash G_{n-1}(F)}
	W_0\left[\pmat{h&\\&1}\right]
	|\det h|^{-\frac12}\int_{K_0(\p^r)}\overline{W\left[\pmat{h&\\&1}k\right]}\,dk\,dh.
	\]
	The inner $K_0(\p^r)$-integral projects $W$ onto $\sigma^{K_0(\p^r)}$, which is trivial unless $c(\sigma)\le r$. This proves the second part.

	Assume now $c(\sigma)\le r$. We choose a basis $\B(\sigma)$ containing the normalized
	newvector $\frac{W_\sigma}{\| W_\sigma \|}$ of $\sigma$. Then, by positivity and invariance of the newvector we have
	\begin{equation}\label{eq:primary-bound-local-weight}
		H(\sigma)\ge N(\p^r)^{n-1}\vol\left(K_0(\p^r)\right)\frac{\zeta(1)\zeta(\frac{n}{2})^2}{\zeta(n)}\,\frac{|Z^\sharp(\tfrac{1}{2},W_0,W_\sigma)|^2}
		{\|W_\sigma\|^2_\sigma|L(\frac12,\pi_0\otimes \overline{\sigma})|^2},
	\end{equation}
	where
	\[Z^\sharp(s,W_1,W_2):=
	\int_{N_{n-1}(F)\backslash G_{n-1}(F)}W_1\left[\pmat{h&\\&1}\right]
	\overline{W_2\left[\pmat{h&\\&1}\right]}|\det h|^{s-1}\,dh.\]
	Assume that $W_j$ are normalized newvectors $W_{\pi_j}$ where $\pi_j$ are generic irreducible representations of $\bar{G}_n(F)$.
	Then from \eqref{shintani} and an application of the Cauchy identity \cite[Theorem 38.1]{bump2004lie} it follows that
	\begin{equation}\label{eq:cauchy-identity}
		Z^\sharp(s,W_1,W_2) = \left(1-\delta_{c(\pi_1)=0=c(\pi_2)}N(\p)^{-ns}\right)L\left(s,\pi_{1,\mathrm{ur}}\otimes\ol{\pi}_{2,\mathrm{ur}}\right)
	\end{equation}
	initially for $\Re(s)$ large and then as an identity of meromorphic functions; see, \emph{e.g.} \cite[\S3.1.3]{cogdell2007functions} or \cite[Proof of Theorem 3.5]{jo2023local} for a similar computation.
	As $\pi_0$ is unramified it follows that
	\begin{equation}\label{Local-RS-unramified-and-ramified}
		Z^\sharp(\tfrac{1}{2},W_0,W_\sigma)\\
		=\begin{cases}
			\zeta(\frac n2)^{-1}L(\frac12,\pi_0\otimes \overline{\sigma}),
			\text{ if }\sigma\text{ is unramified}, \\
			L(\tfrac12,\pi_0\otimes \overline{\sigma}),\text{ otherwise}.
		\end{cases}
	\end{equation}
	On the other hand, plugging $s=1$ into \eqref{eq:cauchy-identity} and recalling the definition of the inner product from \eqref{eq:def-inner-prod-whittaker} we obtain
	\begin{equation}\label{eq:upper-bound-newvector-norm}
		\|W_\sigma\|^2_\sigma = \frac{\zeta(n)}{L(1,\sigma_{\mathrm{ur}}\otimes\tilde{\sigma}_{\mathrm{ur}})}Z^\sharp(1,W_\sigma,W_\sigma)=\begin{cases}
			1,
			& \text{ if }\sigma\text{ is unramified}, \\
			\zeta(n), & \text{ otherwise}.
		\end{cases}.
	\end{equation}
	Finally, we use the identity
	$\vol\left(K_0(\p^r)\right)=N(\p^r)^{-(n-1)}\frac{\zeta(n)}{\zeta(1)}$ and combine \eqref{eq:primary-bound-local-weight}, \eqref{Local-RS-unramified-and-ramified}, and \eqref{eq:upper-bound-newvector-norm} to get
	\[H(\sigma)\ge\begin{cases}1,&\text{ if }\sigma\text{ is unramified},\\ \frac{\zeta(\frac{n}{2})^2}{\zeta(n)},&\text{ otherwise}.\end{cases}\]
	Noting that $\frac{\zeta(\frac{n}{2})^2}{\zeta(n)}\ge 1$ we conclude.
\end{proof}

\subsection{Analysis of the archimedean weight}

We prove the following lemma.

\begin{lem}\label{lem:spec_wt_arch}
	Fix $v$ to be a place such that $v\mid\infty$ and $0\le\vartheta<\tfrac{1}{2}$. Let $\sigma$ be an irreducible generic unitary $\vartheta$-tempered representation of $\bar{G}_n(F_v)$. Recall $h_v$ from \eqref{eq:def-h-sigma}. For every $C\ge 1$ there exists a $\tau>0$ such that the following is true
	$$h_v(\sigma)\gg_{C,\pi_{0,v}} 1\quad\text{if}\quad C(\sigma)\le C$$
	where $W_{0,v}$ and $\Phi_v$ in the definition of $h_v$ have $\tau$-dependent support condition as described in \S\ref{sec:choices}.
\end{lem}

\begin{proof}
	We drop $v$ from the subscript in this proof. A more complicated version of this lemma is proved in \cite[Proposition 7.1]{Jana2020RS}, albeit when $F_v=\R$. Here we give a short and simple proof of this weaker statement but for $F_v$ is $\R$ or $\C$.

	We start by choosing $\B(\sigma)\ni W$ such that $W\left[\left(\begin{smallmatrix}\bullet&\\&1\end{smallmatrix}\right)\right]=W_0\left[\left(\begin{smallmatrix}\bullet&\\&1\end{smallmatrix}\right)\right]$. Then by positivity and using Bruhat coordinates we write
	\begin{equation*}
		h(\sigma) \ge \left\vert\,\intop_{F^{n-1}}\intop_{N_{n-1}(F)\bs G_{n-1}(F)}W_0\left[\begin{pmatrix}h&\\c&1\end{pmatrix}\right]\overline{W\left[\begin{pmatrix}h&\\c&1\end{pmatrix}\right]} f\left[\begin{pmatrix}\mathrm{I}_{n-1}&\\c&1\end{pmatrix}\right]\,\frac{dh}{|\det h|^{\frac{1}{2}}}\, dc\right\vert^2.
	\end{equation*}
	First of all, from \eqref{eq:supp-f-arch} we restrict the $c$-integral to a $\tau$-radius ball. Using the mean value theorem and \cite[Lemma 3.1]{JaNu2021reciprocity} we obtain
	\begin{equation*}
		W\left[\begin{pmatrix}h&\\c&1\end{pmatrix}\right]=W\left[\begin{pmatrix}h&\\&1\end{pmatrix}\right] + \delta_{B_n}\left[\begin{pmatrix}h&\\&1\end{pmatrix}\right]O_{C,\pi_0}(\tau),
	\end{equation*}
	and a similar statement for $W_0$; \emph{cf}.\ \cite[Lemma 7.2]{Jana2020RS}.
	Working as in \cite[Proof of Proposition 7.1]{Jana2020RS} we write the above integral as
	\begin{equation*}
		\intop_{N_{n-1}(F)\bs G_{n-1}(F)}\left\vert W_0\left[\begin{pmatrix}h&\\&1\end{pmatrix}\right]\right\vert^2\,\frac{dh}{|\det h|^{\frac{1}{2}}}\intop_{F^{n-1}} f\left[\begin{pmatrix}\mathrm{I}_{n-1}&\\c&1\end{pmatrix}\right]\, dc+O_{C,\pi_0}(\tau)\int_{F^{n-1}}|f[\pmat{\mathrm I_{n-1}&\\c&1}]|\,dc,
	\end{equation*}
	where the implicit constant is uniform for $C(\sigma)\le C$. The choice of $W_0$ as in \eqref{eq:support-arch-whittaker} yields that the above $h$-integral is $\asymp \|W_0\|^2$. Applying the normalization \eqref{eq:arch-normalization-W} and \eqref{eq:away-from-zero}, and making $\tau$ sufficiently small in terms of $C$ and $\pi_0$ we conclude.
\end{proof}

\subsection{Plancherel average of the spectral weight}\label{sec:Planch_Whitt_spec_wt}

In this section we calculate the average value of the spectral weight $H_v$ and $h_v$, as defined in \S\ref{sec:choices}, as follows.

We denote the isomorphism class of irreducible unitary representations of $\bar{G}_n(F_v)$ by $\widehat{\bar{G}_n(F_v)}$ and equip it with a local Plancherel density $d\mu_{\loc}$, so that the Whittaker--Plancherel formula holds. Namely, for $J\in \mathcal{S}\left(N_n(F_v)\bs\bar{G}_n(F_v),\psi_v\right)$ we have
\begin{equation}\label{eq:whittaker-plancherel}
	\int_{N_n(F_v)\bs \bar{G}_n(F_v)}|J|^2 = \int_{\widehat{\bar{G}_n(F_v)}}\sum_{W\in\B'(\sigma)}\left|\int_{N_n(F_v)\bs\bar{G}_n(F_v)} J\cdot\overline{W}\right|^2 \, d\mu_{\loc}\sigma,
\end{equation}
where $\B'$ denotes an orthonormal basis of the Whittaker model of $\sigma$ where the inner product is defined by the un-normalized (compared to \eqref{eq:def-inner-prod-whittaker}) integral over $N_{n-1}(F_v)\bs G_{n-1}(F_v)$; see \cite[p.160]{Delorme2013}. It is known that $d\mu_{\loc}$ is supported on tempered representations.

\begin{lem}\label{lem:average-value-spec-wt}
	For every place $v\mid\q$ we have
	\[\int_{\hat{\bar{G}_n(F_v)}} H_v(\sigma)\, d\mu_{\loc}\sigma \asymp \mathrm{vol}^{-1}\left(K_0\left(\p_v^{r_v}\right)\right),\]
	where the implied constant is absolute.
\end{lem}

\begin{proof}
	We drop the subscript $v$ in the proof.
	We start by renormalizing the weight function
	\[H_0(\sigma)\defeq\frac{\zeta(n)^2}{\zeta(\frac{n}{2})^2 L(1,\sigma_{\mathrm{ur}}\otimes \tilde{\sigma}_{\mathrm{ur}})L(1,\pi_0\otimes \tilde{\pi}_0)} \lvert L(1/2,\tilde{\sigma}\otimes\pi_0)\rvert^2\cdot H(\sigma).\]
	As $\pi_0$ is $\vartheta$-tempered with $\vartheta\le\tfrac12-\tfrac{1}{n^2+1}$ at finite places and $\sigma$ is tempered on the Plancherel support, it follows that
	\[H(\sigma)\asymp H_0(\sigma)\]
	with the implied constant independent of $\sigma$, $\pi_0$ and $\p^r$.
	We thus have
	\begin{equation*}
		\int_{\widehat{\bar{G}_n(F)}} H_0(\sigma)\,
		d\mu_{\loc}\sigma= \frac{\vol^{-1}(K_0(\p^r))\zeta(1)}{L(1,\pi_0\otimes \tilde{\pi}_0)}
		\int_{\widehat{G_n(F)}} \sum_{W\in\B'(\sigma)}
		\Big\lvert \int_{N_n(F)\bs \bar{G}_n(F)} W_0 \cdot\overline{W}\cdot
		f\Big \rvert^2 \, d \mu_{\loc}\sigma.
	\end{equation*}
	By \eqref{eq:whittaker-plancherel} the above integral evaluates to
	\begin{align*}
		\int_{N_n(F)\bs \bar{G}_n(F)} \left\lvert W_0\cdot f\right\rvert^2
		& =\int_{N_{n-1}(F)\backslash G_{n-1}(F)} \left\lvert W_0\left[\begin{pmatrix}a&\\&1\end{pmatrix}\right]\right\rvert^2\,\frac{da}{\delta_{B_{n-1}}(a)}\int_{K}|f(k)|^2\,dk \\
		& =\frac{L(1,\pi_0\otimes \tilde{\pi}_0)}{\zeta(n)} N(\p^r)^{n-1}\vol(K_0(\p^r)).
	\end{align*}
	The first equality above follows after using Iwasawa coordinates and sphericality of $W_0$. The second equality follows from \eqref{eq:test_vector_on_kx} and \eqref{eq:cauchy-identity} at $s=1$.
\end{proof}

\subsection{Combining estimates}

We finish this section combining Lemma \ref{lem:abstract-rankin-selberg}, Lemma \ref{lem:spec_wt_ramified}, Lemma \ref{lem:spec_wt_arch}, and Lemma \ref{lem:average-value-spec-wt}, which yields the following proposition.

\begin{prop}\label{prop:spec_exp_second_moment}
	Let $(C_v)_{v\mid \infty}$ be a collection of fixed real numbers satisfying $C_v\geq 1$ and $\q$ and $\pi_0\ni\phi_0$ be as in \S\ref{sec:choices}. Then there exists a spectral weight function $\H$ defined on the generic automorphic representations for $\bar{G}_n(F)$, factorizing as
	\begin{equation*}
		\H(\sigma)=\prod_{v<\infty}H_v(\sigma_v)\prod_{v\mid\infty}h_v(\sigma_v),\quad \sigma=\otimes'_v\sigma_v,
	\end{equation*}
	such that we have the spectral expansion of the period
	\begin{multline*}
		\mathrm{vol}^{-1}(K_0(\q))\frac{\zeta_\q(1)\zeta_\q(\frac{n}{2})^2}{\zeta_\q(n)}\int_{[\bar{G}_{n}]}|\phi_0\cdot\Eis(f)|^2 = \sum_{P}n_P^{-1}\sum_{\pi\in\Pi_{\mathrm{c}}(\bar{M})}\intop_{i\a^\ast_P/i\a^\ast_G}\frac{|L\left(\frac{1}{2},\mathcal{I}(\pi,\lambda)\otimes\tilde{\pi}_0\right)|^2}{\ell(\mathcal{I}(\pi,\lambda))}\mathcal{H}\left(\mathcal{I}(\pi,\lambda)\right)\, d\lambda,
	\end{multline*}
	where $\H$ satisfies the following properties:
	\begin{itemize}
		\item If $v<\infty$ and $c(\sigma_v)>r_v$ then $H_v(\sigma_v)=0$;
		\item If $v<\infty$ and $c(\sigma_v)\le r_v$ then $H_v(\sigma_v)\geq 1$, with equality when $r_v=0$;
		\item If $v\mid\infty$ and $C(\sigma_v)\leq C_v$ then $h_v(\sigma_v) \gg_{C_v} 1$;
		\item $\int_{\widehat{\bar{G}_n(F_v)}}H_v(\sigma_v)\, d\mu_{\loc}\sigma_v\asymp \mathrm{vol}^{-1}\left(K_0\left(\p_v^{r_v}\right)\right)$;
	\end{itemize}
	as $N(\q)\to\infty$.
\end{prop}

\section{Explicit Computations of the Degenerate Terms}\label{sec:degenerate_calculations}

Recall the choices of the cuspidal representation $\pi_0\ni\phi_0$ and the induced vectors $f,f^0$ from \S\ref{sec:choices}. We apply Theorem \ref{thm:spectral-expansion} to spectrally expand the period integral
\begin{equation*}
	\int_{[\bar{G}_{n}]}|\phi_0|^2\cdot\overline{\Eis(f)}\cdot\Eis(f_s),
\end{equation*}
for $\Re(s)$ sufficiently small, as
\begin{align*}
	M_{00}(s)+M_{10}(s) +M_{01}(s)+M_{11}(s)+I_2^{\mathrm{D}}(s)+I_2^{\mathrm{C}}(s)
\end{align*}
where
\begin{align*}
	M_{00}(s) & =\IP{|\phi_0|^2,\Eis\left(f\cdot\ol{f_s}\right)},         \\
	M_{01}(s) & =\IP{|\phi_0|^2,\Eis\left(f\cdot\ol{\M f_s}\right)},    \\
	M_{10}(s) & =\IP{|\phi_0|^2,\Eis\left(\M f\cdot\ol{f_s}\right)},     \\
	M_{11}(s) & =\IP{|\phi_0|^2,\Eis\left(\M f\cdot\ol{\M f_s}\right)}
\end{align*}
and $I_2^{\mathrm{D}}(s)$ and $I_2^{\mathrm{C}}(s)$ are defined using Proposition \ref{prop:cont-I2D} and Proposition \ref{prop:cont-I2C} (and Remark \ref{rmk:abuse-def-I2C}), respectively.
In this section, we analyze the degenerate terms $M_{ij}(s)$ when $s\neq 0$. For that we define local factors $h_{ij,v}$ for each place $v\mid \q$ as follows:
\begin{align*}
	h_{00,v}(s) & =\frac{\zeta_v(n)}{\zeta_v(1)\zeta_v(\tfrac n2)\zeta_v\left(\tfrac n2+ns\right)},      \\
	h_{01,v}(s) & =N(\p_v^{r_v})^{-\frac{n-2}{2}-ns}\frac{\zeta_v(n)}{\zeta_v^2(1)\zeta_v(\tfrac n2)},                          \\
	h_{10,v}(s) & =N(\p_v^{r_v})^{-\frac{n-2}{2}}\frac{\zeta_v(n)}{\zeta_v^2(1)\zeta_v(\tfrac n2+ns)},                  \\
	h_{11,v}(s) & =N(\p_v^{r_v})^{-ns}
	\frac{\zeta_v(n)\zeta_v(n-1-ns)}{\zeta_v(1)\zeta_v(n-1)\zeta_v(\frac n2)\zeta_v(\tfrac n2-ns)}                             \\
	& \times\left(1-N(\p_v)^{-1-r_v(n-1-ns)}\frac{
		\zeta_v(n-1)\zeta_v(\tfrac n2)\zeta_v(\frac n2-ns)}
	{ \zeta_v(\frac{n-2}{2}) \zeta_v(ns)\zeta_v(\tfrac {n-2}2-ns)}\right).
\end{align*}
Finally, we define $M^0_{ij}(s)$ by the above inner products which equal $M_{ij}(s)$ but replacing $f_s$ by $f^0_s$. For example, $M^0_{11}(s):=\IP{|\phi_0|^2,\Eis\left(\M f^0\cdot\ol{\M f^0_s}\right)}$.
The main result of this section is the following relation between $M$, $M^0$, $h$.

\begin{prop}\label{prop:main_terms_calcs}
	We have the following explicit expressions for our four degenerate terms:
	\begin{equation*}
		M_{ij}(s)=M^0_{ij}(s) h_{ij,\q}(s),\quad\text{for } i,j\in\{0,1\},
	\end{equation*}
	for $s$ lying at a general position.
\end{prop}

\subsection{Analysis of $M_{00}(s)$, $M_{01}(s)$ and $M_{10}(s)$}\label{ssec:local-calcs-M00-M01-M10}

In this subsection, we prove Proposition \ref{prop:main_terms_calcs} for $(i,j)=(1,0)$; the proofs of the cases $(i,j)=(0,0)$ and $(i,j)=(0,1)$ are completely analogous.

Fix $s$ with sufficiently large $\Re(s)$. Let $\Q:=\Q_1\cap\Q_{n-1}$, as before. We unfold $M_{10}(s)$ and use Iwasawa coordinates to write
\begin{align*}
	M_{10}(s)
	& =\int_{\left[Z_n\backslash M_\Q\right]\times K_n}(|\phi_0|^2)_\Q(mk) \overline{\M f}(mk)f_s(mk)\,\frac{dm\,dk}{\delta_{\Q}(m)}                                               \\
	& =\int_{\left[Z_n\backslash M_\Q\right]\times K_n}(|\phi_0|^2)_\Q(mk)\frac{\delta^{1/2}_{\Q_{1}}(m)\delta^{1/2+s}_{\Q_{n-1}}(m)}{\delta_{\Q}(m)}\ol{\M f}(k)f_s(k)\, dm\, dk.
\end{align*}
Recall the choice of $f^0$ from \S\ref{sec:choices}. Note that for all $v\nmid\q$ we have $f_s=f^0_s$ and thus for sufficiently large $\Re(s)$ we have
\begin{equation}\label{eq:factorization-f-f0}
	f_s = f_{s}^{0,\q}f_{\q,s}\quad\text{and}\quad\M f_s = \M^\q f_{s}^{0,\q}\cdot \M_\q f_{\q,s},
\end{equation}
both of which now hold for $s$ in general position by virtue of meromorphicity.
Since $\phi_0$ is spherical at every $v\mid\q$, factorizing the above $K_n$-integral into $K_{n,v}$-integrals we obtain
\begin{multline}\label{eq:M10-before-local-comp}
	M_{10}(s)=\int_{\left[Z_n\backslash M_\Q\right]\times K^\q_{n}}(|\phi_0|^2)_\Q(mk)\frac{\delta^{1/2}_{\Q_{1}}(m)\delta^{1/2+s}_{\Q_{n-1}}(m)}{\delta_{\Q}(m)}\ol{\M^\q f^{0,\q}}(k)f^{0,\q}_s(k)\, dm\, dk\\
	\times \int_{K_{n,\q}}\ol{\M_\q f_\q}(k)f_{\q,s}(k)\, dk.
\end{multline}
A similar formula holds for $M^0_{10}(s)$. 
By \eqref{eq:test_vector_on_kx}, for every $v\mid \q$ we have
\[\int_{K_{n,v}}\ol{\M_v f_{v}}(k_v) f_{v,s}(k_v)\, dk_v=N(\p_v^{r_v})^{\frac{n-1}{2}}\int_{K_0(\p_v^{r_v})}\ol{\M_v f_{v}}(k_v)\, dk_v.\]
Note that the choice of $f_v$ as in \S\ref{sec:choices} ensures that $f_{v,s}$
is right $K_0(\p_v^{r_v})$-invariant, and consequently, so is $\M_v f_{v,s}$.
Moreover, by adapting the argument leading to \cite[eq.(3.5)]{Jana2020RS} to
this local setting,
we can show that there exists a constant $c$, depending only on $v$
(and not on $\Phi_v$), such that
\begin{align*}
	\M_v f_v(g)= c\int_{F_v^\times}\widehat{\Phi}_v(te_nwg^{-\top})
  \lvert t\rvert^{\frac{n}{2}} \,d^\times t.
\end{align*}
Recalling our choices in \S \ref{sec:choices}, we calculate
$$
\hat\Phi(x_1,...,x_n)=N(\p_v^{r_v})^{-\frac{n-1}{2}}\mathbbm{1}_{\p_v^{-r_v}}(x_1)...\mathbbm{1}_{\p_v^{-r_v}}(x_{n-1})(\,\!\mathbbm{1}_{\o_v}(x_n)-|\varpi_v|_v\mathbbm{1}_{\o_v}(\varpi_v x_n)\,\!).
$$
In particular, we get
\begin{align*}
	\M_v f_v(1)= c\,\frac{N(\p_v^{r_v})^{-\frac{n-1}{2}}}{\zeta_v(1)}\int_{F_v^\times}\mathbbm{1}_{\p_v^{-r_v}}(t) \lvert t\rvert^{\frac{n}{2}} \,d^\times t=c\,N(\p_v^{r_v})^{1/2}\,\frac{\zeta_v(\tfrac n2)}{\zeta_v(1)}
\end{align*}
and, for $k_v\in K_0(\p_v^{r_v})$, one deduces that
\[
\M_v f_{v}(k_v)=
\M_v f_{v}(1)=c\,N(\p_v^{r_v})^{1/2}\frac{\zeta_v(\tfrac n2)}{\zeta_v(1)}.
\]
Therefore,
\[\int_{K_{n,v}}\ol{\M_v f_{v}}(k_v)f_{v,s}(k_v)\, dk_v=\bar c\,N(\p_v^{r_v})^{-\frac{n-2}2}\frac{\zeta_v(n)\zeta_v(\tfrac n2)}{\zeta_v(1)^2}.\]
One also has the unramified computation
\[\int_{K_{n,v}}\ol{\M_v f^0_{v}}(k_v) f^0_{v,s}(k_v)\, dk_v=\ol{\M_v f^0_{v}}(1) f^0_{v,s}(1)=\bar c\,\zeta_v(\tfrac n2)\zeta_v(\tfrac n2+ns),\]
where, similarly, we get $\M_v f^0_{v}(1)=c\,f^0_{v}(1)=c\,\zeta_v(\tfrac n2)$.
Consequently, the constant $c$ cancels and we deduce that
\[
\int_{K_{n,v}}\ol{ \M_v f_{v}}(k_v) f_{v,s}(k_v)\, dk_v=h_{10,v}(s)\int_{K_{n,v}}\ol{ \M_v f_{v}^0}(k_v) f^0_{v,s}(k_v)\, dk_v.
\]
Applying this to \eqref{eq:M10-before-local-comp} and folding the integral back, we arrive at
\[
M_{10}(s)=M^0_{10}(s) h_{10,\q}(s).
\]
Since both sides of the equality above are meromorphic in $s$ we conclude the proof.

\subsection{Analysis of $M_{11}(s)$}\label{ssec:local-calcs-M11}

In this section, we prove Proposition \ref{prop:main_terms_calcs} for $(i,j)=(1,1)$. The argument follows the same lines as those of the previous cases. The only difference is that the ramified computation is more complicated.

For $\Re(s)$ sufficiently negative, we start again by unfolding the integral $M_{11}(s)$. Performing a similar maneuver as in \eqref{eq:M10-before-local-comp} with the factorization of the integral using \eqref{eq:factorization-f-f0} we obtain
\begin{multline*}
	M_{11}(s)=\int_{\left[Z_n\backslash M_{\Q_1}\right]\times K^\q_{n}}(|\phi_0|^2)_{\Q_1}(mk)\delta^{-s}_{\Q_1}(m)\ol{\M^\q f^{0,\q}}(k)\M^\q f^{0,\q}_s(k)\, dm\, dk
	\times \int_{K_{n,\q}}\ol{\M_\q f_\q}(k)\M_\q f_{\q,s}(k)\, dk.
\end{multline*}
As before, thus it suffices to show that
\begin{equation}\label{eq:local-calc-of-M_11}
	\int_{K_{n,v}} \overline{\M f_v}(k)\, \M f_{v,s}(k)\, dk
	=h_{11,v}(s)\int_{K_{n,v}} \overline{\M f^0_v}(k)\, \M f^0_{v,s}(k)\, dk,
\end{equation}
for every $v\mid \q$ and for sufficiently large $\Re(s)$.

We fix a place $v\mid \q$ until the end of the section and drop any subscript.
Let $\varphi\in \1_{n-1}\boxplus \1_1$ be a spherical vector such that
$\M\varphi_\lambda(1)\neq 0$, where $\lambda=\lambda_{\mathrm{E}}(\tfrac12-s)$
as defined in \eqref{eq:def-lambda-E}. Thus applying Remark \ref{rmk:local-period-equality-M11} we find an $s$-dependent scalar $c_s$ so that
\begin{align*}
	&\int_{K_n} \overline{\M f}\cdot \M f_{s}
	 =\frac{1}{\M \varphi_{\lambda}(1)}\int_{K_n}\M \varphi_{\lambda}
	\cdot\overline{\M f}\cdot \M f_{s}\\
	&=c_s\intop_{K_n}\intop_{Z_2(F)N_2(F)\bs G_2(F)}\W(\varphi_\lambda)\cdot\overline{\W(f)}\cdot f_s\left[\begin{pmatrix}\mathrm{I}_{n-2}&\\&h\end{pmatrix}k\right]
	|\det h|^{n-2}\, dh\, dk.
\end{align*}
Similarly, we also obtain
\begin{multline*}
	\int_{K_n} \overline{\M f^0}(k)\, \M f^0_{s}(k)\, dk=c_s\intop_{K_n}\intop_{Z_2(F)N_2(F)\bs G_2(F)}\W(\varphi_\lambda)\cdot\overline{\W(f^0)}\cdot f_s^0\left[\begin{pmatrix}\mathrm{I}_{n-2}&\\&h\end{pmatrix}k\right]
	|\det h|^{n-2}\, dh\, dk
\end{multline*}
with the same $s$-dependent constant $c_s$ as before. Thus to prove \eqref{eq:local-calc-of-M_11} it suffices to show that
\begin{multline}\label{eq:equivalent-form-of-lemma-M11}
	\intop_{K_n}\intop_{Z_2(F)N_2(F)\bs {G}_2(F)}\W(\varphi_\lambda)\cdot\overline{\W(f)}\cdot f_s\left[\begin{pmatrix}\mathrm{I}_{n-2}&\\&h\end{pmatrix}k\right]
	|\det h|^{n-2}\, dh\, dk\\
	=h_{11}(s)\intop_{K_n}\intop_{Z_2(F)N_2(F)\bs G_2(F)}\W(\varphi_\lambda)\cdot\overline{\W(f^0)}\cdot f^0_s\left[\begin{pmatrix}\mathrm{I}_{n-2}&\\&h\end{pmatrix}k\right]
	|\det h|^{n-2}\, dh\, dk.
\end{multline}
We start with the integral on the right-hand side. First, using uniqueness of the spherical vectors we compute
\begin{equation}\label{eq:GL2-realization}
	\W(\varphi_\lambda)=\zeta(n-ns)^{-1}W_{\pi_s},\quad\W(f^0)=W_{\pi_{1/2}},\quad\Phi^0(e_n\bullet)=\mathbbm{1}_{\o^2}(e_2\cdot),
\end{equation}
when evaluated on $\left(\begin{smallmatrix}\mathrm{I}_{n-2}\\&\bullet\end{smallmatrix}\right)$, for $\pi_s:=|\cdot|^{\frac{1}{2}-s}\boxplus|\cdot|^{(n-1)s-n+\frac{3}{2}}$.
Thus by the classical unramified $\GL_2\times\GL_2$ computation the right hand side of \eqref{eq:equivalent-form-of-lemma-M11} equals
\begin{equation}\label{eq:unramified-M11-zeta-version}
	h_{11}(s)\frac{\zeta(n-1)\zeta(\tfrac n2)\zeta(ns)\zeta(ns-\frac{n-2}{2})}{\zeta(n-ns)};
\end{equation}
\emph{cf}.\ the relevant computation in \S\ref{sec:unramified-local-choice}.

We now focus on the integral on the left-hand side of \eqref{eq:equivalent-form-of-lemma-M11}.
Using the Iwasawa coordinates and changing variable in the $k$-integral we write this as
\begin{equation*}
	\intop_{K_n}\intop_{F^\times}\W(\varphi_\lambda)\cdot\ol{\W(f)}\cdot f_s
	\left[\pmat{\mathrm{I}_{n-2}\\&y\\&&1}k\right]|y|^{n-3}
	\,d^\times y\, dk.
\end{equation*}
Using \eqref{eq:test_vector_on_kx}, and right $K_0(\p^r)$-invariance of $\varphi_\lambda$ and $\W(f)$ we rewrite the above as
\begin{multline}\label{eq:intWWf-equal-vol-sth-intWW}
	\vol(K_0(\p^r))N(\p^r)^{\frac{n-1}{2}}
	\times\int_{F^\times}\W(\varphi_\lambda)\left[\pmat{\mathrm{I}_{n-2}\\&y\\&&1}\right]
	\ol{\W(f)}\left[\pmat{\mathrm{I}_{n-2}\\&y\\&&1}\right]
	|y|^{n-5/2+s} \, d^\times y.
\end{multline}
Now, recalling the definition of $\Phi$ from \S\ref{sec:choices}
and writing $\mathbbm{1}_{\o^\times}(u)=\mathbbm{1}_{\o}(u)-\mathbbm{1}_{\o}(u\varpi^{-1})$ and $\mathbbm{1}_{\p^r}(u)=\mathbbm{1}_{\o}(u\varpi^{-r})$ we write
\[
\Phi(e_n g)=N(\p^r)^{\frac{n-1}{2}}\left(\Phi^0\left[e_n g\pmat{\varpi^{-r}\mathrm{I}_{n-1}\\&1}\right]-\Phi^0\left[e_n g\pmat{\varpi^{-r}\mathrm{I}_{n-1}\\&\varpi^{-1}}\right]\right),
\]
for $g\in G_n(F)$. Consequently, from \eqref{eq:GL2-realization} and the invariance of $f^0$ under the center it follows that
\[
\W(f)\left[\pmat{\mathrm{I}_{n-2}\\&h}\right]=W_{\pi_{1/2}}\left[ h\pmat{1\\&\varpi^{r}}\right]-N(\p)^{-1/2}W_{\pi_{1/2}}\left[ h\pmat{1\\&\varpi^{r-1}}\right],
\]
for $h\in G_2(F)$. In particular, we have
\begin{multline*}
	\W(f)\left[\pmat{\mathrm{I}_{n-2}\\&y\\&&1}\right]
	=N(\p^r)^{\frac{n-2}{2}}\left(W_{\pi_{1/2}}\left[\pmat{y\varpi^{-r}\\&1}\right] -N(\p)^{-\frac {n-1}2}W_{\pi_{1/2}}\left[\pmat{y\varpi^{-r+1}\\&1}\right]\right).
\end{multline*}
Thus using \eqref{eq:GL2-realization} we write the integral in \eqref{eq:intWWf-equal-vol-sth-intWW} as
\begin{equation}\label{eq:intWW-equal-Ar-minus-Ar-1}
	N(\p^r)^{\frac{n-2}{2}}\,\frac{A(\varpi^{-r})-N(\p)^{-\frac {n-1}2}A(\varpi^{-r+1})}{\zeta(n-ns)},
\end{equation}
where
\begin{equation*}A(z):=\int_{F^\times}W_{\pi_s}\left[\pmat{y\\&1}\right]
	\ol{W_{\pi_{1/2}}}\left[\pmat{yz\\&1}\right]
	|y|^{n-5/2+s}\, d^\times y.
	\end{equation*}
Using Shintani's formula \eqref{shintani}  (or \cite[Theorem 4.6.5]{bump19927automorphic}) we write
\[W_{\pi_s}\left[\pmat{y\\&1}\right]=
\left(\zeta(n-1-ns)|y|^{(n-1)s-n+2}+\zeta(ns-n+1)|y|^{1-s}\right)\cdot\mathbbm{1}_\o(y),\]
and consequently, after changing variable $y\mapsto yz^{-1}$, we obtain
\begin{multline*}
	A(z)=\zeta(n-1-ns)\int_{F^\times}\overline{W_{\pi_{1/2}}}  \left[\pmat{y\\&1}\right]\mathbbm{1}_{z\o}(y)|yz^{-1}|^{ns-\frac12}\, d^\times y\\
	+\zeta(ns-n+1)\int_{F^\times}\overline{W_{\pi_{1/2}}}  \left[\pmat{y\\&1}\right]\mathbbm{1}_{z\o}(y)|yz^{-1}|^{n-\frac32}\, d^\times y.
\end{multline*}
Note that $\mathbbm{1}_{z\o}(y)=1$ for $z=\varpi^{-r}$ or $z=\varpi^{-r+1}$ and $\left(\begin{smallmatrix}y&\\&1\end{smallmatrix}\right)$ in the support of $W_{\pi_{1/2}}$. Thus using the unramified computation of local $\GL_2\times\GL_1$ zeta integral (see \cite[Theorem 3.3]{cogdell2007functions}) we evaluate
\begin{align*}
	A(z) & =\zeta(n-1-ns)|z|^{\frac12-ns} L(ns,\ol{\pi_{1/2}})+\zeta(ns-n+1)|z|^{\frac32-n}L(n-1,\ol{\pi_{1/2}}) \\
	& =|z|^{\frac 12-ns}\zeta(n-1-ns)\zeta(ns)\zeta(ns-\tfrac{n-2}{2})+
	|z|^{\frac32-n}\zeta(ns-n+1)\zeta(n-1)\zeta(\tfrac{n}{2}).
\end{align*}
Thus, plugging in the above expression of $A(z)$ in \eqref{eq:intWW-equal-Ar-minus-Ar-1} we compute \eqref{eq:intWWf-equal-vol-sth-intWW} (which is the left hand side of \eqref{eq:equivalent-form-of-lemma-M11}) as
\begin{equation*}
	N(\p^r)^{-ns}\frac{\zeta(n)\zeta(n-1-ns)\zeta(ns)\zeta(ns-\tfrac{n-2}{2})}{\zeta(1)\zeta(n-ns)\zeta(\frac n2-ns)}
	\Bigg(1-N(\p^{r+1})^{ns+1-n}\frac{\zeta(n-1)\zeta(\tfrac{n}{2})\zeta(\tfrac n2-ns)}{\zeta(-\frac{n-2}{2})\zeta(ns)\zeta(ns-\frac{n-2}{2})}\Bigg).
\end{equation*}
Recalling \eqref{eq:unramified-M11-zeta-version} and using the fact that $\zeta(s)=-N(\p)^s\zeta(-s)$ we conclude the proof of
\eqref{eq:equivalent-form-of-lemma-M11} and hence that of Proposition \ref{prop:main_terms_calcs}.

\section{Bounding the spectral weights on the dual side}\label{sec:bounding_error_term}
Recall $I_2^{\mathrm D}$ and $I_2^{\mathrm C}$ from Proposition \ref{prop:before-meromorphic-continuation} and their continuation to $s=0$ from $\Re(s)>0$ via Proposition \ref{prop:cont-I2D} and Proposition \ref{prop:cont-I2C}, respectively. In particular, recall the \emph{new} definition of $I^{\mathrm C}_2(s)$ near $s=0$ via Remark \ref{rmk:abuse-def-I2C}.

We fix $\phi$ to be $|\phi_0|^2$ and $\Phi,f$ to be as chosen in \S\ref{sec:choices}. In this section we prove the following result.

\begin{prop}\label{prop:error_term_bound}
	With the above choices of $\phi$ and $\Phi$ we have
	\[I_2^{\mathrm{D}}(0)\ll_{\pi_0,\epsilon} N(\q)^{-\frac{n-1}{2}+\vartheta_2+\epsilon},\]
	\[I_2^{\mathrm{C}}(0)\ll_{\pi_0,\epsilon} N(\q)^{-\frac{n-1}{2}+\epsilon},\]
	where $0\le\vartheta_2\le\tfrac{7}{64}$ is a bound towards the Generalized Ramanujan Conjecture for $\GL_2$.
\end{prop}

\begin{proof}
	We start with $I_2^\mathrm{D}(0)$. Let $S:=\{v\mid\mathfrak{D}\q\}\cup\{v\mid\infty\}$ and recall $\phi_0$ from
	\S\ref{sec:choices}, which is spherical at all $v<\infty$. This ensures that the $\pi$-sum in the expression of
	$I_2^{\mathrm D}$ is supported on the ones that are unramified at all $v<\infty$. Moreover, we choose
	$\B(\1_{n-2}\boxplus\pi)\ni\varphi$ so that $\varphi$ is factorizable, $\varphi_v$ is the normalized spherical vector
	in $(\1_{n-2}\boxplus\pi)_v$ for $v<\infty$ and $\varphi_v$ is the $\Delta$-eigenvector for $v\mid\infty$, where
	$\Delta$ is the Laplacian on $G_n(F_\infty)$ as described in the proof of Proposition \ref{prop:inflation-prop}.

	First, we estimate $\P_S(0;\varphi,f,f)$ for the above choice of $\varphi$.
	By the proof of Lemma \ref{lem:factorization-period} we see that the integral defining
	$\P_v$ in \eqref{eq:main-RS-period-local} converges absolutely for $s=0$ (and $\lambda=0$).
	Moreover, the trivial estimates
	\begin{equation*}
		\P_v(0;\varphi_v,f_v,f_v)
		\begin{cases}\ll_F 1                       & \text{ if }v\mid\mathfrak{D}, \\
			\ll_{\pi_0}\nu_\varphi^{O(1)} & \text{ if }v\mid\infty,\end{cases}
	\end{equation*}
	where $\nu_\varphi$ is the $\Delta$-eigenvalue of $\varphi$, follow from the same proof.

	We focus on estimating it for $v\mid\q$ now, which from \eqref{eq:main-RS-period-local} equals
	\[\intop_{K_{n,v}}\intop_{N_2(F_v)\bs \bar{G}_2(F_v)}\W(\varphi_{\lambda,v})\cdot\overline{\W(f_v)}\cdot f_v\left[\begin{pmatrix}\mathrm{I}_{n-2}&\\&h\end{pmatrix}k\right]
	|\det h|^{n-2}\, dh\, dk,\]
	where $\varphi_v\in(\1_{n-2}\boxplus\pi)_v$ is realized in $\1_{n-2}\boxplus|\cdot|_v^\nu\boxplus|\cdot|_v^{-\nu}$, so $\lambda=(0,\nu,-\nu)\in\a^\ast_{\widetilde{\Q_{n-2}},\C}$ for some $\nu\in \mathbb{C}$ with $|\Re(\nu)|\le \vartheta_2<\tfrac{1}{2}$. Invoking the same argument in \S \ref{ssec:local-calcs-M11} that we used to compute the left hand side of \eqref{eq:equivalent-form-of-lemma-M11}, we rewrite the above local integral as
	\begin{multline}\label{eq:local_period-error-term}
		\vol(K_0(\p_v^{r_v}))N(\p_v^{r_v})^{\frac{n-1}{2}}\int_{F_v^{\times}}\W(\varphi_{\lambda,v})\left[\pmat{\mathrm{I}_{n-2}\\&y\\&&1}\right]
		\ol{\W(f_v)}\left[\pmat{\mathrm{I}_{n-2}\\&y\\&&1}\right]|y|^{n-5/2}\,d^\times y,
	\end{multline}
	\emph{cf}.\ \eqref{eq:intWWf-equal-vol-sth-intWW}.
	Similarly, using the arguments used to obtain \eqref{eq:GL2-realization} and \eqref{eq:intWW-equal-Ar-minus-Ar-1} we write
	\begin{align*}& \W(f_v)\left[\pmat{\mathrm{I}_{n-2}\\&y\\&&1}\right]\\ & \quad=N(\p_v^{r_v})^{\frac{n-2}{2}}\left(W_\tau\left[\pmat{y\varpi_v^{-r_v}\\&1}\right] -N(\p_v)^{-\frac {n-1}2}W_\tau\left[\pmat{y\varpi_v^{-r_v+1}\\&1}\right]\right) \\
		& \quad\ll N(\p_v^{r_v})^{\frac{n-2}{2}}\left(\left(N(\p_v^{r_v})|y|\right)^{-\frac{n-3}{2}}\mathbbm{1}_{\p_v^{r_v}}(y)+\left(N(\p_v^{r_v-1})|y|\right)^{-\frac{n-3}{2}}N(\p_v)^{-\frac {n-1}2}\mathbbm{1}_{\p_v^{r_v-1}}(y)\right),
	\end{align*}
	for $\tau:=\1_{1,v}\boxplus |\cdot|_v^{-\frac{n-2}2}$, and
	\begin{equation*}
		\W(\varphi_{\lambda,v})\left[\pmat{\mathrm{I}_{n-2}\\&y&\\&&1}\right]= W_{\sigma}\left[\pmat{y&\\&1}\right]\ll_\epsilon |y|^{-\frac{n-3}{2}-\vartheta_2-\epsilon}\mathbbm{1}_\o(y)
	\end{equation*}
	for $\sigma:=|\cdot|_v^{\nu-\frac{n-2}2}\boxplus |\cdot|_v^{-\nu-\frac{n-2}2}$. Here the implied constants, in particular, do not depend on $v$ (and $\q$).

	Thus we estimate \eqref{eq:local_period-error-term} as
	\begin{multline*}
		\ll_\epsilon\vol(K_0(\p_v^{r_v}))N(\p_v^{r_v})^{n-\frac32}\left(N(\p_v^{r_v})^{-\frac{n-3}{2}}\int_{\p_v^{r_v}}|y|^{\frac{1}{2}-\vartheta_2-\epsilon}\, d^\times y\right.\\
		\left.+N(\p_v^{r_v-1})^{-\frac{n-3}{2}}N(\p_v)^{-\frac {n-1}2}\int_{\p_v^{r_v-1}}|y|^{\frac{1}{2}-\vartheta_2-\epsilon}\, d^\times y\right).
	\end{multline*}
	As $\vartheta_2<\tfrac{1}{2}$ the integrals above are absolutely convergent and bounded by
	\begin{equation*}
		\ll_\epsilon N(\p_v^{r_v})^{-\frac{n-1}{2}+\vartheta_2+\epsilon}\left(1+N(\p_v)^{-\frac{1}{2}-\vartheta_2}\right).
	\end{equation*}
	Consequently, using the convexity bound and standard zero-free region (see \cite[Theorem 3]{brumley2006effective}) of the cuspidal $L$-functions we estimate
	\begin{equation*}
		\frac{L^S\left(\tfrac{n-1}{2}+\tfrac{n}{2}s,\pi\right)L^S\left(\tfrac{1}{2}+\tfrac{n}{2}s,\pi\right)}
		{\sqrt{L^S(1,\pi,\Ad)}}
		\P_S(0;\varphi,f,f)
		\ll_{F,\pi_{0,\infty},\epsilon}N(\q)^{-\frac{n-1}{2}+\vartheta_2+\epsilon} C(\pi_\infty)^{O(1)}\nu_\varphi^{O(1)}.
	\end{equation*}
	Now recalling Remark \ref{rmk:extra-decay-rep-lambda} we see that $I_2^\mathrm{D}(0)$ is estimated by
	\begin{align*}
		& \ll_{F,\pi_0,\epsilon} N(\q)^{-\frac{n-1}2+\vartheta_2+\epsilon}
		\sum_{\substack{\pi\in\Pi_{\mathrm{c}}(\bar{G}_2) \\ \pi^\infty\text{ is unramified}}}C(\pi_\infty)^{O(1)}\sum_{\substack{\varphi\in\B\left((\1_{n-2}\boxplus\pi)\right)\\\varphi^\infty\text{ is normalized spherical}}}
		\left\lvert\left\langle|\phi_0|^2,\Eis(\varphi,0)\right\rangle_{[G^1_n]}\nu_\varphi^{O(1)}\right\rvert                                                                                   \\
		& \ll_{N,\pi_0,\epsilon} N(\q)^{-\frac{n-1}2+\vartheta_2+\epsilon} \sum_{\pi\in\Pi_{\mathrm{c}}(\bar{G}_2)} C(\pi)^{-N}\ll_{\phi_0,\epsilon} N(\q)^{-\frac{n-1}2+\vartheta_2+\epsilon},
	\end{align*}
	where the last estimate follows from Weyl's law for some sufficiently large $N$; see \emph{e.g.}, \cite[eq.(2.16), \S2.6.5]{MV2010subconvexity}.

	We now estimate $I_2^\mathrm{C}(0)$, whose proof follows lines similar to those of $I^{\mathrm D}_2(0)$. For $n>3$, unitary $\chi$ that are unramified at the finite places, $\varphi\in\1_{n-2}\boxplus\chi\boxplus\chi^{-1}$ normalized spherical at finite places, and $\lambda=(0,z,-z)$ with $\Re(z)=0$ we obtain
	\begin{equation*}
		\frac{\prod_{\pm}L^S\left(\tfrac{n-1}{2}\pm z,\chi^\pm\right)L^S\left(\tfrac{1}{2}\pm z,\chi^\pm\right)} {L^S(1+2z,\chi^2)}
		\P_S(\lambda;\varphi,f,f)
		\ll_{F,\pi_0,\epsilon} N(\q)^{-\frac{n-1}{2}+\epsilon}
		\left((1+|z|)C(\chi)\right)^{O(1)}\nu_\varphi^{O(1)},
	\end{equation*}
	where the improvement in the exponent of $N(\q)$ is due to the temperedness of the continuous spectrum.
	Consequently, we estimate $I_2^{\mathrm C}(0)$ by $O_\epsilon\left(N(\q)^{-\frac{n-1}2+\epsilon}\right)$.
	For $n=3$ the same maneuver gives the same estimate for the $\chi\neq 1$ part of $I_2^{\mathrm C}(0)$.

	The case with $n=3$ and $\chi=1$ is slightly more subtle, as the quotient of $L$-functions has a
	simple pole at $z=0$. However, this pole cancels with the zero of $\Eis(\varphi,(0,z,-z))$ at $z=0$ (\emph{cf}.\ the proof of Lemma \ref{lem:evenness-I2C}). Thus using the standard zero-free region of $\xi(z)$ and using the idea of Remark \ref{rmk:stronger-holomorphic-region} we can deform the $z$-contour of the $\chi=1$ summand of $I_2^{\mathrm C}$ to avoid $z=0$ without crossing any pole. On this new contour the quotient of $L$-functions and the Eisenstein series remain holomorphic. Working as before we estimate similarly and conclude the proof.
	\end{proof}

\section{Proof of Theorem \ref{thm:second_moment_asymptotic}}\label{sec:proof_of_thm_C}

In this section we prove Theorem \ref{thm:second_moment_asymptotic}. We recall the choices $\phi_0\in\pi_0$, $\Phi$ and $\Phi^0$ from \S\ref{sec:choices}. Recall $M_{ij}(s)$ from Proposition \ref{prop:main_terms_calcs}. It follows from Proposition \ref{prop:reg-main-term} and Proposition \ref{prop:reg-second-main-term} that the expressions $M_{11}(s)+M_{00}(s)$ and $M_{01}(s)+M_{10}(s)$ are regular at $s=0$. First, we asymptotically evaluate the limits of the above two expressions as $s\to 0$.

\begin{prop}\label{prop:main_term_asymp}
	We have an explicit expression for the main term
	\[M:=\lim_{s\to 0}\,\left(M_{00}(s) + M_{11}(s)\right)=\frac{\zeta_\q(n)}{\zeta_\q(1)\zeta^2_\q(\tfrac{n}{2})}\left(\Delta^{\mu}L(1,\pi_0,\Ad) \frac{n\zeta^*(1)\zeta(\frac{n}{2})^2}{\zeta(n)}\log N(\q)+B_\q\right),\]
	where  $\mu$ is a constant that only depends on $n$ and $B_\q$ is an explicit $\q$-dependent term described in \eqref{eq:def-Bq} and is $O_{F,\pi_0}(1)$.
	\end{prop}

\begin{proof}
	Recall from \S\ref{sec:degenerate_calculations} and Proposition \ref{prop:main_terms_calcs} that $M_{ij}(s)=M_{ij}^0(s)h_{ij,\q}(s)$ for $(i,j)=(0,0)$ and $(1,1)$, where $h_{ij,v}$ satisfy
	\[h_{00,v}(0)=h_{11,v}(0)=\frac{\zeta_v(n)}{\zeta_v(1)\zeta_v(\tfrac n2)^2}.\]
	It follows from the proof of Proposition \ref{prop:reg-main-term} that $M_{00}^0$ and $M_{11}^0$ have simple poles at $s=0$ and
	\begin{equation*}
		R\defeq \Res_{s=0}M_{00}^0(s)=-\Res_{s=0}M_{11}^0(s).
	\end{equation*}
	Moreover, from a standard Rankin--Selberg computation we obtain that for some constant $\mu$, depending only on $n$, we have
	\begin{equation*}
		R=\Delta^{\mu}L(1,\pi_0,\Ad)\zeta^*(1) \frac{\zeta(\frac{n}{2})^2}{\zeta(n)}\int_{N_n(F_\infty)\bs \bar{G}_n(F_\infty)}\lvert W_{0,\infty}f_\infty \rvert^2=\Delta^{\mu}L(1,\pi_0,\Ad) \frac{\zeta^*(1)\zeta(\frac{n}{2})^2}{\zeta(n)},
	\end{equation*}
	where  the final equality follows from \cite[Lemma 5.1]{Jana2020RS} and the choices of normalizations \eqref{eq:arch-normalization-f} and \eqref{eq:arch-normalization-W}.
	Thus following a similar computation as in \cite[\S 5]{Jana2020RS} we obtain
	\begin{align*}
		M=M^0\cdot h_{00,\q}(0)+R\cdot\partial_{s=0}\left(h_{00,\q}(s)-h_{11,\q}(s)\right),
	\end{align*}
	where $M^0:=\lim_{s\to 0}\,\left(M^0_{00}(s) + M^0_{11}(s)\right)$. We compute
	\begin{equation*}
		\partial_{s=0}h_{00,\q}(s)
		=\frac{\zeta_\q(n)}{\zeta_\q(1)\zeta_\q(\tfrac n2)^2}\cdot
		n\sum_{v\mid \q} \zeta_v(\tfrac{n}{2}) \frac{\log N(\p_v)}{N(\p_v)^{\frac{n}{2}}}.
	\end{equation*}
	Noting that $\zeta_v(ns)^{-1}$ has a zero at $s=0$ we compute
	\begin{multline*}
		\partial_{s=0}h_{11,\q}(s)=\frac{\zeta_\q(n)}{\zeta_\q(1)\zeta_\q(\tfrac n2)^2}\left(-n\log N(\q)+n\sum_{v\mid\q}\log N(\p_v)\left(\frac{\zeta_v(n-1)}{N(\p_v)^{n-1}}-\frac{\zeta_v\left(\tfrac n2\right)}{{N(\p_v)^{n/2}}}\right)\right.\\
		\left.-n\sum_{v\mid \q}\frac{\log N(\p_v)}{N(\p_v)^{1+r_v(n-1)}} \frac{\zeta_v(n-1)\zeta^2_v(\frac n2)}
		{\zeta_v^2(\frac{n-2}{2})}\right).
	\end{multline*}
	Combining, we obtain
	\[ M=\frac{\zeta_\q(n)}{\zeta_\q(1)\zeta^2_\q(\tfrac{n}{2})}\left(nR\cdot \log N(\q) + B_\q\right),\]
	where
	\begin{equation}\label{eq:def-Bq}
		B_\q:=M^0+nR\sum_{v\mid\mathfrak q}\log N(\mathfrak p_v)\left(\frac{2\zeta_v(\frac n2)}{N(\mathfrak p_v)^{n/2}}-\frac{\zeta_v(n-1)}{N(\mathfrak p_v)^{n-1}}+\frac{\zeta_v(n-1)\zeta_v(\frac n2)^2}{N(\mathfrak p_v)^{1+r_v(n-1)}\zeta_v(\frac{n-2}{2})^2}\right).
	\end{equation}
	We conclude by noting that trivially $B_\q=O_{F,\pi_0}(1)$.
\end{proof}

\begin{prop}\label{prop:secondary_term_asymp}
	We have an explicit expression for the secondary main term
	\[S:=\lim_{s\to 0}\,(M_{01}(s) + M_{10}(s))=N(\q)^{-\frac{n-2}2}\frac{\zeta_\q(n)}{\zeta_\q^2(1)\zeta_\q(\frac{n}{2})}\left(C\log N(\q)+D_\q\right),\]
	where $C$ is a $\q$-independent constant given by \eqref{eq:def-C}
	and $D_\q$ is an explicit $\q$-dependent term described in \eqref{eq:def-Dq}  and is $O_{F,\pi_0}(1)$.
\end{prop}

\begin{proof}
	We compute this limit in a similar fashion to the proof of Proposition \ref{prop:main_term_asymp}. Recalling Proposition \ref{prop:main_terms_calcs} we compute
	\begin{equation*}
		S=S^0\cdot h_{01,\q}(0)+R'\cdot\partial_{s=0}\left(h_{10,\q}(s)-h_{01,\q}(s)\right),
	\end{equation*}
	where
	\[h_{01,v}(0)=h_{10,v}(0)=N(\p_v^{r_v})^{-\frac{n-2}2} \frac{\zeta_v(n)}{\zeta_v^2(1)\zeta_v(\frac{n}{2})}.\]
	Here \(S^0\defeq \lim_{s\to 0}\,\left(M_{01}^0(s)+M_{10}^0(s)\right)\) and
	$R':=\Res_{s=0}M^0_{10}(s)=-\Res_{s=0}M^0_{01}(s)$ exist by the proof of Proposition \ref{prop:reg-second-main-term}.
	We compute
	\[\partial_{s=0}h_{01,\q}(s)=-n\,N(\q)^{-\frac{n-2}{2}}\frac{\zeta_\q(n)}{\zeta_\q^2(1)\zeta_\q(\tfrac n2)}\log N(\q)\]
	and
	\[\partial_{s=0}h_{10,\q}(s)=N(\q)^{-\frac{n-2}{2}}\frac{\zeta_\q(n)}{\zeta_\q(\tfrac n2)\zeta_\q^2(1)}\cdot n \sum_{v\mid \q} \zeta_v(\tfrac n2)\frac{\log N(\p_v)}{N(\p_v)^{\frac{n}{2}}}.\]
	Combining, we obtain
	\begin{equation*}
		S=N(\q)^{-\frac{n-2}2} \frac{\zeta_\q(n)}{\zeta_\q^2(1)\zeta_\q(\frac n2)}
		\left(C\log N(\q)+ D_\q\right),
	\end{equation*}
	where
	\begin{equation}\label{eq:def-C}
		C:=nR',
	\end{equation}
	and
	\begin{equation}\label{eq:def-Dq}
		D_\q:=S^0+ nR'\sum_{v\mid \q}  \zeta_v(\tfrac n2)\frac{\log N(\p_v)}{N(\p_v)^{\frac{n}{2}}}.
	\end{equation}
	Again, trivially we check $D_\q=O_{F,\pi_0}(1)$.
\end{proof}

We are now ready to prove Theorem \ref{thm:second_moment_asymptotic}.
\begin{proof}[Proof of Theorem \ref{thm:second_moment_asymptotic}]
	We start with the normalized period
	\[
	\vol(K_0(\q))^{-1}\frac{\zeta_\q(\frac{n}{2})^2\zeta_\q(1)}{\zeta_\q(n)}\int_{[\bar{G}_{n}]}|\phi_0|^2|\Eis(f)|^2,
	\]
	where $f$ and $\phi_0$ are as chosen in \S \ref{sec:choices}.
	It follows from Proposition \ref{prop:spec_exp_second_moment} that the period above equals
	\[
	\sum_{P}n^{-1}_P\sum_{\pi\in\Pi^G_{\mathrm{c}}(M)}
	\intop_{i\a^\ast_P/i\a^\ast_G}\frac{|L(\frac{1}{2},
		\mathcal{I}(\pi,\lambda)\otimes\tilde{\pi}_0)|^2}
	{\ell(\mathcal{I}(\pi,\lambda))}
	\mathcal{H}\left(\mathcal{I}(\pi,\lambda)\right)\,d\lambda
	\]
	where
	$\H(\sigma)=H_\q(\sigma) h_\infty(\sigma)$ satisfies all the properties
	required in Theorem \ref{thm:second_moment_asymptotic}.

	On the other hand, we apply Theorem \ref{thm:spectral-expansion} to arrive at
	\begin{equation*}
		\int_{[\bar{G}_{n}]}|\phi_0|^2|\Eis(f)|^2 =
		\lim_{s\to 0}\,\left(M_{00}(s)+M_{11}(s)\right)+\lim_{s\to 0}\,\left(M_{10}(s)+M_{01}(s)\right)+I_2^{\mathrm D}(0)+I_2^{\mathrm C}(0).
	\end{equation*}
	We conclude the proof by invoking Propositions \ref{prop:main_term_asymp} and \ref{prop:secondary_term_asymp} to evaluate the first two limits, and Proposition \ref{prop:error_term_bound} to estimate the last two terms on the right hand side above.
	\end{proof}

\bibliography{references.bib}
\bibliographystyle{alpha}
\end{document}